\documentclass[11pt]{article}
\pdfoutput=1

\def\red#1{}

\usepackage{lineno,lmodern}      

\usepackage[dvipsnames]{xcolor}

\def\beq{\begin{equation} }\def\eeq{\end{equation} }\def\1{\mathbf{1}}

\usepackage[framemethod=default]{mdframed}
\usepackage{caption}
\usepackage{indentfirst}
\usepackage{bm, mathrsfs, graphics,float,amssymb,amsmath,subeqnarray,setspace,graphicx,amsthm,epstopdf,subfigure, enumerate, color}
\usepackage[utf8]{inputenc}
\usepackage[colorlinks,
linkcolor=red,
anchorcolor=blue,
citecolor=blue
]{hyperref}
\usepackage{natbib}

\usepackage{fullpage}
\numberwithin{equation}{section}

\newtheorem{lemma}{Lemma}
\newtheorem{theorem}{Theorem}
\newtheorem{proposition}{Proposition}
\newtheorem{definition}{Definition}
\newtheorem{corollary}[theorem]{Corollary}
\newtheorem{remark}{Remark}

\ifx\assumption\undefined
\newtheorem{assumption}{Assumption}
\fi

\newcommand{\EE}{\mathbb{E}}
\newcommand{\RR}{\mathbb{R}}

\newcommand{\cF}{{\bm{\mathcal{F}}}}

\newcommand{\PP}{\mathbb{P}}

\newcommand{\TG}{\tilde{G}}
\newcommand{\BG}{\bar{G}}

\newcommand{\TB}{\tilde{B}}
\newcommand{\HSig}{\hat{\Sigma}}
\newcommand{\Hu}{\hat{\mu}}

\newcommand{\tg}{\tilde{g}}
\newcommand{\tpsi}{\tilde{\psi}}
\newcommand{\argmin}{\mathop{\mathrm{argmin}}}

\usepackage{bbm}

\def\cS{\mathcal{S}}

\usepackage{multirow}
\usepackage{tablefootnote}
\usepackage{colortbl}
\usepackage{hhline}

\usepackage{algorithm}
\usepackage{algorithmic}

\begin{document}
\title{
	Mirror Natural Evolution Strategies 
}

\author{
	Haishan Ye \\
	hsye\_cs@outlook.com\\
	Shenzhen Research Institute of Big Data\\ 
	The Chinese University of Hong Kong, Shenzhen
	\and
	Tong Zhang\\
	tongzhang@ust.hk\\
	Computer Science and Mathematics \\
	Hong Kong University of Science and Technology
}
\date{}

\maketitle
\begin{abstract}
	Evolution Strategies such as \texttt{CMA-ES} (covariance matrix adaptation evolution strategy) and \texttt{NES} (natural evolution strategy) have been widely 
	used in machine learning applications, where an objective function is optimized without using its derivatives.   
	However, the convergence behaviors of these algorithms have not been carefully studied. In particular, there is no rigorous analysis for the convergence of the estimated covariance matrix, and it is unclear how does the estimated covariance matrix help the converge of the algorithm. 
	The relationship between Evolution Strategies and derivative free optimization algorithms is also not clear.
	In this paper, we propose a new algorithm closely related to \texttt{NES}, which we call \texttt{MiNES} (mirror descent natural evolution strategy), for which we can establish rigorous convergence results.
	We show that the estimated covariance matrix of \texttt{MiNES} converges to the inverse of Hessian  matrix of the objective function with a sublinear convergence rate.
	Moreover, we show that some derivative free optimization algorithms are special cases of \texttt{MiNES}.
	Our empirical studies demonstrate that \texttt{MiNES} is a query-efficient optimization algorithm competitive to classical algorithms including \texttt{NES} and \texttt{CMA-ES}.
\end{abstract}


\def\TH{\tilde{H}}
\newcommand{\ti}[1]{\tilde{#1}}
\def\diag{\mathrm{diag}}
\newcommand{\norm}[1]{\left\|#1\right\|}
\newcommand{\dotprod}[1]{\left\langle #1\right\rangle}
\def\tr{\mathrm{tr}}

\section{Introduction}

Evolutionary strategies (\texttt{ES}) are an important class of zeroth-order algorithms for optimization problems that only have access to function value evaluations.
ES attracts much attention since it was introduced by Ingo Rechenberg and Hans-Paul Schwefel in the 1960s and 1970s \citep{schwefel1977numerische}, and many variants have been proposed \citep{beyer2001self,hansen2001completely,wierstra2008natural,glasmachers2010exponential}. 
\texttt{ES} tries to evaluate the fitness of real-valued genotypes in batches, after which only the best genotypes are kept and used to produce the next batch of offsprings. 
A covariance matrix is incorporated into evolutionary strategies to capture the dependency variables so that independent `mutations' can be generated for the next generation. 
In this general algorithmic framework, the most well-known algorithms are the covariance matrix adaptation evolution strategy (\texttt{CMA-ES}) \citep{hansen2001completely} and natural evolution strategies (\texttt{NES}) \citep{wierstra2008natural}. 

Evolutionary strategies have been widely used in machine learning applications. For example, \texttt{NES} has been used in  deep reinforcement learning \citep{salimans2017evolution,conti2018improving}, and has advantages in high parallelizability. \texttt{NES} and \texttt{CMA-ES} are widely used in black-box adversarial attack of deep neural networks \citep{ilyas18a,chen18,dong2019efficient,chen2019boundary}. 
Evolutionary strategies can also be used for hyper-parameter tuning.  
\texttt{CMA-ES} has been used to tune the hyper-parameters of deep neural networks \citep{loshchilov2016cma}, and it is shown that this approach can produce better hyper-parameters than Bayesian optimization.

The goal of this paper is to derive rigorous convergence analysis for a variant of the NES family of algorithms, and establish the relations between NES and existing zeroth-order algorithms. 
This work solves the open questions left in the convergence analysis of previous works by proposing a different algorithm called Mirror Natural Evolution Strategies or \texttt{MiNES}, for which we can establish rigorous convergence results. 
For quadratic functions that are strongly convex, we show that \texttt{MiNES} enjoys an $O(1/k)$-convergence rate in terms of the optimality gap, 
the estimated covariance matrix converges to the inverse of the Hessian matrix.

\paragraph{More Related Literature.}
There have been many efforts to better understand these methods in the literature \citep{Beyer14,OllivierAAH17,auger2016linear,malago2015information,AkimotoNOK10}. 
For example, the authors in \citet{AkimotoNOK10} revealed a connection between \texttt{NES} and \texttt{CMA-ES}, and showed that  \texttt{CMA-ES}  is a special version of \texttt{NES}. 
\texttt{NES} can be derived from information geometry because it employs a natural gradient descent \citep{wierstra2014natural}. 
Therefore,  the convergence of \texttt{NES} and \texttt{CMA-ES} has been studied from the information geometry point of view  \citep{Beyer14,auger2016linear}. 
\citet{Beyer14} analyzed the convergence of \texttt{NES} with infinitestimal learning rate using ordinary differential equation with the objective function being quadratic and strongly convex. However,
it does not directly lead to a convergence rate result with finite learning rate. 
Moreover, \citet{Beyer14} does not show how the covariance matrix converges to the inverse of the Hessian, although this is conjectured in \citep{hansen2016cma}. 

Recently, \citet{auger2016linear} showed that for strongly convex functions, the estimated mean value in ES with comparison-based step-size adaptive randomized search (including \texttt{NES} and \texttt{CMA-ES}) can achieve linear convergence. However, \citet{auger2016linear} have not shown how covariance matrix converges and how covariance matrix affects the convergence properties of the estimated mean vector. Although it has been conjectured that for the quadratic function, the covariance matrix of \texttt{CMA-ES} will converge to the inverse of the Hessian  up to a constant factor, no rigorous proof has been provided \citep{hansen2016cma}. 

This work solves the open questions left in the convergence analysis of previous works by proposing a different algorithm called  \texttt{MiNES} (mirror descent natural evolution strategy), for which we can establish rigorous convergence results. For the quadratic and strongly convex case, we show that the objective value converges to the optimal value at a rate of $O(1/k)$, where $k$ is the iteration number, and the estimated covariance matrix converges to $c \cdot H^{-1}$ for a constant scalar $c$, where $H$ is the Hessian matrix. Moreover, convergence can also be obtained for non quadratic functions.

Another important line of research in zeroth-order optimization is the derivative free algorithms from the optimization literature. 
The idea behind these algorithms is to create a stochastic oracle to approximate (first-order) gradients using (zeroth-order) function value difference at a random direction, and then apply the update rule of (sub-)gradient descent \citep{Nesterov2017,ghadimi2013stochastic,duchi2015optimal}. 
On the other hand, \citet{conn2009global} propose to  utilize curvature information in constructing quadratic approximation model under a slightly
modified trust region regime.
It seems that the derivative free algorithms and evolutionary strategies are totally different algorithms since they are motivated from different ideas.
However, they are closely related.
To improve the convergence rate of \texttt{NES}, \citet{salimans2017evolution} proposed ‘antithetic sampling’ technique. In this case, \texttt{NES} shares the same algorithmic form with derivative free algorithm \citep{Nesterov2017}. 
\texttt{NES} with ‘antithetic sampling’ is widely used in black-box adversarial attack \citep{chen18,ilyas18a}. 
Nevertheless, the mathematical relationship between \texttt{NES} and derivative free algorithms have not been explored.

\paragraph{Contribution.}
We summarize our contribution as follows:
\begin{enumerate}
	\item We propose a regularized objective function and show that the covariance part of its minimizer is close to the Hessian inverse. Based on this new objective function, we propose a novel \texttt{NES}-style algorithm called \texttt{MiNES}, which guarantees that the covariance matrix converges to the inverse of the Hessian when the function is quadratic. We provide a convergence analysis of \texttt{MiNES}, leading to the first rigorous convergence analysis of covariance matrix of \texttt{ES}-type algorithms. 
	\item The algorithmic procedure of \texttt{MiNES} shares the same  algorithmic form with derivative free algorithm. This connection shows that derivative free algorithm can be derived from natural evolution strategies.
	\item We empirically study the convergence of \texttt{MiNES}, and show that it is competitive to state of the art \texttt{ES} algorithms. \texttt{MiNES} converges faster than derivative free algorithm because \texttt{MiNES} exploits the Hessian information of the underlying objective function, while derivative free algorithms only use function values to approximate the first order gradient.
\end{enumerate}

\noindent\textbf{Organization.}
The rest of this paper is organized as follows.
In Section~\ref{sec:background}, we introduce the background and preliminaries will be used in this paper.
In Section~\ref{sec:ROF}, we propose a novel regularized objective function and prove that the mean part of its minimizer is close to the minimizer of the original problem and the covariance part is close to the corresponding Hessian inverse.
Section~\ref{sec:MiNES} gives the detailed description of the proposed mirror natural evolution strategies.
In Section~\ref{sec:conv}, we analyze the convergence properties of \texttt{MiNES}. We provide the first rigorous analysis on the convergence rate of covariance matrix of \texttt{NES}-type algorithms.
In Section~\ref{sec:experiments}, we empirically evaluate the performance of \texttt{MiNES} and compare it with classical algorithms.
Finally, we conclude our work in Section~\ref{sec:conclusion}.
The detailed proofs are deferred to the appendix in appropriate orders.

\section{Background and Preliminaries}
\label{sec:background}
In this section, we will introduce the natural evolutionary strategies and preliminaries.

\subsection{Natural Evolutionary Strategies}

The Natural Evolutionary Strategies (\texttt{NES}) reparameterize the objective function $f(z)$ ($z\in\RR^d$) as follows:
\begin{equation}\label{eq:J_org}
J(\theta) = \EE_{z \sim \pi(\cdot|\theta)} [f(z)] = \int f(z)\pi(z|\theta)\; dz,
\end{equation}
where $\theta$ denotes the parameters of density $\pi(z|\theta)$ and $f(z)$  is commonly referred as the fitness function for samples $z$.
Such transformation can help to develop algorithms to find the minimum of $f(z)$ by only accessing to the function value. 

\paragraph{Gaussian Distribution and Search Directions.}

In this paper, we will  only investigate the Gaussian distribution, that is, 
\begin{equation}
\label{eq:pi_z}
\pi(z|\theta) \sim N(\mu,\bar{\Sigma}).
\end{equation}
Accordingly, we have
\begin{equation}\label{eq:z}
z = \mu + \bar{\Sigma}^{1/2} u, \quad u\sim N(0, I_d),
\end{equation}
where $d$ is the dimension of $z$.
Furthermore, the density function $\pi(z|\theta)$ can be presented as 
\begin{equation*}
\pi(z|\theta) = \frac{1}{\sqrt{(2\pi)^d\det(\bar{\Sigma})}}\cdot \exp\left(-\frac{1}{2}(z - \mu)^\top \bar{\Sigma}^{-1}(z-\mu)\right)
\end{equation*}
In order to compute the derivatives of $J(\theta)$, we can use the so-called `log-likelihood trick' to obtain the following \citep{wierstra2014natural}
\begin{align}
\nabla_\theta J(\theta) = \nabla_\theta \int f(z) \pi(z|\theta)\; dz
=\EE_z \big[ f(z) \nabla_\theta \log \pi(z|\theta) \big]. \label{eq:log_lh}
\end{align}
We also have that 
\begin{align*}
\log \pi(z|\theta) = -\frac{d}{2}\log(2\pi) -\frac{1}{2} \log \det \bar{\Sigma} - \frac{1}{2}(z-\mu)^\top \bar{\Sigma}^{-1}(z-\mu).
\end{align*} 
We will need its derivatives with respect to $\mu$ and $\bar\Sigma$, that is, $\nabla_\mu\log\pi(z|\theta)$ and $\nabla_{\bar{\Sigma}}\log\pi(z|\theta)$. The first is trivially
\begin{align}
\nabla_\mu\log\pi(z|\theta) = \bar{\Sigma}^{-1}(z-\mu) , \label{eq:nab_mu}
\end{align}
while the latter is 
\begin{equation}\label{eq:nab_Sig}
\nabla_{\bar{\Sigma}}\log\pi(z|\theta) = \frac{1}{2}\bar{\Sigma}^{-1}(z-\mu)(z-\mu)^\top\bar{\Sigma}^{-1} - \frac{1}{2} \bar{\Sigma}^{-1}.
\end{equation}
Let us denote 
\begin{equation*}
\theta = [\mu^\top,\;\mathrm{vec}({\bar{\Sigma}})^\top]^\top,
\end{equation*}
where $\theta\in\RR^{d(d+1)}$-dimensional column vector consisting of all the elements of the mean vector $\mu$ and the covariance matrix $\bar{\Sigma}$. $\mathrm{vec}(\cdot)$ denotes a rearrangement operator from  a matrix to a column vector.
The Fisher matrix with respect to $\theta$ of $\pi(z|\theta)$  for a Gaussian distribution is well-known \citep{AkimotoNOK10}, 
\begin{align}
F_\theta 
= 
\EE_z \left[ \nabla_\theta \log \pi(z|\theta)  \nabla_\theta \log \pi(z|\theta)^\top\right] 
=
\left[
\begin{array}{cc}
\bar{\Sigma}^{-1}, & 0 \\
0, &\frac{1}{2}\bar{\Sigma}^{-1}\otimes\bar{\Sigma}^{-1}
\end{array}
\right] \notag
\end{align}
where $\otimes$ is the Kronecker product. 
Therefore, the natural gradient of the log-likelihood of $\pi(z|\theta)$ is 
\begin{equation}
\label{eq:nat_grad}
F^{-1}_\theta \nabla_\theta \log \pi(z|\theta) 
=
\left[
\begin{array}{c}
z - \mu\\
\mathrm{vec}\left((z-\mu)(z-\mu)^\top - \bar{\Sigma}\right)
\end{array}
\right]
\end{equation} 
Combining Eqn.~\eqref{eq:log_lh} and \eqref{eq:nat_grad}, we obtain the estimate of the natural gradient from samples $z_1, \dots, z_b$ as
\begin{equation}
\label{eq:ng}
F^{-1}_\theta \nabla J(\theta) 
\approx \frac{1}{b}\sum_{i=1}^{b} f(z_i) 
\left[
\begin{array}{c}
z_i - \mu\\
\mathrm{vec}\left((z_i-\mu)(z_i-\mu)^\top - \bar{\Sigma}\right)
\end{array}
\right]
\end{equation} 
Therefore, we can obtain the meta-algorithm of \texttt{NES} (Algorithm 3 of \citet{wierstra2014natural} with $F^{-1}_\theta \nabla J(\theta) $ approximated as Eqn.~\eqref{eq:ng})
\begin{equation}
\left\{
\begin{aligned}
\mu =&  \mu - \eta \cdot\frac{1}{b}\sum_{i=1}^{b} f(x+\bar{\Sigma}^{1/2}u_i) \bar{\Sigma}^{1/2}u_i\\
\bar{\Sigma} =& \bar{\Sigma} - \eta\cdot\frac{1}{b}\sum_{i=1}^{b} f(x+\bar{\Sigma}^{1/2}u_i) \left(\bar{\Sigma}^{1/2}u_i u_i^\top \bar{\Sigma}^{1/2} - \bar{\Sigma}\right),
\end{aligned}
\right.
\end{equation}
where $\eta$ is the step size.
\subsection{Notions}

Now, we introduce some important notions which is widely used in optimization.
\paragraph{$L$-smooth} 
A function $f(\mu)$ is \textit{$L$-smooth}, if it holds that,  for all $\mu_1,\mu_2\in\RR^d$
\begin{equation}\label{eq:L_1}
\norm{\nabla f(\mu_1) - \nabla f(\mu_2)} \leq L\norm{\mu_1 - \mu_2}
.
\end{equation}

\paragraph{$\sigma$-Strong Convexity} 
A function  $f(\mu)$ is \textit{$\sigma$-strongly convex}, if it holds that, for all $\mu_1,\mu_2\in\RR^d$
\begin{equation}\label{eq:sig_1}
f(\mu_1) - f(\mu_2) \geq \dotprod{\nabla f(\mu_2), \mu_1 - \mu_2} + \frac{\sigma}{2}\norm{\mu_1 - \mu_2}^2.
\end{equation}

\paragraph{$\gamma$-Lipschitz Hessian}
A function $f(\mu)$ \textit{admits $\gamma$-Lipschitz Hessians} if it holds that,  for all $\mu_1,\mu_2 \in \RR^d$, it holds that 
\begin{align}
\norm{\nabla^2 f(\mu_1) - \nabla^2f(\mu_2)} \leq \gamma \norm{\mu_1-\mu_2}
.
\label{eq:gamma_1}
\end{align}

Note that $L$-smoothness and $\sigma$-strongly convexity imply $\sigma I\preceq \nabla^2 f(\mu) \preceq L I$.

\section{Regularized Objective Function}
\label{sec:ROF}

Conventional \texttt{NES} algorithms are going to minimize $J(\theta)$ (\eqref{eq:J_org}). 
Instead, we propose an novel regularized objective function to reparameterize $f(z)$:
\begin{equation}\label{eq:Q}
Q_\alpha(\theta) = J(\theta) - \frac{\alpha^2}{2} \log\det\Sigma,
\end{equation}
where $\alpha$ is a positive constant. 
Furthermore, we represent $\bar{\Sigma}$ in Eqn.~\eqref{eq:pi_z} as $\bar{\Sigma} = \alpha^2\Sigma$. 
Accordingly, $\theta(\mu, \Sigma)$ denotes the parameters of a Gaussian density $\pi(z|\theta) = N(\mu,\alpha^2 \Sigma)$.
By such transformation, $J(\theta)$ can be represent as 
\begin{equation*}
J(\theta) = \EE_u [f(\mu + \alpha \Sigma^{1/2}u)], \quad\mbox{with}\quad u\sim N(0,I_d).
\end{equation*}
Then $J(\theta)$ is the \emph{Gaussian approximation function} of $f(z)$ and  $\alpha$  plays a role of smoothing parameter \citep{Nesterov2017}.
Compared with $J(\theta)$, $Q_\alpha(\theta)$ has several advantages and we will first introduce the intuition we propose   $Q_\alpha(\theta)$.

\paragraph{Intuition Behind $Q_\alpha(\theta)$}

Introducing the regularization brings an important benefit which can help to clarify the minimizer of $\Sigma$. 
This benefit can be shown when $f(z)$ is a quadratic function where $f(z)$ can be expressed as  
\begin{equation}\label{eq:quad}
f(z) = f(\mu)+\dotprod{\nabla f(\mu), z-\mu}+ \frac{1}{2}(z-\mu)^\top  H (z-\mu) ,
\end{equation}
where $H=\nabla^2 f(\mu)$ denotes the Hessian matrix. 
Note that when $f(z)$ is quadratic, the Hessian matrix is independent on different $z$. 
In the rest of this paper, we will use $H$ to denote the Hessian matrix of a quadratic function.
Since we have $z = \mu+\alpha\Sigma^{1/2}u$ (by Eqn.~\eqref{eq:z}), $J(\theta)$ can be explicitly expressed as
\begin{align}
J(\theta) =& \EE_u\left[f(\mu)+\alpha\dotprod{\nabla f(\mu), \Sigma u}+\frac{\alpha^2}{2}u^\top \Sigma^{1/2} H \Sigma^{1/2}u\right]\notag\\
=&f(\mu) + \frac{\alpha^2}{2}\dotprod{H, \Sigma}. \label{eq:J_quad}
\end{align}
where $\dotprod{A,B} = \tr(A^\top B)$.
By setting $\nabla_{\theta} Q_\alpha(\theta) = 0$, we can obtain that
\begin{equation}
\frac{\partial Q_\alpha}{\partial \mu} = \nabla_\mu f(\mu) = 0, \quad \frac{\partial Q_\alpha}{\partial \Sigma} = \frac{\alpha^2}{2}H - \frac{\alpha^2}{2} \Sigma^{-1} = 0.
\end{equation}
Thus, we can obtain that the minimizer $\mu$ of  $Q_\alpha$ is $\mu_*$ --- the minimizer of $f(\mu)$ and the minimizer $\Sigma$ of  $Q_\alpha$ is $H^{-1}$ --- the inverse of the  Hessian matrix.
In contrast, without the regularization, $\frac{\partial Q_\alpha}{\partial \Sigma} $ will reduce to $\frac{\partial J(\theta)}{\partial \Sigma}$:
\begin{equation*}
\frac{\partial J(\theta)}{\partial \Sigma} =  \frac{\alpha^2}{2}H.
\end{equation*} 
Thus, the  $\partial  J(\theta)/\partial \Sigma$ does not provide useful information about what covariance matrix is the optimum of $J(\theta)$. 

Therefore, in this paper, we will consider the regularized objective function $Q_\alpha(\theta)$. 
To obtain a concise theoretical analysis of the convergence rate of $\Sigma$, we are going to solve the following constrained optimization problem
\begin{equation}\label{eq:modi_prob}
\min_{\mu\in\RR^d, \Sigma\in \mathcal{S}} Q_\alpha(\theta(\mu,\Sigma))
\end{equation}
with $\cS$ is defined as
\begin{equation}\label{eq:cS}
\cS = \bigg\{\Sigma\bigg|\zeta^{-1}\cdot I\preceq\Sigma\preceq \tau^{-1}\cdot I \bigg\},
\end{equation}
where $\zeta$ and $\tau$ are positive constants which satisfy $\tau\leq \zeta$. 
Note that, the constraint on $\Sigma$ is used to keep $\Sigma$ bounded and this property will be used in the convergence analysis of $\Sigma$.

\begin{remark}
	Note that \texttt{NES} and \texttt{CMA-ES} are well-known algorithms to minimize $J(\theta)$ by natural gradient descent \citep{wierstra2014natural,AkimotoNOK10,hansen2016cma}. 
	Our algorithm introduces an additional regularizer, differing the underlying objective function  from the previous \texttt{NES}-type algorithms.
	In our formulation, $\Sigma$ converges approximately to the Hessian inverse matrix, which can not be guaranteed by minimizing $J(\theta)$.
\end{remark}

In the rest of this section, we show that the mean vector ($\Hu_*$) of minimizer of Eqn.~\eqref{eq:modi_prob} is close to $\mu_*$ which is the minimizer of $f(\mu)$. 
Furthermore, if $\zeta$ and $\tau$ in Eqn.~\eqref{eq:cS} are chosen properly, the covariance matrix $\HSig_*$ will be close to the Hessian inverse. 
As a special case, if $f(\cdot)$ is quadratic, then we show that $\Hu_*$ is equal to $\mu_*$ and $\HSig_*$ is equal to $[\nabla^2f(\mu_*)]^{-1}$.

\subsection{Quadratic Case}

We will first investigate the case that function $f(\cdot)$ is quadratic, because the solution in this case is considerably simple. 
Existing works conjecture that the covariance of \texttt{CMA-ES} will converge to the Hessian inverse matrix, up to a constant factor \citep{hansen2016cma}. Experiments seem to support this conjecture, but there is not any rigorous theoretical proof \citep{hansen2016cma}. 
In contrast,  the minimizer of $Q_\alpha(\theta)$ satisfies the following proposition.
\begin{theorem}\label{prop:quad_case}
	If the function $f(\cdot)$ is quadratic so that $f(z)$ satisfies Eqn.~\eqref{eq:quad}. Assume $f(\cdot)$ also satisfies Assumption 1 and 2.
	Let $\mu_*$ be the minimizer of $f(\cdot)$, then the minimizer of $Q_\alpha(\theta)$ is 
	\begin{equation*}
	\left(\mu_*, \Pi_{\cS}\left(H^{-1}\right)\right) = \argmin_{\mu, \Sigma\in\cS} Q_\alpha(\mu, \Sigma),
	\end{equation*}
	where $\Pi_{\cS}(A)$ is the projection of symmetric $A$ on to $\cS$, that is, $\Pi_{\cS}(A) = \argmin_{X\in\cS}\norm{A-X}$ with $\norm{\cdot}$ being the Frobenius norm.
\end{theorem}

If $\zeta\geq L$ and $\tau\leq \sigma$ in $\cS$ (defined in Eqn.~\eqref{eq:cS}), we can observe that the above proposition shows that the optimal covariance matrix $\Sigma$ is the Hessian inverse matrix. 
\subsection{General Strongly Convex Function with Smooth Hessian} 

Next we will consider the general convex function case with its Hessian being $\gamma$-Lipschitz continuous.
First, we can rewrite $Q_\alpha(\theta)$ as 
\begin{equation*}
Q_\alpha(\theta) = J(\theta) + R(\Sigma),
\end{equation*}
where we denote $$R(\Sigma) = - \frac{\alpha^2}{2} \log\det\Sigma$$
which  is the regularizer on $\Sigma$.
We will show that the value $Q_\alpha(\theta)$ at $\theta = (\mu, \Sigma)$ is close to $f(\mu)$ when $\alpha$ is sufficiently small.
\begin{theorem}\label{prop:close}
	Let $Q_\alpha(\theta)$ be defined in \eqref{eq:Q}. Assume that function $f(\cdot)$ is $L$-smooth and its Hessian is $\gamma$-Lipschitz,  then $Q_\alpha(\theta)$ satisfies that
	\begin{equation}
	f(\mu) - \frac{L\alpha^2}{2} \tr(\Sigma) - \alpha^3\phi(\Sigma) + R(\Sigma) \leq Q_\alpha(\theta) 
	\leq f(\mu) + \frac{L\alpha^2}{2} \tr(\Sigma) + \alpha^3\phi(\Sigma) + R(\Sigma) ,
	\end{equation}
	where $\phi(\Sigma)$ is defined as
	\begin{equation*}
	\phi(\Sigma) = \frac{1}{(2\pi)^{d/2}}\int_u \frac{\gamma}{6} \|\Sigma^{1/2} u\|^3 \exp\left(-\frac{1}{2}\norm{u}^2\right) d u.
	\end{equation*}
\end{theorem}

By the above proposition, we can observe that as $\alpha\to 0$, $\frac{L\alpha^2}{2} \tr(\Sigma)$, $\phi(\Sigma)$ and $R(\Sigma)$ will go to $0$. Therefore,  instead to directly minimizing $f(\mu)$, we can minimize $Q_\alpha(\theta)$ with $\theta = (\mu, \Sigma)$. 
Next, we will prove that the $\mu$ part of the minimizer of $Q_\alpha(\theta)$ is also close to the solver of $\min f(\mu)$.
\begin{theorem}\label{prop:mu}
	Let $f(\cdot)$ satisfy the properties in Theorem~\ref{prop:close}. $f(\cdot)$ is also $\sigma$-strongly convex. 
	Let $\hat{\mu}_*$ be the minimizer of $Q_\alpha(\theta)$ under constraint $\Sigma\in\cS$.
	$\mu_*$ denotes the optimum of $f(\mu)$. 
	Then, we have the following properties
	\begin{align*}
	f(\hat{\mu}_*) -f(\mu_*) 
	\leq& 
	\frac{dL\alpha^2}{\tau} 
	+ 
	\frac{\gamma \alpha^3 (d^2+2d)^{3/4}}{3\tau^{3/2}} 
	+
	\frac{d\alpha^2}{2} \left(\tau^{-1} - \zeta^{-1}  \right)\\
	\norm{\mu_* - \hat{\mu}_*}^2 
	\leq&
	\frac{2dL\alpha^2}{\sigma\tau} 
	+ 
	\frac{2\gamma \alpha^3 (d^2+2d)^{3/4}}{3\sigma\tau^{3/2}} 
	+
	\frac{d\alpha^2}{\sigma} \left(\tau^{-1} - \zeta^{-1}  \right),
	\end{align*}
	where $d$ is the dimension of $\mu$.
\end{theorem}

Finally, we will provide the properties how well $\Sigma$ approximates the inverse of Hessian  matrix.
\begin{theorem}\label{prop:Sigma}
	Let $f(\cdot)$ satisfy the properties in Theorem~\ref{prop:mu}. $\hat{\theta} = (\Hu_*, \HSig_*)$ is the minimizer of $Q_\alpha(\theta)$ under the constraint $\Sigma\in\cS$. $\mu_*$ is the minimizer of $f(\mu)$ and $\Sigma_*$ is the minimizer of $Q_\alpha(\theta)$ given $\mu = \mu_*$ under the constraint $\Sigma\in\cS$. 
	Assume that $ \alpha$ satisfies 
	\begin{equation*}
	\alpha
	\leq
	\frac{3\tau^{3/2}\sigma}{\gamma\zeta}
	\cdot
	\left(
	\left(
	d^3
	+6d^2
	+8d
	\right)^{5/6}
	+
	d\left(d^2+2d\right)^{3/4}
	\right)^{-1},
	\end{equation*}
	where $d$ is the dimension of $\mu$,
	then $\Sigma_*$ has the following properties
	\begin{equation*}
	\norm
	{
		\HSig_* - \Pi_{\cS}\left(\left(\nabla^2f(\Hu_*) 	
		\right)^{-1}\right)
	}
	\leq 
	\frac{\alpha\gamma\zeta}{3\tau^{3/2}\sigma^2}
	\cdot
	\left(
	\left(
	d^3
	+6d^2
	+8d
	\right)^{5/6}
	+
	d\left(d^2+2d\right)^{3/4}
	\right),
	\end{equation*}
	and
	\begin{align*}
	\norm{\Sigma_* - \hat{\Sigma}_*}
	\leq&
	\frac{2\alpha\gamma\zeta}{3\tau^{3/2}\sigma^2}
	\cdot
	\left(
	\left(
	d^3
	+6d^2
	+8d
	\right)^{5/6}
	+
	d\left(d^2+2d\right)^{3/4}
	\right)
	\\
	&+
	\frac{\gamma}{\sigma^2}
	\cdot
	\left(
	\frac{2dL\alpha^2}{\sigma\tau} 
	+ 
	\frac{2\gamma \alpha^3 (d^2+2d)^{3/4}}{3\sigma\tau^{3/2}} 
	+
	\frac{d\alpha^2}{\sigma} \left(\tau^{-1} - \zeta^{-1}  \right)
	\right)^{1/2},
	\end{align*}
	where $\Pi_{\cS}(\cdot)$ is the projection operator which projects a symmetric matrix on to $\cS$ with Frobenius norm as distance measure.
\end{theorem}

\begin{remark}
	Theorem~\ref{prop:Sigma} shows that when $\alpha$ is small, then $\HSig_*$  (the $\Sigma$ minimizer of problem~\eqref{eq:modi_prob}) is close to $\Pi_{\cS}\left(\left(\nabla^2f(\mu_*)\right)^{-1}\right)$. If $\tau \leq \sigma$ and $\zeta\geq L$, then $\HSig_*$ is close to the inverse of the Hessian at $\mu_*$.  
\end{remark}

The above propositions show that if $\alpha$ is small, then $\left(\Hu_*, \HSig_*\right)$, which is the minimizer of problem~\ref{eq:modi_prob}, will be close to $\left(\mu_*, \Pi_{\cS}\left(\left(\nabla^2f(\mu_*)\right)^{-1}\right)\right)$. 
In the next section, we propose a mirror natural evolution strategy to minimize problem~\ref{eq:modi_prob}.

\section{Mirror Natural Evolution Strategies}
\label{sec:MiNES}

In the previous sections, we have shown that one can obtain the minimizer of $f(\mu)$ by minimizing its reparameterized function $Q_\alpha(\theta)$ with $\theta = (\mu, \Sigma)$. 
Instead of solving the optimization problem \eqref{eq:modi_prob} by the natural gradient descent, we propose a novel method called MIrror Natural Evolution Strategy (\texttt{MiNES}) to minimize $Q_\alpha(\theta)$. 
\texttt{MiNES} consists of two main update procedures. 
It updates $\mu$ by natural gradient descent but with `antithetic sampling' (refers to Eqn.~\eqref{eq:g_mu}).
Moreover, \texttt{MiNES}  updates $\Sigma$ by the mirror descent method. 
The mirror descent of $\Sigma$ can be derived naturally because $\nabla_\Sigma R(\Sigma) = -\frac{\alpha^2}{2}\Sigma^{-1}$ is a mirror map  widely used in convex optimization \citep{kulis2009low}.    

In the rest of this section, we will first describe our algorithmic procedure in detail. Then we will discuss the connection between \texttt{MiNES} and existing works.

\subsection{Algorithm Description}

We will give the update rules of $\mu$ and $\Sigma$ respectively.

\paragraph{Natural Gradient Descent of $\mu$}

The natural gradient of $Q_\alpha(\theta)$ with respect to $\mu$ is defined as 
\begin{equation}
\label{eq:ng_mines}
g(\mu) = F_\mu^{-1}\frac{\partial Q_\alpha}{\partial \mu}.
\end{equation}
where $F_\mu$ is the Fisher information matrix with respect to $\mu$. 
First, by the properties of the Gaussian distribution, we have the following property.
\begin{lemma}
	Let $F_\mu$ be the Fisher information matrix with respect to $\mu$ and the natural gradient $g(\mu)$ be defined as Eqn.~\eqref{eq:ng_mines}. Then $g(\mu)$ satisfies 
	\begin{equation}
	g(\mu) = 
	\alpha^2\cdot \left(\frac{1}{2\alpha}\cdot\EE_u\left[(f(\mu+\alpha \Sigma^{1/2}u) - f(\mu-\alpha \Sigma^{1/2}u)) \Sigma^{-1/2}u\right]\right).
	\end{equation}
\end{lemma}
\begin{proof}
	The Fisher matrix $F_\mu$ can be computed as follows \citep{wierstra2014natural}:
	\begin{align}
	F_\mu =& \EE_z\left[\nabla_\mu \log \pi(z|\theta)\nabla_\mu \log \pi(z|\theta)^\top\right] \notag\\
	\overset{\eqref{eq:nab_mu}}{=}&\EE_u\left[\alpha^{-2}\Sigma^{-1/2}uu^\top \Sigma^{-1/2}\right] \notag\\
	=&\alpha^{-2}\Sigma^{-1}. \label{eq:F_mu}
	\end{align}
	Note that $\frac{\partial Q_\alpha(\theta)}{\partial \mu} = \frac{\partial J(\theta)}{\partial \mu}$, by Eqn.~\eqref{eq:log_lh}, so we have
	\begin{equation}
	\frac{\partial Q_\alpha(\theta)}{\partial \mu} =\frac{\partial J(\theta)}{\partial \mu} = \EE[f(z)\nabla_\mu\log\pi(z|\theta)].
	\end{equation}
	We also have $z = \mu + \alpha \Sigma^{1/2}u$ with $u\sim N(0, I_d)$, that is, $z\sim N(\mu, \alpha^2 \Sigma)$. We first consider $\mu$ part, by Eqn.~\eqref{eq:nab_mu} with $\bar{\Sigma} = \alpha^2\Sigma$, we have
	\begin{align*}
	\frac{\partial Q_\alpha(\theta)}{\partial \mu} =& \EE_z\left[f(z)\nabla_\mu\log\pi(z|\theta)\right]\\
	=&\EE_u\left[f(\mu+\alpha \Sigma^{1/2}u)\left(\alpha^2\Sigma\right)^{-1}\cdot\alpha \Sigma^{1/2}u\right]\\
	=&\EE_u  \left[f(\mu+\alpha \Sigma^{1/2}u)\alpha^{-1}\Sigma^{-1/2} u\right].
	\end{align*}
	Because of the symmetry of Gaussian distribution, we can set $z = \mu - \alpha\Sigma^{1/2}u$, and we can similarly derive that
	\begin{align*}
	\frac{\partial Q_\alpha(\theta)}{\partial \mu} =& \EE_z\left[f(z)\nabla_\mu\log\pi(z|\theta)\right]\\
	=&\EE_u\left[f(\mu-\alpha \Sigma^{1/2}u)\left(\alpha^2\Sigma\right)^{-1}\cdot\left(-\alpha \Sigma^{1/2}u\right)\right]\\
	=&-\EE_u  \left[f(\mu-\alpha \Sigma^{1/2}u)\alpha^{-1}\Sigma^{-1/2} u\right].
	\end{align*}
	Combining above two equations, we can obtain that
	\begin{equation}
	\frac{\partial Q_\alpha(\theta)}{\partial \mu} = \frac{1}{2\alpha}\cdot\EE_u\left[(f(\mu+\alpha \Sigma^{1/2}u) - f(\mu-\alpha \Sigma^{1/2}u)) \Sigma^{-1/2}u\right].
	\end{equation}
	
	With the knowledge of $\partial Q_\alpha(\theta)/\partial \mu$ and $F_\mu$ in Eqn.~\eqref{eq:F_mu}, we can obtain the result.
\end{proof}

With the natural gradient $g(\mu)$ at hand, we can update $\mu$ by the natural gradient descent as follows:
\begin{align*}
\mu_{k+1} =& \mu_k - \eta_1'g(\mu)\\
=&\mu_k - \eta_1 \cdot \frac{1}{2\alpha}\EE_u\left[(f(\mu+\alpha \Sigma^{1/2}u) - f(\mu-\alpha \Sigma^{1/2}u)) \Sigma^{1/2}u\right]
\end{align*} 
where $\eta_1 = \alpha^2\eta_1'$ is the step size. Note that, during the above update procedure, we need to compute the expectations which is infeasible in real applications. 
Instead,  we sample a mini-batch of size $b$ to approximate $\frac{1}{2\alpha}\EE\left[(f(\mu+\alpha \Sigma^{1/2}u) - f(\mu-\alpha \Sigma^{1/2}u)) \Sigma^{1/2}u\right]$, and we define
\begin{equation}\label{eq:g_mu}
\ti{g}(\mu_k) = \frac{1}{b}\sum_{i=1}^b \frac{f(\mu_k+ \alpha\Sigma_k^{1/2} u_i) -f(\mu_k - \alpha\Sigma_k^{1/2}u_i)}{2\alpha} \Sigma_k^{1/2} u_i\quad \mbox{with}\quad  u_i\sim N(0, I_d).
\end{equation}
Using $\ti{g}(\mu_k)$, we update $\mu$ as follows:
\begin{equation}
\mu_{k+1} = \mu_k - \eta_1 \ti{g}(\mu_k) .
\end{equation}

\paragraph{Mirror Descent of $\Sigma$}

Recall from the definition of $Q_\alpha$, we have $Q_\alpha(\theta) = J(\theta) + R(\Sigma)$.
With the regularizer $R(\Sigma)$, we can define the Bregman divergence with respect to $R(\Sigma)$ as
\begin{align*}
B_R(\Sigma_1, \Sigma_2) =& R(\Sigma_1) - R(\Sigma_2) - \dotprod{\nabla_\Sigma R(\Sigma_2), \Sigma_1 - \Sigma_2}\\
=&-\frac{\alpha^2}{2} \log\det\Sigma_1  - \left(-\frac{\alpha^2}{2} \log\det\Sigma_2 \right) 
-\dotprod{-\frac{\alpha^2}{2}\Sigma_2^{-1}, \Sigma_1-\Sigma_2}\\
=&-\frac{\alpha^2}{2}\left(\log\det\Sigma_1 - \log\det\Sigma_2 -\dotprod{\Sigma_2^{-1}, \Sigma_1 - \Sigma_2} \right)\\
=& -\frac{\alpha^2}{2}\left(\log\det(\Sigma_1\Sigma_2^{-1})
-\dotprod{\Sigma_2^{-1}, \Sigma_1} + d\right) .
\end{align*}
The update rule of $\Sigma$ employing mirror descent is defined as 
\begin{equation}\label{eq:mirror_prox}
\Sigma_{k+1} = \argmin_{\Sigma} \eta_2 \dotprod{\frac{\partial Q_\alpha(\theta)}{\partial \Sigma_k}, \Sigma} + B_R(\Sigma, \Sigma_k) .
\end{equation}
Using $\nabla_\Sigma R(\Sigma)$ as the mapping function, the update rule of above equation can be reduced to 
\begin{align}\label{eq:mirror}
\nabla_\Sigma R(\Sigma_{k+1}) = \nabla_\Sigma R(\Sigma_{k}) - \eta_2 \nabla_\Sigma Q_\alpha(\theta_k).
\end{align}
In the following lemma, we will compute $\frac{\partial Q_\alpha(\theta)}{\partial \Sigma}$. 
\begin{lemma}
	Let $Q_\alpha(\theta)$ be defined in Eqn.~\eqref{eq:Q}. Then it holds that
	\begin{equation}
	\small
	\label{eq:Q_Sig}
	\frac{\partial Q_\alpha(\theta) }{\partial \Sigma}= \frac{1}{4}\cdot\EE_u\left[\left(f(\mu - \alpha \Sigma^{1/2}u) + f(\mu + \alpha \Sigma^{1/2}u) - 2f(\mu) \right)\left(\Sigma^{-1/2}uu^\top \Sigma^{-1/2} -  \Sigma^{-1}\right)\right] - \frac{\alpha^2}{2}\Sigma^{-1}.
	\end{equation}
\end{lemma}
\begin{proof}
	Note that $\partial Q_\alpha(\theta)/\partial \Sigma = \partial J(\theta)/\partial \Sigma +  \partial R(\Sigma)/\partial \Sigma$.
	First, we will compute $\frac{\partial J(\theta)}{\partial \Sigma}$. 
	Let $z = \mu + \alpha \Sigma^{1/2}u$ with $u\sim N(0, I_d)$. 
	By Eqn.~\eqref{eq:nab_Sig}, we can obtain that
	\begin{align*}
	\frac{\partial J(\theta)}{\partial \Sigma} =&\frac{\partial J(\theta)}{\partial \bar{\Sigma}} \cdot \frac{\partial \bar{\Sigma}}{\partial \Sigma}\\
	\overset{\eqref{eq:nab_Sig}}{=}& \EE_z\left[f(z) \left(\frac{1}{2}\Sigma^{-1}(z-\mu)(z-\mu)^\top\Sigma^{-1}\alpha^{-2} - \frac{1}{2} \Sigma^{-1}\alpha^{-2}\right)\right]\cdot\alpha^2\\
	=&\frac{1}{2}\cdot\EE_u\left[f(\mu+ \alpha \Sigma^{1/2}u) \left(\Sigma^{-1/2}uu^\top \Sigma^{-1/2} -  \Sigma^{-1}\right)\right].
	\end{align*}
	Because of the symmetry of Gaussian distribution, we can also have $z = \mu - \alpha \Sigma^{1/2} u$. Then, we can similarly derive that 
	\begin{align*}
	\frac{\partial J(\theta)}{\partial \Sigma}
	=\frac{1}{2}\cdot\EE_u\left[f(\mu- \alpha \Sigma^{1/2}u) \left(\Sigma^{-1/2}uu^\top \Sigma^{-1/2} -  \Sigma^{-1}\right)\right].
	\end{align*}
	Note that, we also have the following identity
	\begin{equation*}
	\EE_u\left[f(\mu) \left(\Sigma^{-1/2}uu^\top \Sigma^{-1/2} -  \Sigma^{-1}\right)\right] = 0.
	\end{equation*}
	Combining above equations, we can obtain that 
	\begin{equation}\label{eq:J_Sig}
	\frac{\partial J(\theta)}{\partial \Sigma} = \frac{1}{4}\cdot\EE_u\left[\left(f(\mu - \alpha \Sigma^{1/2}u) + f(\mu + \alpha \Sigma^{1/2}u) - 2f(\mu) \right)\left(\Sigma^{-1/2}uu^\top \Sigma^{-1/2} -  \Sigma^{-1}\right)\right].
	\end{equation}
	Furthermore, by the definition of $Q_\alpha$ in Eqn.~\eqref{eq:Q}, we have
	\begin{equation*}
	\frac{\partial Q_\alpha(\theta) }{\partial \Sigma}
	= \frac{\partial J(\theta)}{\partial \Sigma} - \frac{\alpha^2}{2}\Sigma^{-1} .
	\end{equation*}
	Therefore, we can obtain the result.
\end{proof}

Since, we also have
\begin{equation*}
\nabla_\Sigma R(\Sigma) = -\frac{\alpha^2}{2}\Sigma^{-1}.
\end{equation*}
Substituting $\nabla_{\Sigma} Q_\alpha(\theta)$ and $\nabla_\Sigma R(\Sigma)$ in Eqn.~\eqref{eq:mirror}, we have
\begin{align}
-\frac{\alpha^2}{2}\Sigma_{k+1}^{-1}=& -\frac{\alpha^2}{2}\Sigma_k^{-1}  - \eta_2 \nabla_{\Sigma} Q_\alpha(\theta_k) \notag\\
\Rightarrow \Sigma_{k+1}^{-1} = &\Sigma_k^{-1} + 2\eta_2\alpha^{-2} \frac{\partial Q_\alpha(\theta_k)}{\partial \Sigma}. \label{eq:Sig_up}
\end{align}

Similar to the update of $\mu$, we only sample a small batch points to query their values and use them to estimate $\partial Q_\alpha(\theta)/\partial \Sigma$. 
We can construct the approximate gradient with respect to $\Sigma$ as follows: 
\begin{equation}\label{eq:TG}
\TG(\Sigma_k) 
= 
\frac{1}{2b\alpha^2} \sum_{i=1}^{b}
\left[\left(f(\mu_k - \alpha \Sigma_k^{1/2}u_i) + f(\mu_k + \alpha \Sigma_k^{1/2}u_i) - 2f(\mu_k) \right)\left(\Sigma_k^{-1/2}u_iu_i^\top \Sigma_k^{-1/2} -  \Sigma_k^{-1}\right)\right] 
- \Sigma_k^{-1}.
\end{equation}
The following lemma shows that an important property of $\TG(\Sigma_k) $.
\begin{lemma}\label{prop:TG}
	Let $\TG(\Sigma_k)$ be defined in Eqn.~\eqref{eq:TG}, then $\TG(\Sigma_k)$ is an unbiased estimation of $2 \frac{\partial Q_\alpha(\theta)}{\alpha^2\partial \Sigma}$ at $\Sigma_k$.
\end{lemma}
\begin{proof}
	By Eqn.~\eqref{eq:J_Sig}, we can observe that 
	\begin{equation*}
	\BG(\Sigma_k)=\frac{1}{2b\alpha^2} \sum_{i=1}^{b}
	\left[\left(f(\mu_k - \alpha \Sigma_k^{1/2}u_i) + f(\mu_k + \alpha \Sigma_k^{1/2}u_i) - 2f(\mu_k) \right)\left(\Sigma_k^{-1/2}u_iu_i^\top \Sigma_k^{-1/2}\Sigma_k^{-1/2} -  \Sigma_k^{-1}\right)\right] 
	\end{equation*}
	is an unbiased estimation of $2\alpha^{-2}\partial J(\theta)/\partial \Sigma$ at $\Sigma_k$. Furthermore, we have
	\begin{equation*}
	\nabla_\Sigma R(\Sigma) = -\frac{\alpha^2}{2}\Sigma^{-1}.
	\end{equation*}
	Therefore, we can conclude that $\TG(\Sigma_k)$ is an unbiased estimation of $2\alpha^{-2} \partial Q_\alpha(\theta)/\partial \Sigma$ at $\Sigma_k$.
\end{proof}
Replacing $2\alpha^{-2} \partial Q_\alpha(\theta_k)/\partial \Sigma$ with  $\TG(\Sigma_k)$ in Eqn.~\eqref{eq:Sig_up}, we update $\Sigma$ as follows
\begin{equation}
\Sigma_{k+1}^{-1} = \Sigma_k^{-1} + \eta_2 \ti{G}(\Sigma_k).
\end{equation}
where $\eta_2$ is the step size. 

\paragraph{Projection to The Constraint}
Because of the constraint that $\Sigma_{k+1}\in\cS$, we need to project $\Sigma_{k+1}$ back to $\cS$. 
Since we update $\Sigma^{-1}$ instead of directly updating $\Sigma$,  we define another convex set $\mathcal{S}'$
\begin{equation}\label{eq:cS_p}
\mathcal{S}'=\bigg\{\Sigma^{-1}\bigg| \Sigma\in\mathcal{S}\bigg\}
\end{equation}
It is easy to check that for any $\Sigma^{-1}\in\cS'$, then it holds that $\Sigma\in\cS$. 
Taking the extra projection to $\cS'$, 
we modify the update rule of $\Sigma$ as follows;
\begin{equation}\label{eq:Sig_modi}
\left\{
\begin{aligned}
\Sigma_{k+0.5}^{-1} =& \Sigma_k^{-1} + \eta_2 \TG(\Sigma_k)\\
\Sigma_{k+1}^{-1} =& \Pi_{\mathcal{S}'}\left(\Sigma_{k+0.5}^{-1}\right)
\end{aligned}
\right.
\end{equation}

The projection $\Pi_{\mathcal{S}'}(\Sigma^{-1})$ is conducted as follows. 
First, we conduct the spectral decomposition $\Sigma^{-1} = U\Lambda U^\top$, where $U$ is an orthonormal matrix and $\Lambda$ is a diagonal matrix with $\Lambda_{i,i} = \lambda_i$. 
Second, we truncate $\lambda_i$'s. If $\lambda_i > \zeta$, we set $\lambda_i = \zeta$. If $\lambda_i < \tau$, we set $\lambda_i = \tau$, that is
\begin{equation}\label{eq:proj_cs_prime}
\Pi_{\cS'}\left(\Sigma^{-1}\right) = U\bar{\Lambda}U^\top, 
\quad
\mbox{with}
\quad
\bar{\Lambda}_{i,i} = \left\{
\begin{aligned}
&\tau^{-1}    \qquad \mbox{if}\;\lambda_i(\Sigma^{-1})>\tau^{-1}\\
&\zeta^{-1}   \qquad \mbox{if}\;\lambda_i(\Sigma^{-1})<\zeta^{-1}\\
&\lambda_i(\Sigma^{-1})    \qquad \mbox{otherwise}
\end{aligned}
\right.
\end{equation} 
where $\bar{\Lambda}$ is a diagonal matrix. It is easy to check the
correctness of Eqn.~\eqref{eq:proj_cs_prime}. For completeness, we
prove it in Proposition~\ref{prop:project} in the Appendix.

\begin{algorithm}[tb]
	\caption{Meta-Algorithm \texttt{MiNES}}
	\label{alg:zero_order}
	\begin{small}
		\begin{algorithmic}[1]
			\STATE {\bf Input:} $\mu_1$, $\Sigma_1^{-1}$, $\eta_1$, $\eta_2$ and $\alpha$
			\FOR {$k=1,\dots,$ }
			\STATE Compute $\ti{g}(\mu_k) = \frac{1}{b}\sum_{i=1}^b \frac{f(\mu_k+ \alpha\Sigma_k^{1/2} u_i) -f(\mu_k - \alpha\Sigma_k^{1/2}u_i)}{2\alpha} \Sigma_k^{1/2} u_i$ with $u_i\sim N(0, I_d)$
			\STATE Compute $\ti{G}(\Sigma_k) = \frac{1}{2b\alpha^2} \sum_{i=1}^{b}\left[\left(f(\mu_k - \alpha \Sigma_k^{1/2}u_i) + f(\mu_k + \alpha \Sigma_k^{1/2}u_i) - 2f(\mu_k) \right)\left(\Sigma_k^{-1/2}u_iu_i^\top \Sigma_k^{-1/2} -  \Sigma_k^{-1}\right)\right]-\Sigma_k^{-1}$
			\STATE Update $\mu_{k+1} = \mu_k - \eta_1 \ti{g}(\mu_k)$  \label{step:update} 
			\STATE Update $\Sigma^{-1}_{k+1} = \Pi_{\cS'}\left(\Sigma^{-1}_k + \eta_2\ti{G}(\Sigma_k)\right) $ \label{step:Sig_update}
			\ENDFOR
		\end{algorithmic}
	\end{small}
\end{algorithm}	

\paragraph{Algorithmic Summary of \texttt{MiNES}}

Now, we summarize the algorithmic procedure of \texttt{MiNES}. First, we update $\mu$ by  natural gradient descent and update $\Sigma$ by mirror descent as 
\begin{align*}
\mu_{k+1} =& \mu_k - \eta_1 \ti{g}(\mu_k)\\
\Sigma^{-1}_{k+1} =& \Pi_{\cS'}\left(\Sigma^{-1}_k +
\eta_2\ti{G}(\Sigma_k)\right) ,
\end{align*} 
where $\ti{g}(\mu_k)$ and $\TG(\Sigma_k)$ are defined in Eqn.~\eqref{eq:g_mu} and~\eqref{eq:TG}, respectively. The detailed algorithm description is in Algorithm~\ref{alg:zero_order}. 

\subsection{Relation to Existing Work}

First, we compare \texttt{MiNES} to derivative free algorithms in the
optimization literature, which uses function value differences to estimate the gradient. In the work of \cite{Nesterov2017}, one approximates the gradient as follows
\begin{equation}
g(\mu_k) = \frac{1}{b}\sum_{i=1}^b \frac{f(\mu_k+ \alpha u_i) -f(\mu_k
	- \alpha u_i)}{2\alpha}  u_i \quad\mbox{with} \quad u_i\sim N(0,
I_d) ,
\end{equation}
and update  $\mu$ as 
\begin{equation}
\mu_{k+1} = \mu_k - \eta_1 g(\mu_k) .
\end{equation}
Comparing $\ti{g}(\mu)$ to $g(\mu)$, we can observe that the
difference lies on the estimated covariance matrix. 
By utilizing $\Sigma_k$ to approximate the inverse of the Hessian , $\ti{g}(\mu_k)$  is an estimation of the natural gradient. 
In contrast, $g(\mu_k)$ just uses the identity matrix hence it only estimates the gradient of $f(\mu_k)$. 
Note that if we don't track the Hessian information by updating
$\Sigma_k$, and set $\Sigma_k$ to the identity matrix, then
\texttt{MiNES} becomes the derivative free algorithm of \cite{Nesterov2017}. This establishes a connection between \texttt{NES} and derivative free algorithms.

We may compare \texttt{MiNES} to the classical derivative-free algorithm \citep{conn2009global}, which is also a second order method.
However, unlike \texttt{MiNES}, the gradient and Hessian of function $f(\cdot)$ are computed approximately by estimating each component of the corresponding vector and matrix using regression. This requires $O(d^2)$ queries to function values at each step, which is significantly more costly than \texttt{MiNES} and other \texttt{NES}-type algorithms (when $d$ is large).

We may also compare \texttt{MiNES} to the traditional \texttt{NES} algorithms (including \texttt{CMA-ES} since \texttt{CMA-ES} can be derived from \texttt{NES} \citep{AkimotoNOK10}). 
There are  two differences between \texttt{MiNES} and the conventional \texttt{NES} algorithms.
First, \texttt{MiNES} minimizes $Q_\alpha$, while \texttt{NES} minimizes $J(\theta)$ defined in Eqn.~\eqref{eq:J_org}. 
Second, the update rule of $\Sigma$ is different. 
\texttt{MiNES} uses the mirror descent to update $\Sigma^{-1}$.
In comparison, \texttt{NES} uses the natural gradient to update $\Sigma$.

\section{Convergence Analysis}
\label{sec:conv}

In this section we analyze the convergence properties of \texttt{MiNES}. 
We will only consider the situation that function $f(\cdot)$ is quadratic.
First we will give the convergence rate of $\Sigma$ for \texttt{MiNES}. 
Second, we will analyze the convergence rate of $f(\mu)$ and show how $\Sigma$ affects the convergence behavior of $f(\mu)$.
For simplicity, we will only consider the case that batch size is one, that is, $b=1$. 

\subsection{Convergence Analysis of $\Sigma$} 

Since the function $f(\cdot)$ is quadratic, its Hessian  is a constant matrix independent of $\mu$.  Let us denote that
$H = \nabla^2f(\mu)$.

\begin{theorem}\label{thm:Sig_conv}
	If we choose the step size $\eta_2^{(k)} = \frac{1}{k}$ in \texttt{MiNES}, then  
	\begin{equation*}
	\EE\left[\norm{\Sigma_{k}^{-1}-\Pi_{\cS'}(H)}^2\right] \leq \frac{1}{k} \cdot\max \{\norm{\Sigma_1 - H }^2, M\}, 
	\end{equation*}
	where 
	\begin{equation*}
	M = \frac{L^2\zeta^2}{4\tau^2}\left(d^4+11d^3+34d^2+32d\right) +2d\zeta^2 + \norm{H}^2 .
	\end{equation*}
\end{theorem}
Note that if we set $\zeta \geq L$ and $\tau \leq \sigma$, then $\Pi_{\cS'}(H)=H$. Thus, $\Sigma_k$ converges to $H$. 
\begin{corollary}
	If we choose $\zeta \geq L$ and $\tau \leq \sigma$, and let the step size $\eta_2^{(k)} = \frac{1}{k}$ in Algorithm~\ref{alg:zero_order}, then \texttt{MiNES} converges as 
	\begin{equation*}
	\EE\left[\norm{\Sigma_{k}^{-1}-H}^2\right] \leq \frac{1}{k} \cdot\max \{\norm{\Sigma_1 - H }^2, M\}.
	\end{equation*}
\end{corollary}

From the above corollary, we can observe that under stuiable assumptions, $\Sigma_{k}$ converges to the Hessian inverse matrix with a sublinear rate.  This result is the \emph{first} rigorous analysis of the convergence rate of $\Sigma$ for \texttt{ES} algorithms.

\begin{remark}
	Though Theorem~\ref{thm:Sig_conv} only provides the convergence rate of $\Sigma$ when the function $f(\cdot)$ is quadratic, we can extend it to the general strongly convex case. 
	By Theorem~\ref{prop:mu} and~\ref{prop:Sigma}, we know that \texttt{MiNES} converges to $\left(\Hu_*, \HSig_*\right)$ and $[\HSig_*]^{-1}$ is close to $\nabla^2f(\mu_*)$.
	Hence, Theorem~\ref{thm:Sig_conv} implies the local convergence properties of $\Sigma$ for the general strongly convex case. That is, $\Sigma_k^{-1}$ converges to $\nabla^2f(\mu_*)$ sublinearly when $\mu$ is close to $\mu_*$.  
\end{remark}

Next, we will give a high probability version of the convergence rate of $\Sigma_k$. 
\begin{theorem}\label{thm:Sig_conv_hp}
	Let $\delta \in (0, 1/2)$ and assume $T\geq 4$. 	If we choose the step size $\eta_2^{(k)} = \frac{1}{k}$ in \texttt{MiNES}, then it holds with probability at least $1-2\delta$ that for any $k\leq T$,
	\begin{equation*}
	\norm{\Sigma_{k}^{-1}-\Pi_{\cS'}(H)}^2 \leq \frac{1}{k} R',
	\end{equation*}
	where $R'$ is defined as 
	\begin{align*}
	R' = 16\cdot& \left( \frac{16d^2\zeta^2L}{\tau} + \frac{2d^3L(\zeta^2 - \tau^2)}{\tau} + dL\zeta + dL^2\right) \log^{5/4} \left(\frac{12 + 2592 T}{\delta}\right)\\
	&+ \left(\frac{L^2\zeta^2}{\tau} \left((2d + 3\log (T/\delta))^4 + (2d + 3\log (T/\delta))^2\right) + 2L^2\right).
	\end{align*}
\end{theorem}
\subsection{Convergence Analysis of $f(\mu)$}

The convergence analysis of $f(\mu)$ is similar to that of \texttt{ZOHA} \citep{ye2018hessian}. 
\citet{ye2018hessian} has provided the convergence analysis on $\mu$ where $\Sigma^{-1}$ can be viewed as an approximate Hessian matrix. 
In this paper, we will only analyze the case that the function $f(\cdot)$ is quadratic. For the general strongly convex case, one can find it in  \citep{ye2018hessian}.

First, we give a lemma describing how well $\Sigma_k^{-1}$ approximate the Hessian.

\begin{lemma}\label{lem:appr}
	Assume that $\zeta\geq L$ and $\tau \leq \sigma$ and $k\geq \frac{16R'}{\sigma^2}$ where $R'$ is defined in Theorem~\ref{thm:Sig_conv_hp}. Then \texttt{MiNES} with $\eta_2^{(k)} = 1/k$ satisfies that it holds with probability at least $1-2\delta$ with $\delta\in(0,0.5)$ that  
	\begin{equation*}
	\left(1 -\frac{\sigma/4}{L + \sigma/4}\right)\Sigma_{k}^{-1} \preceq H\preceq \left(1 + \frac{1}{3}\right)\Sigma_{k}^{-1} .
	\end{equation*}
\end{lemma}
Next, we will give the convergence rate of $f(\mu)$.
\begin{theorem}\label{thm:mu}
	Let quadratic $f(\cdot)$ be $\sigma$-strongly convex and $L$-smooth. Assume that the covariance matrix $\Sigma_k$ satisfy the properties described in Lemma~\ref{lem:appr}. 
	By setting the step size $\eta_1^{(k)} = \frac{1}{2(d+2)}$, \texttt{MiNES} with batch size $b = 1$ has the following convergence properties:
	\begin{align*}
	\EE\left[f(\mu_{k+1})-f(\mu_*)\right] 
	\leq \left(1 - \frac{1}{3(d+2)}\right)\left(f(\mu_k) -f(\mu_*) \right).
	\end{align*}
\end{theorem}

\begin{remark}
	Theorem~\ref{thm:mu} gives the convergence rate of $\mu$ of \texttt{MiNES}. 
	We can observe that $\Sigma$ helps to improve the convergence rate of $\mu$ if $\Sigma^{-1}$ approximates the Hessian well. The above theorem shows that the convergence rate of \texttt{MiNES} is condition number free when $k\geq \frac{16R'}{\sigma^2}$. 
\end{remark}
\begin{remark}
	It is well known that a strongly convex function can be well approximated by a quadratic function if $\mu$ is near the minimizer. 
	Thus, Theorem~\ref{thm:mu} implies a condition number independent local convergence rate of a general strongly convex function. In particular,
	by Theorem~\ref{prop:mu}, we know that $\mu_k$ in \texttt{MiNES} converges to $\hat{\mu}_*$ which is close to $\mu_*$ if $\alpha$ is small.
\end{remark}
\begin{table}[]
	\centering 
	\begin{tabular}{c|c|c}  
		\hline  
		~~~ssphere	~~~&~~~$f_{\mathrm{ssphere}} = \sqrt{\sum_{i=1}^{d}x_i^2}$~~~& ~~~$d = 400$~~~\\ 
		\hline
		~~~quadratic~~~&~~~$f_{\mathrm{quadratic}} = \frac{1}{2}x^\top A x$ ~~~&~~~$d = 200$~~~ \\
		\hline
		~~~diff.powers~~~&~~~$f_{\mathrm{diffpow}} = \sum_{i=1}^{d} |x_i|^{2+10\frac{i-1}{n-1}}$~~~&~~~$d = 100$~~~ \\
		\hline
	\end{tabular}
	\caption{Function Description} 
\label{tb:functions}  
\end{table}

\section{Experiments}
\label{sec:experiments}

In previous sections, we proposed  \texttt{MiNES} and analyze its convergence rate. 
In this section, we will study \texttt{MiNES} empirically. 
First, we will conduct experiments on three synthetic functions. 
Second, we evaluate our algorithm on logistic regression with different datasets. We will compare \texttt{MiNES} with derivative free algorithm (\texttt{DF} \citep{Nesterov2017}), \texttt{NES} \citep{wierstra2014natural} and \texttt{CMA-ES} \citep{hansen2016cma}.

\subsection{Empirical Study on Synthetic Functions}

Four synthetic functions are selected to evaluate \texttt{MiNES}. They are `quadratic function', `ssphere', and `diffpow'. The dimensions $d$ of these functions vary from $100$ to $400$.  The quadratic function takes the form ~$f_{\mathrm{quadratic}} = \frac{1}{2}x^\top A x$ with $A$ being positive definite. 
In our experiments, $A$ is a $200\times 200$ matrix with a condition number $2.306\times 10^3$. 
The detailed descriptions of other synthetic functions  are listed in Table~\ref{tb:functions}. 
We report the results in Figure~\ref{fig:syn}.

\begin{figure}[]
	\subfigtopskip = 0pt
	\begin{center}
		\centering
		\subfigure[\text{quadratic}]{\includegraphics[width=45mm]{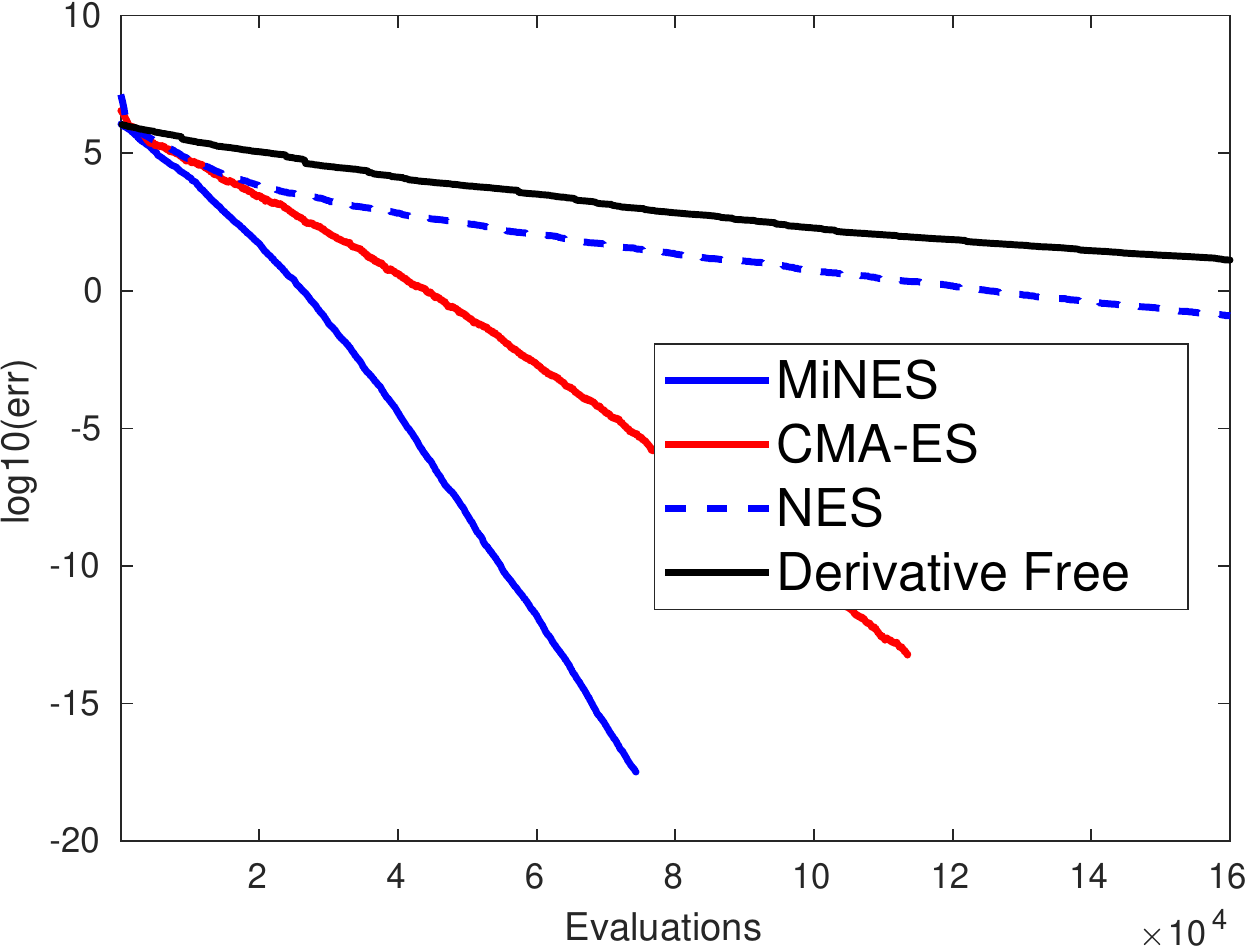}}~
		\subfigure[\text{ssphere}]{\includegraphics[width=43mm]{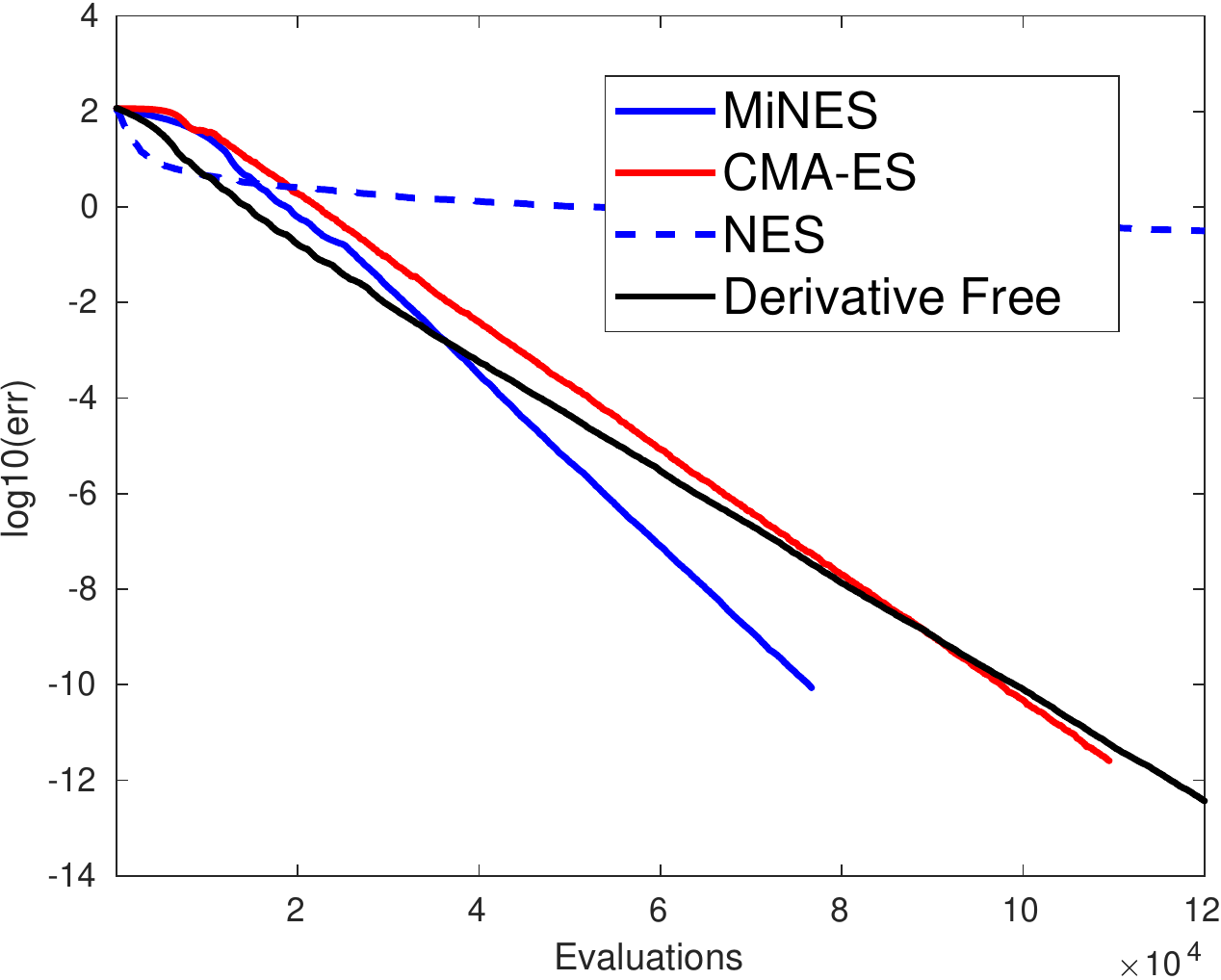}}~
		\subfigure[\text{diff powers}]{\includegraphics[width=43mm]{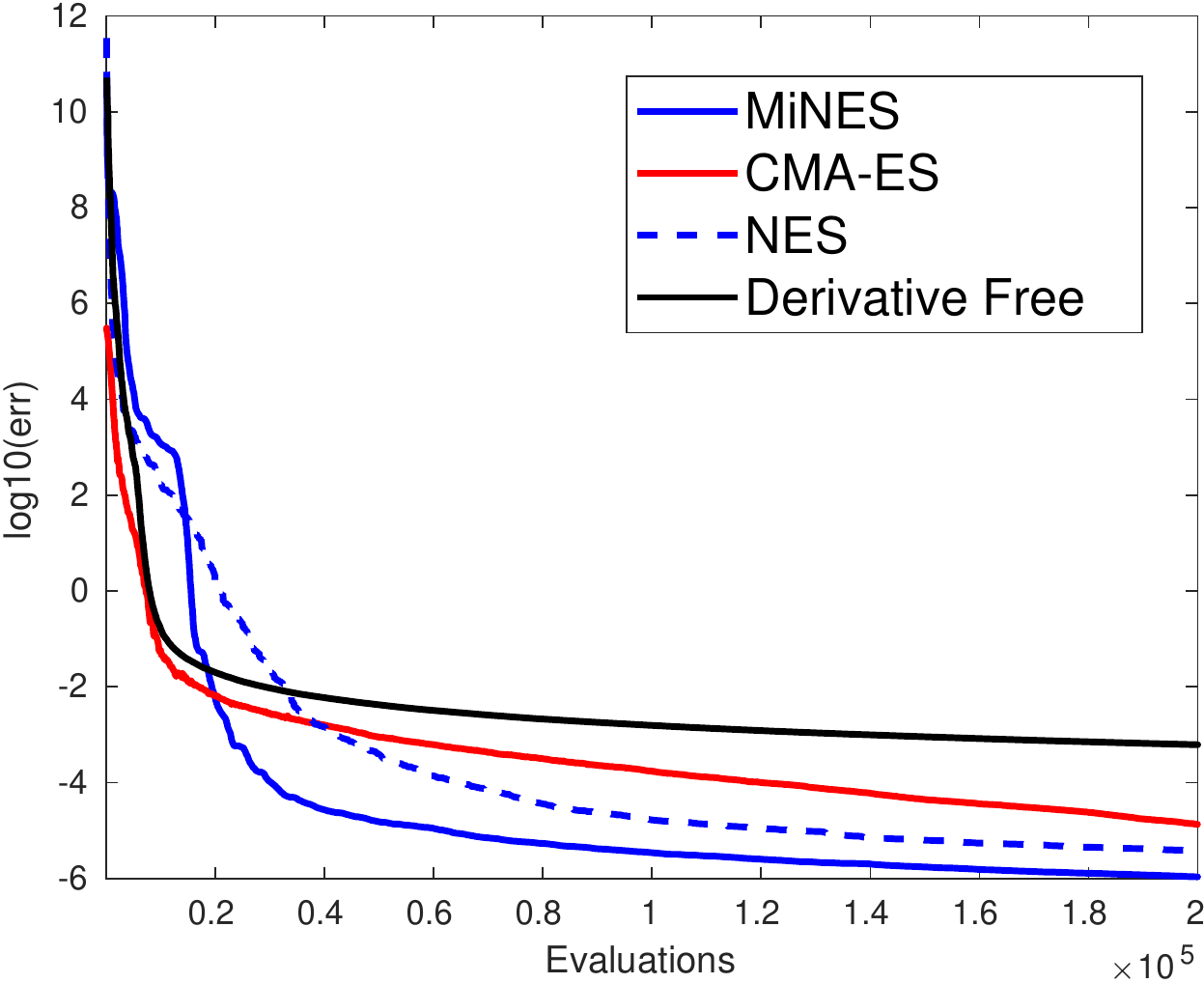}}
	\end{center}
	\caption{Evaluation on synthetic functions}
	\label{fig:syn}
\end{figure}

From Figure~\ref{fig:syn}, we can observe that \texttt{MiNES} outperforms the derivative free (\texttt{DF}) algorithm. 
This is because \texttt{DF} does not exploit the Hessian information while \texttt{MiNES} tracks the Hessian and uses it to accelerate the convergence.
Furthermore, we can observe that \texttt{MiNES} has better performance than \texttt{CMA-ES} on the synthetic functions. Note that the first three synthetic functions in Table~\ref{tb:functions} are all convex. 

Although the experimental results show that \texttt{MiNES} achieves performance comparable to that of \texttt{CMA-ES}, \texttt{MiNES} has an extra tuning parameters than \texttt{CMA-ES}, which is $\eta_2$ in Algorithm~\ref{alg:zero_order}. Parameter $\eta_2$ is important and needs to be well tuned since it can affect the convergence rate of Algorithm~\ref{alg:zero_order} greatly. Therefore, \texttt{CMA-ES} is also competitive in most cases because \texttt{CMA-ES} is easy to tune.

\begin{table*}[]
	\centering
	\begin{tabular}{cccc}
		\hline
		Dataset~~~~ &~~~~ $n$~~~~&~~~~$d$~~~~&~~~~source \\ \hline
		mushroom~~~~ &~~~~ $8,124$~~~~&~~~~$112$~~~~&~~~~libsvm dataset   \\
		splice~~~~   &~~~~ $1,000$~~~~&~~~~$60$~~~~&~~~~libsvm dataset     \\ 
		a9a~~~~&~~~~$32,561$~~~~ &~~~~ $123$~~~~ &~~~libsvm dataset     \\ 
		w8a ~~~~&~~~~$49,749$~~~~&~~~~$300$~~~~&~~~~libsvm dataset      \\ 
		a1a~~~~&~~~~$1,605$~~~~&~~~~$123$~~~~&~~~~libsvm dataset     \\
		ijcnn1~~~&~~~$49,990$~~~~&~~~~$22$~~~~&~~~~libsvm dataset     \\
		\hline
	\end{tabular}
		\caption{Datasets summary}
\label{tb:data}
\end{table*}

\begin{figure}[]
	\subfigtopskip = 0pt
	\begin{center}
		\centering
		\subfigure['mushroom' training ]{\includegraphics[width=36mm]{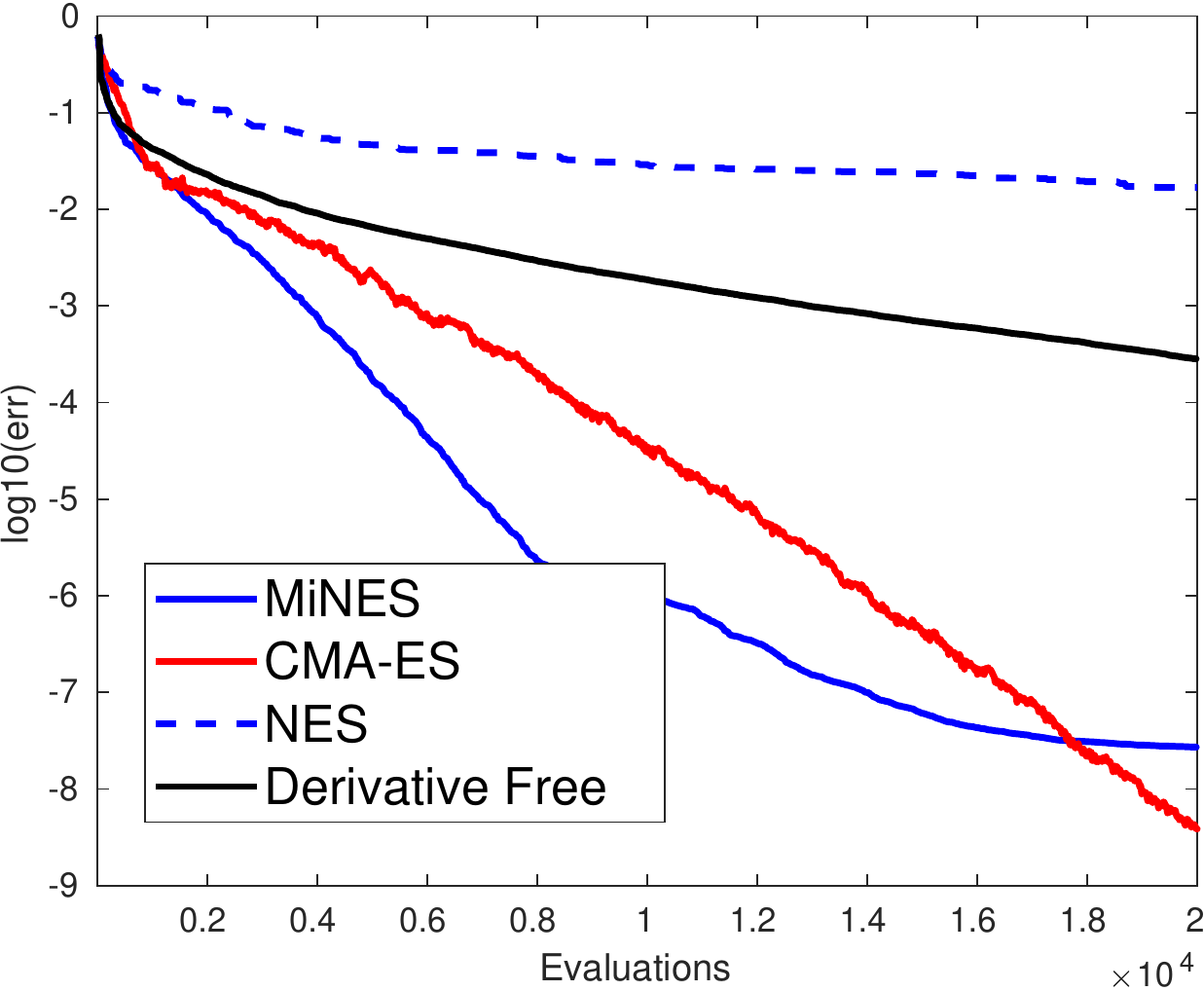}}~
		\subfigure['mushroom' test ]{\includegraphics[width=38mm]{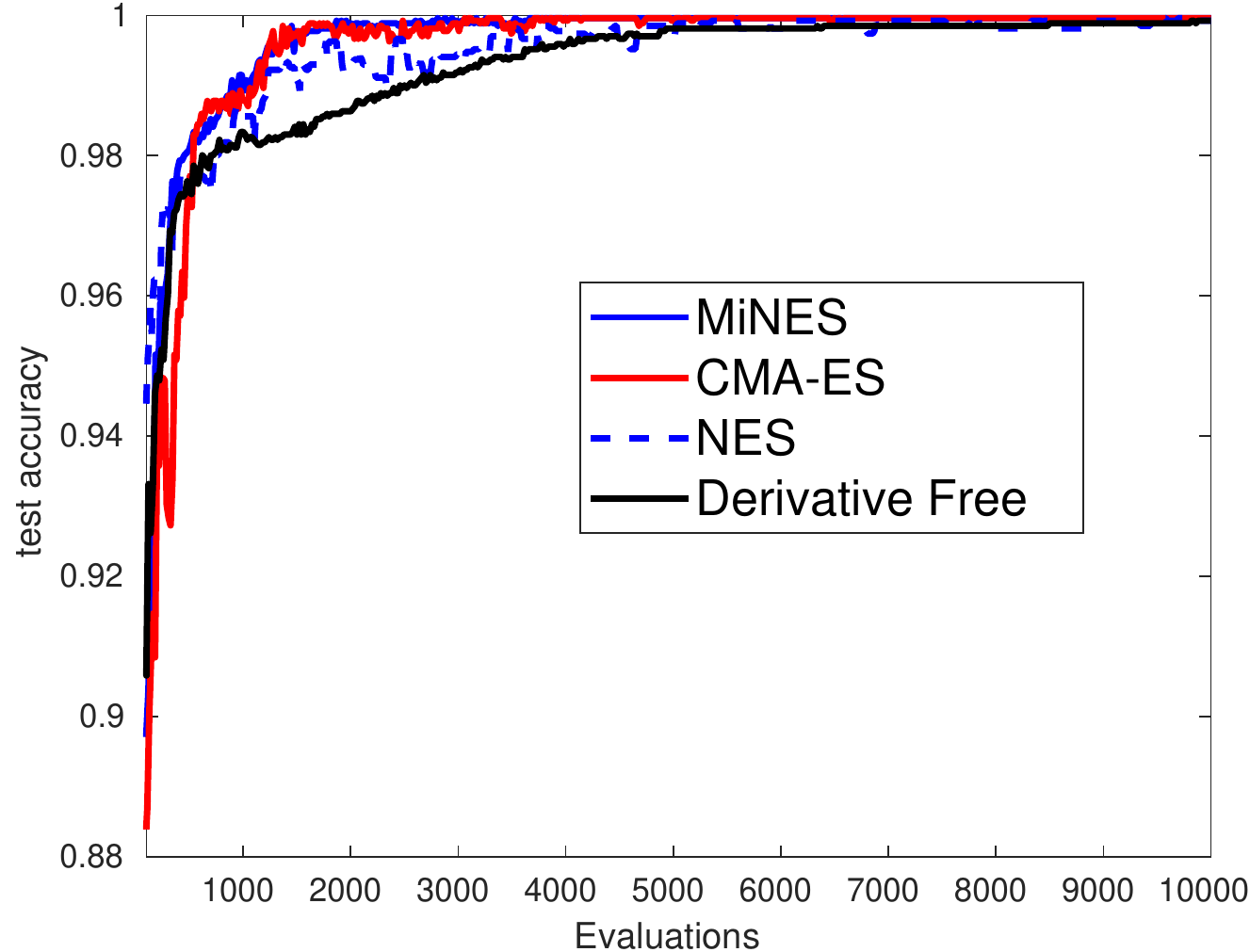}}~
		\subfigure[`splice' training ]{\includegraphics[width=35mm]{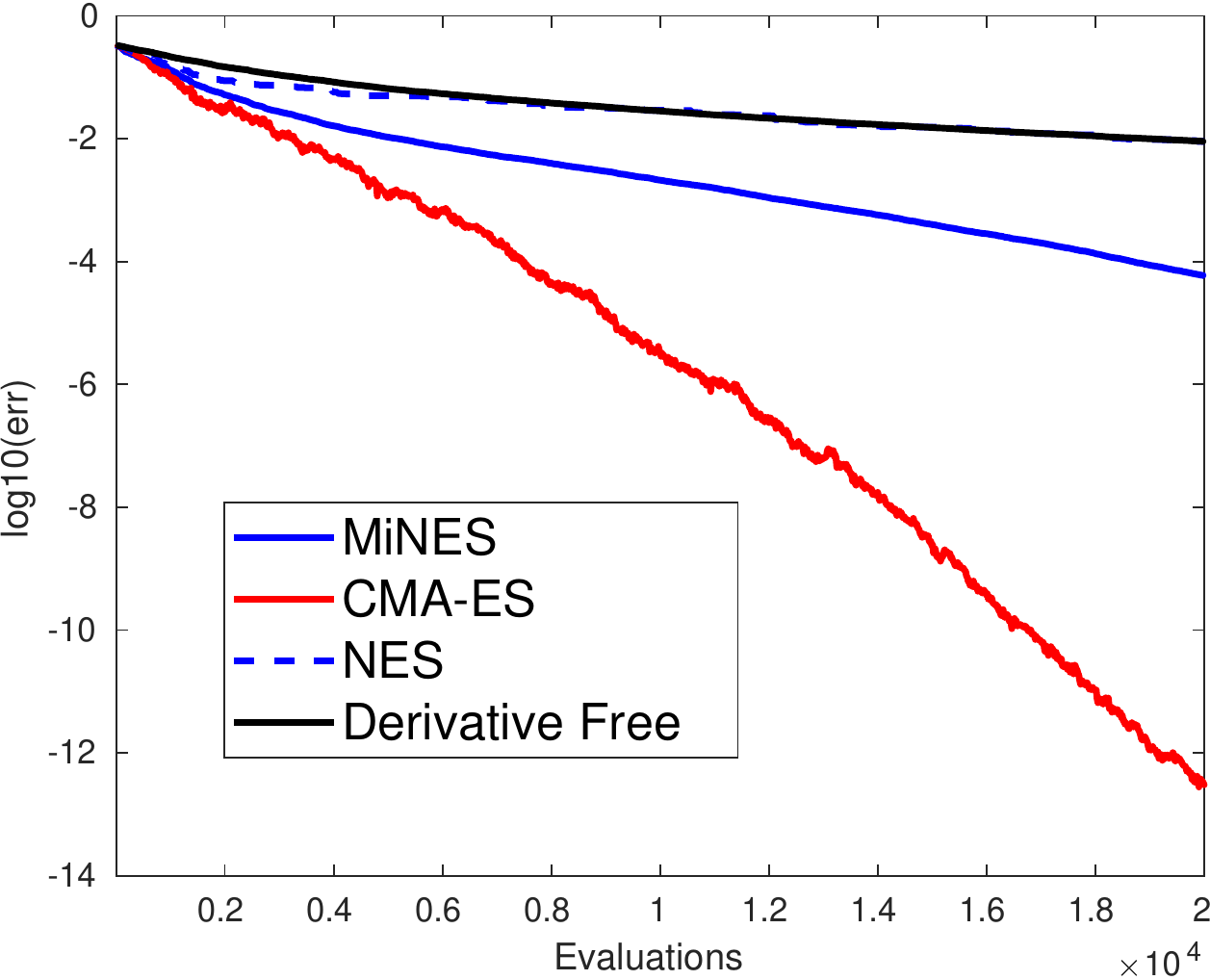}}~
		\subfigure[`splice' test ]{\includegraphics[width=39mm]{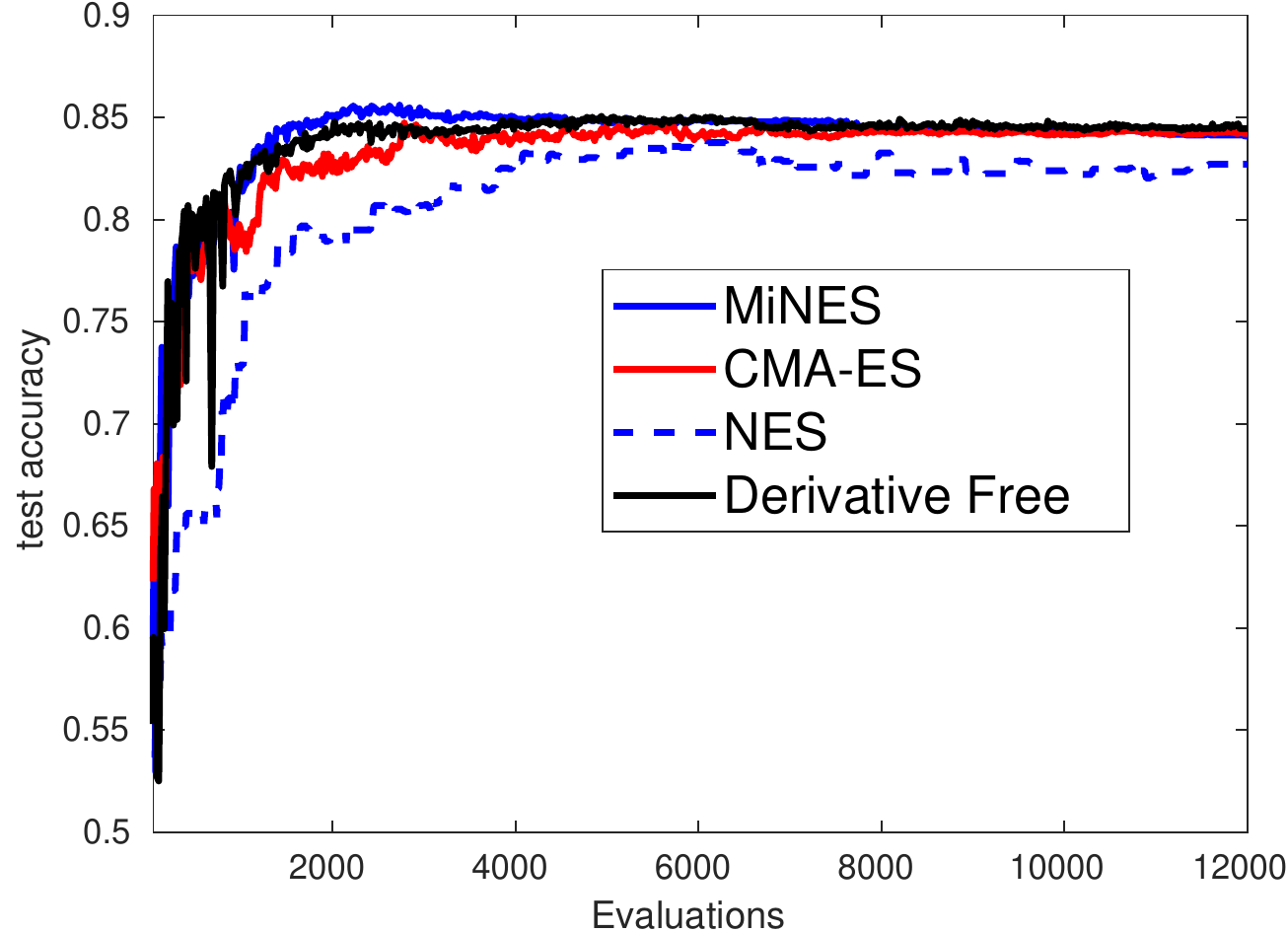}}\\
		\subfigure[`a9a' training ]{\includegraphics[width=39mm]{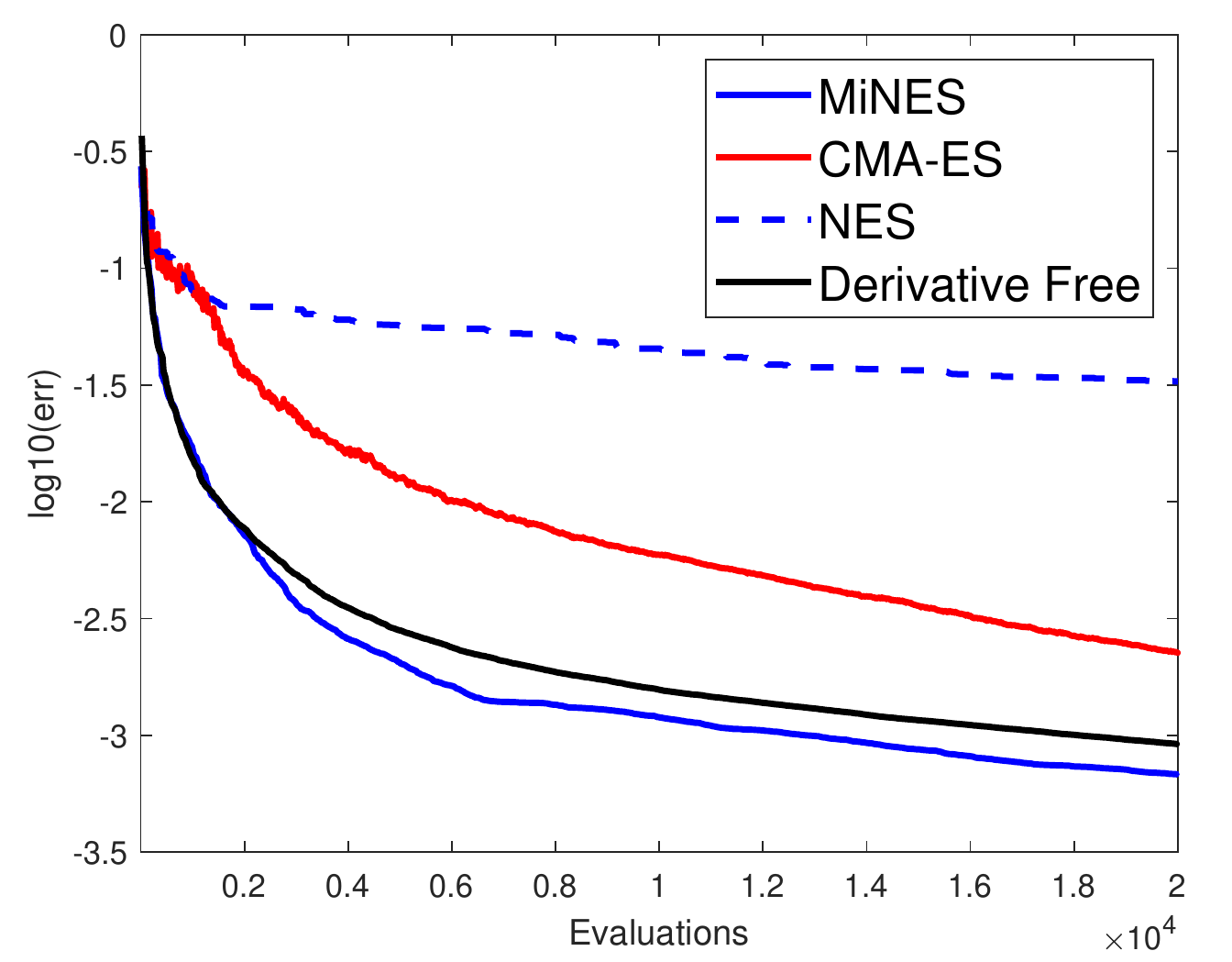}}~
		\subfigure[`a9a' test ]{\includegraphics[width=39mm]{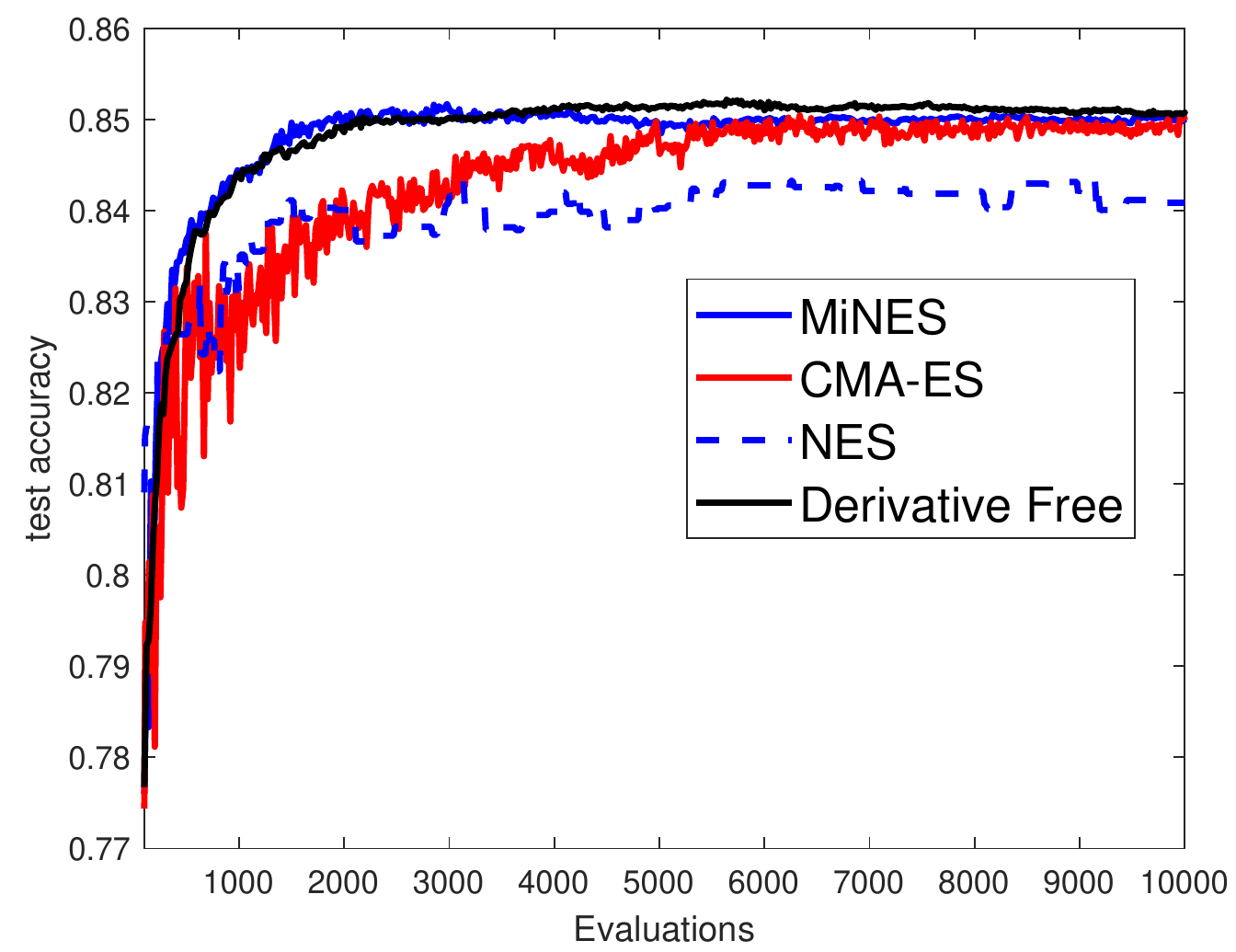}}~
		\subfigure[`w8a' training]{\includegraphics[width=38mm]{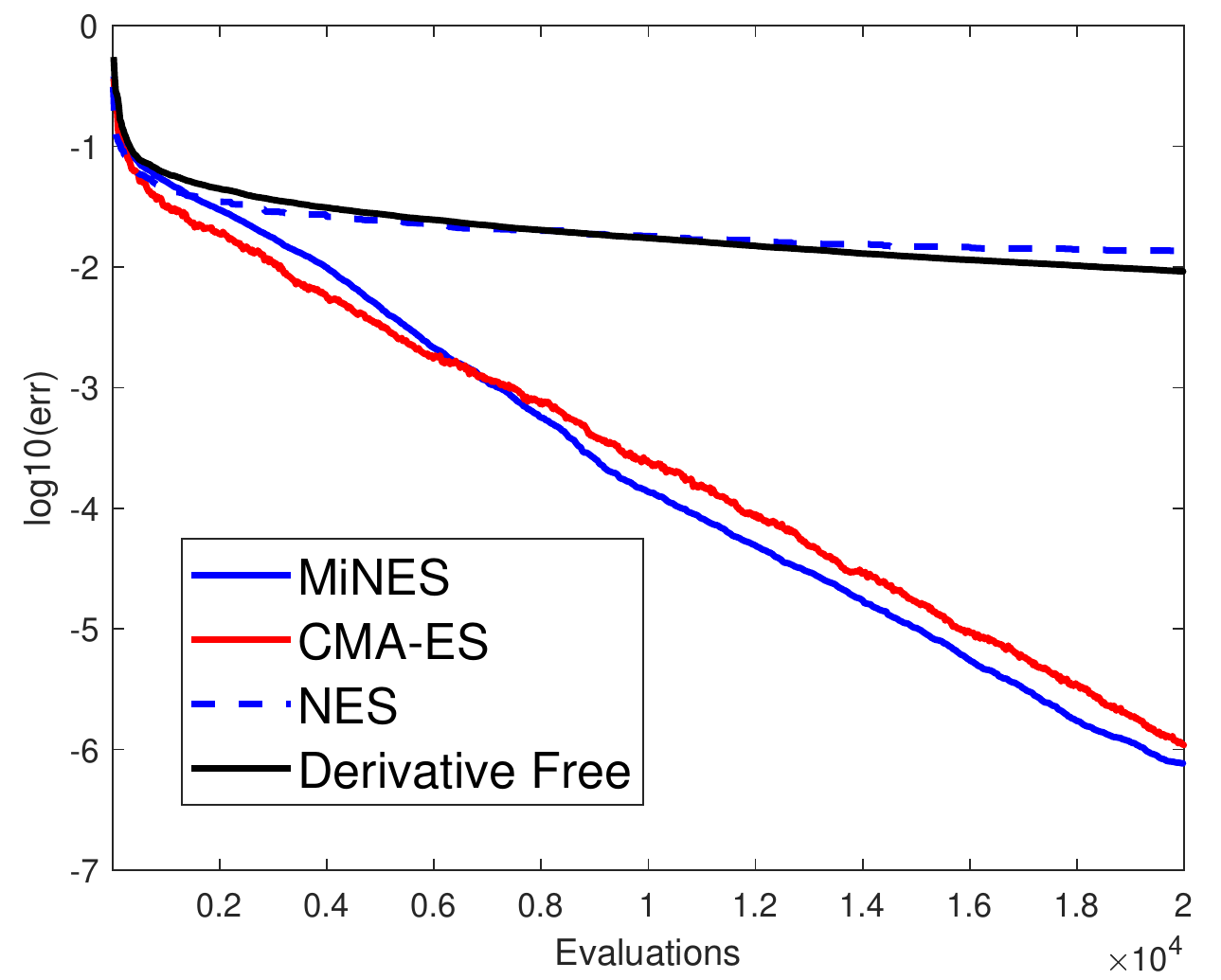}}~
		\subfigure[`w8a' test]{\includegraphics[width=39mm]{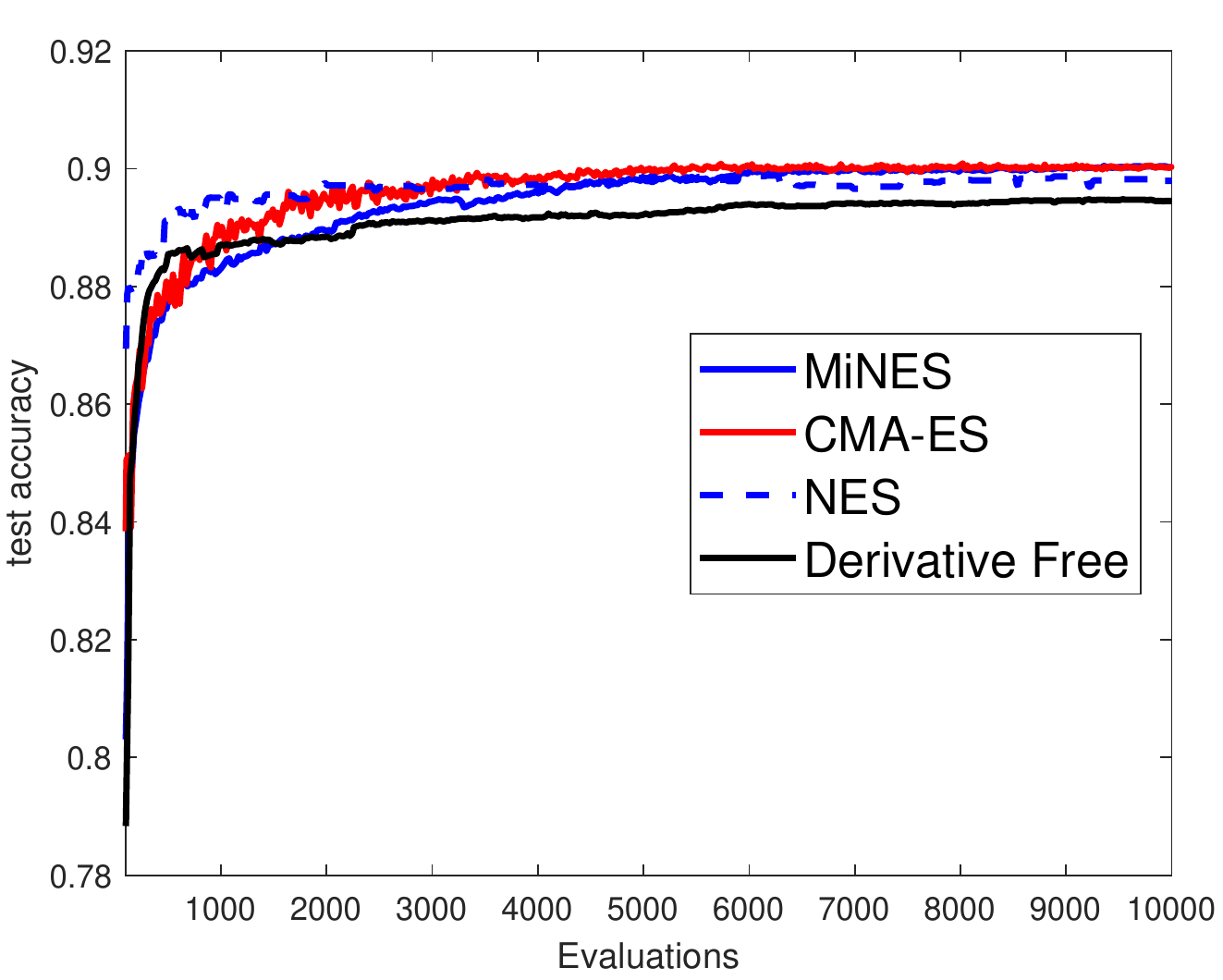}}\\
		\subfigure[`a1a' training]{\includegraphics[width=39mm]{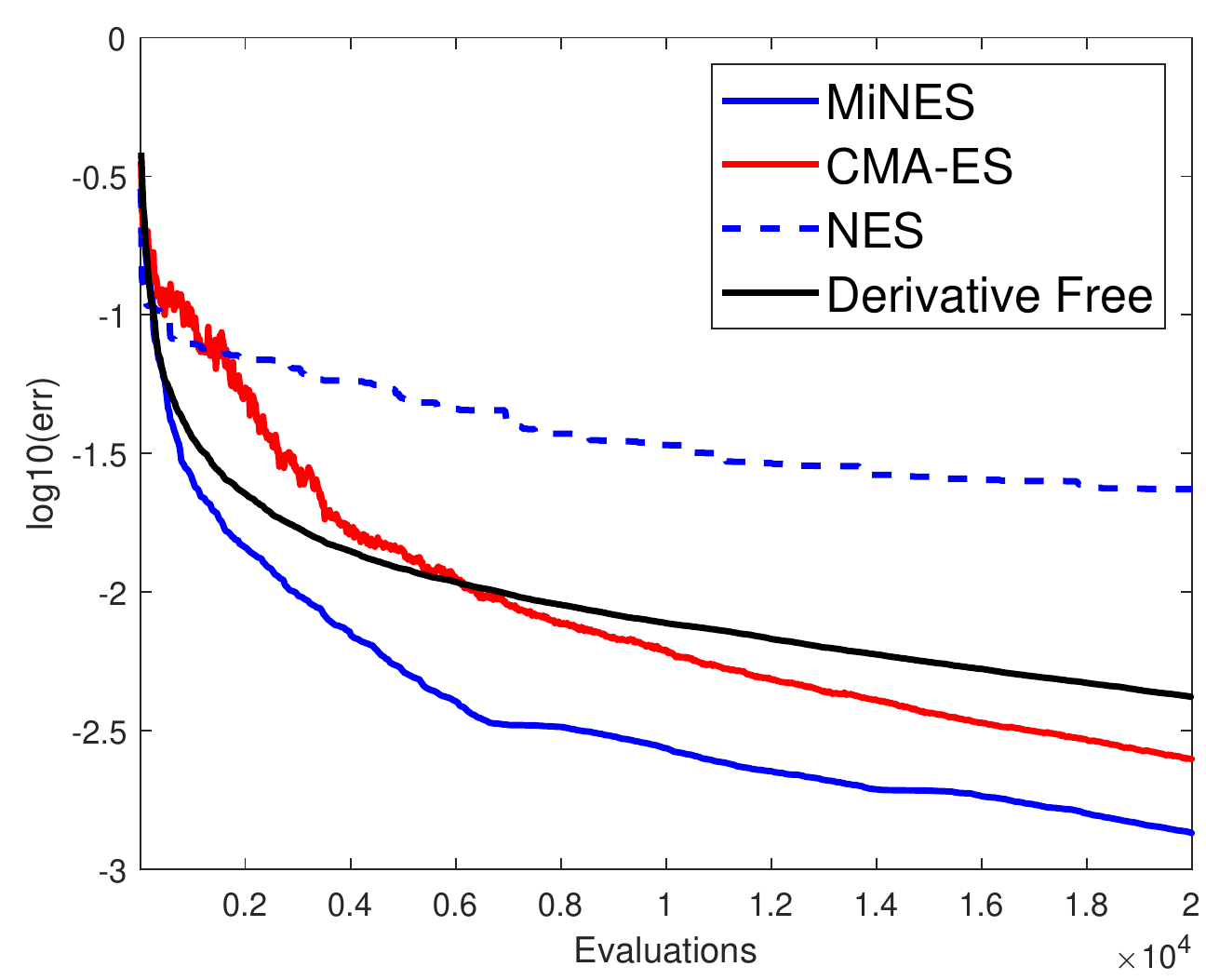}}~
		\subfigure[`a1a' test]{\includegraphics[width=39mm]{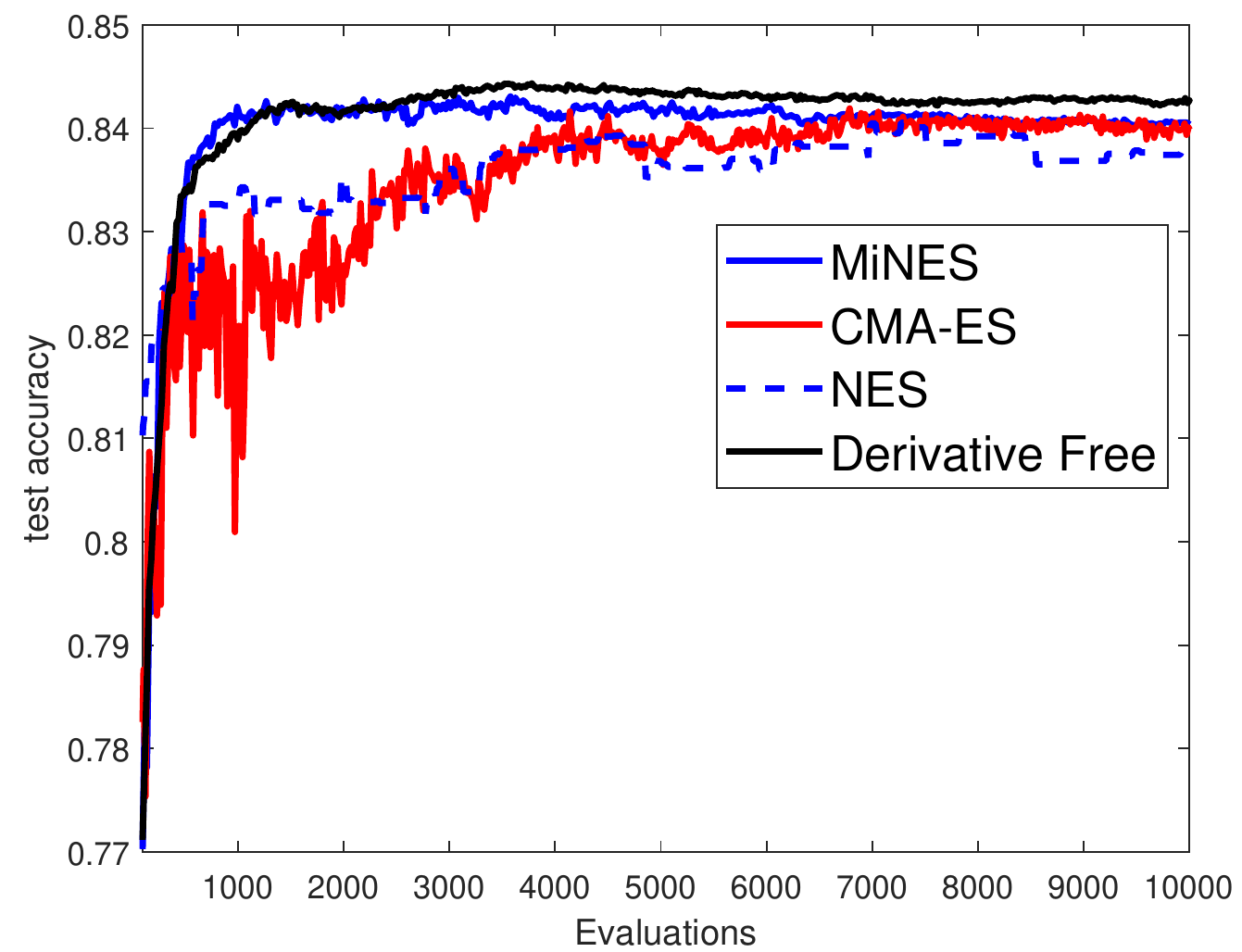}}~
		\subfigure[`ijcnn1' training]{\includegraphics[width=38mm]{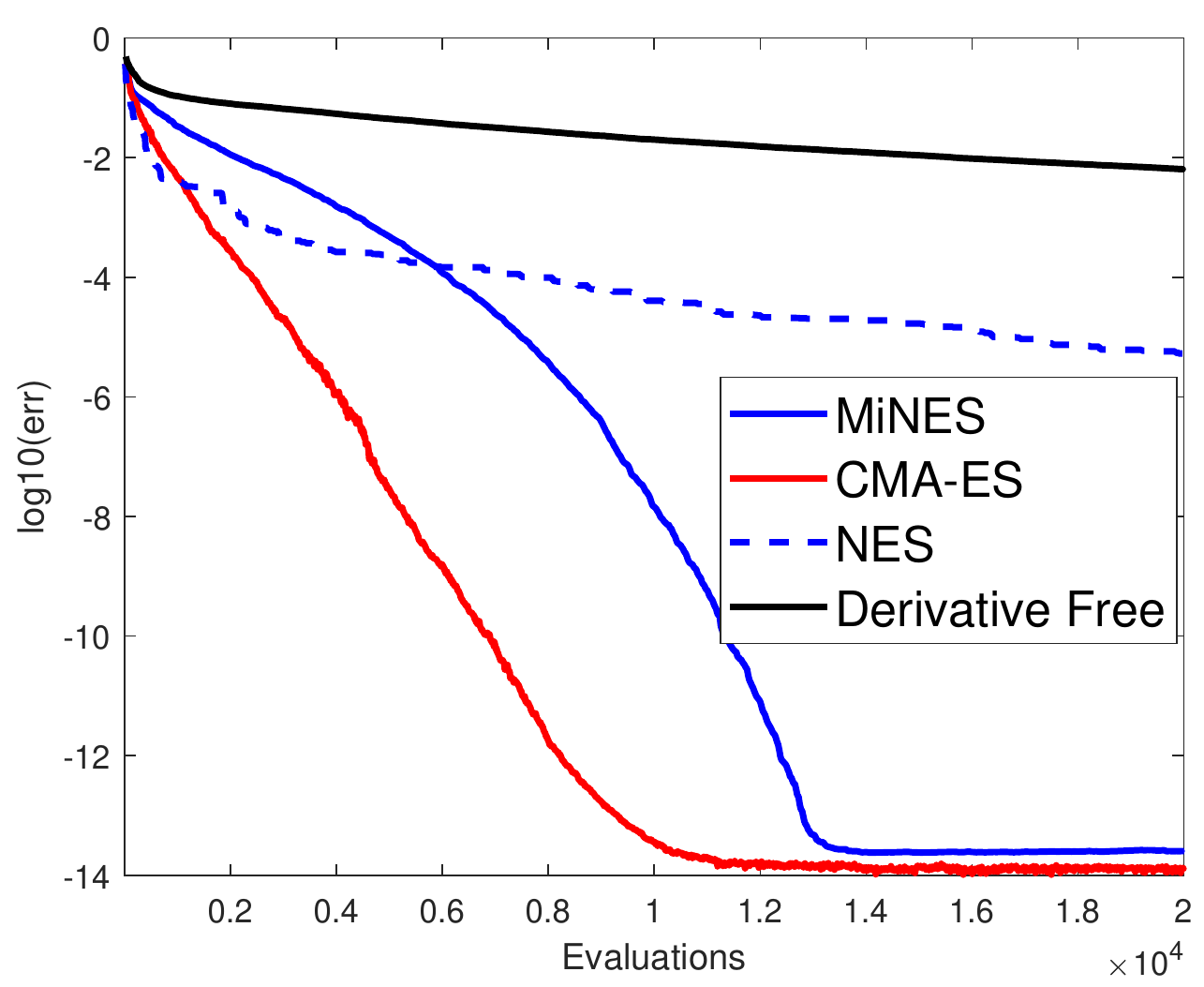}}~
		\subfigure[`ijcnn1' test]{\includegraphics[width=39mm]{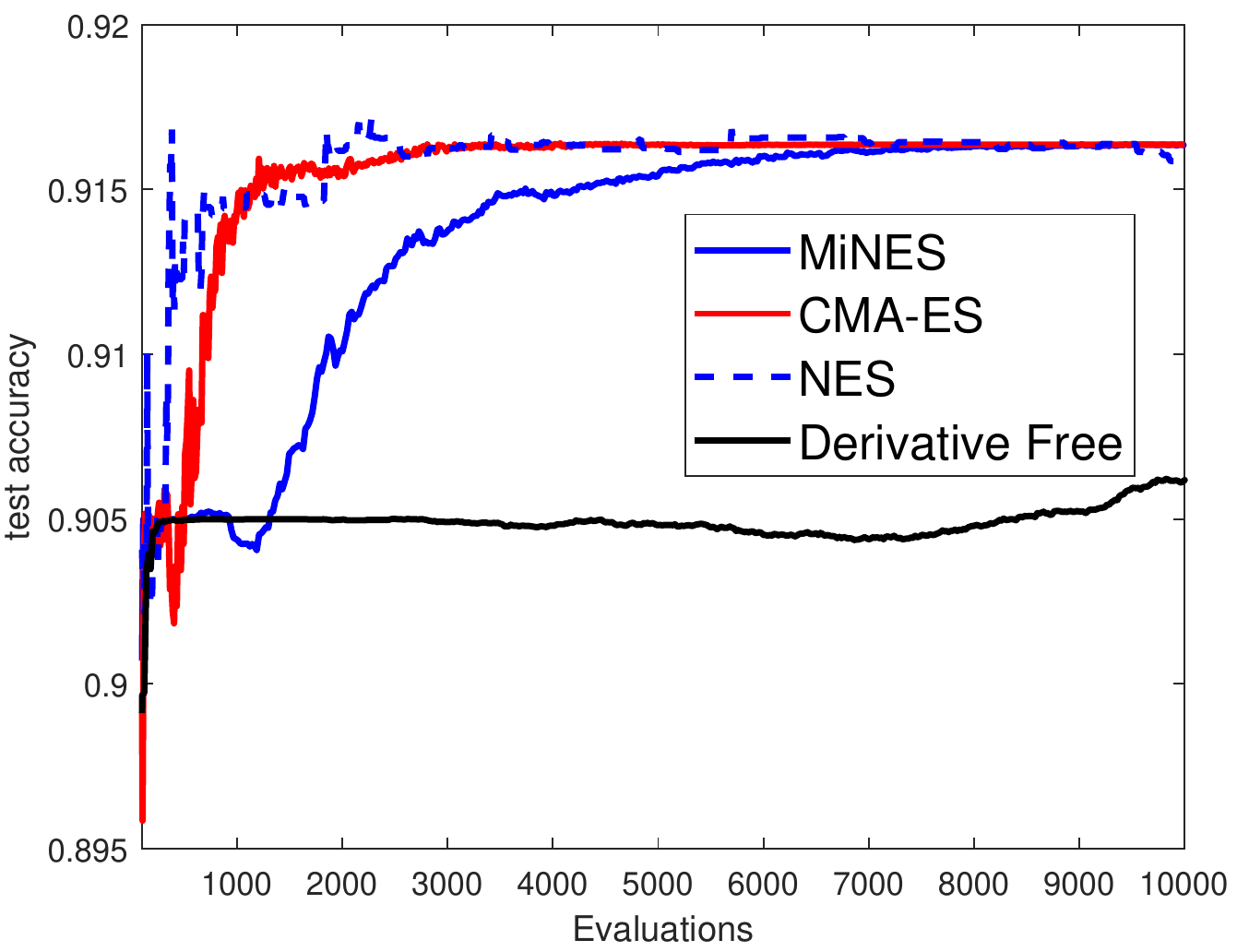}}
	\end{center}
	\caption{(a)-(b): The training loss and test accuracy of `mushrooms', respectively. (c)-(d): The training loss and test accuracy of `splice', respectively. (e)-(f): The training loss and test accuracy of `a9a', respectively. (g)-(h): The training loss and test accuracy of `w8a', respectively. (i)-(j): The training loss and test accuracy of `a1a', respectively. (k)-(l):  The training loss and test accuracy of `ijcnn1', respectively.}  
	\label{fig:logi}
\end{figure}

\subsection{Experiments on Logistic Regression}

In this section, we conduct experiments on logistic regression with a loss function 
$$f(x) = \frac{1}{n}\sum_{i=1}^{n}\log[1+\exp(-y_i\dotprod{a_i, x})] + \frac{\beta}{2}\norm{x}^2,$$ 
where $a_i\in\RR^d$ is the $i$-th input vector, and  $y_i \in\{-1, 1\}$ is the corresponding label. $\beta$ is the regularizer parameter. 
We conduct experiments on `mushrooms', `splice', `a9a', `w8a', `a1a', and `ijcnn1' which can be downloaded from libsvm datasets and the detailed description is listed in Table~\ref{tb:data}.  
In our experiments, we set $\beta = 0.0001$ for all datasets.  We set batch size $b = 10$ for \texttt{MiNES} and derivative free (\texttt{DF}) algorithm. We report the result in Figure~\ref{fig:logi}.

From Figure~\ref{fig:logi}, we can observe that \texttt{MiNES} converges faster than \texttt{DF} on all datasets. This shows that the Hessian information can effectively assist to improve the convergence rate of DF algorithm since \texttt{MiNES} exploits the Hessian information while \texttt{DF} only uses the first order information. 
Furthermore, we can also observe that \texttt{MiNES} achieves a fast convergence rate on the training loss comparable to \texttt{CMA-ES}. Moreover, for the test accuracies, \texttt{MiNES} commonly outperforms \texttt{CMA-ES}.

\section{Conclusion}
\label{sec:conclusion}

In this paper, we proposed a new kind of \texttt{NES} algorithm called \texttt{MiNES}. 
We showed that the covariance matrix of \texttt{MiNES} converges to the inverse of the Hessian, and we presented a rigorous convergence analysis of \texttt{MiNES}.
This result fills a gap that there was no rigorous convergence analysis of covariance matrix in previous works. 
Furthermore, \texttt{MiNES} can be viewed as an extension of the traditional first order derivative free algorithms in the optimization literature.
This clarifies the connection between \texttt{NES} algorithm and derivative free methods.
Our empirical study showed that \texttt{MiNES} is a query efficient algorithm that is competitive to other methods.

\bibliography{ref.bib}
\bibliographystyle{apalike2}

\cleardoublepage
\appendix
\section{Convexity of $\cS$ and $\cS'$}

First, we will show that $\cS$ and $\cS'$ are convex.
\begin{proposition}
	The sets $\cS$ and $\cS'$ are convex.
\end{proposition}
\begin{proof}
	Let $\Sigma_1$ and $\Sigma_2$ belong to $\cS$, then we have
	\begin{equation*}
	\frac{\Sigma_1}{2}
	+
	\frac{\Sigma_2}{2}
	\preceq
	\frac{1}{2}
	\left(
	\tau^{-1} I + \tau^{-1} I
	\right)
	=\tau^{-1} I.
	\end{equation*}
	Similarly, we have
	\begin{equation*}
	\frac{\Sigma_1}{2}
	+
	\frac{\Sigma_2}{2} 
	\succeq
	\zeta^{-1} I.
	\end{equation*}
	Therefore, $\cS$ is a convex set. 
	The convexity of $\cS'$ can be proved similarly.
\end{proof}

\begin{proposition}\label{prop:project}
	Let $A$ be a symmetric matrix. $A = U\Lambda U^\top$ is the spectral decomposition of $A$. The diagonal matrix $\bar{\Lambda}$ is defined as 
	\begin{equation*}
	\bar{\Lambda}_{i,i} = \left\{
	\begin{aligned}
	&\tau^{-1}    \qquad \mbox{if}\;\Lambda_{i,i}>\tau^{-1}\\
	&\zeta^{-1}  \qquad \mbox{if}\;\Lambda_{i,i}<\zeta^{-1}\\
	&\Lambda_{i,i}    \qquad \mbox{otherwise}
	\end{aligned}
	\right.
	\end{equation*}
	Let $\Pi_{\cS}(A)$ be the projection of symmetric $A$ on to $\cS$ defined in Eqn.~\eqref{eq:cS}, that is, $\Pi_{\cS}(A) = \argmin_{X\in\cS}\norm{A-x}$ with $\norm{\cdot}$ being Frobenius norm, then we have
	\begin{equation*}
	\Pi_{\mathcal{S}'}(A) = U\bar{\Lambda}U^\top.
	\end{equation*}
\end{proposition}
\begin{proof}
	We have the following Lagrangian \citep{lanckriet2004learning}
	\begin{equation}
	L(X, A_1, A_2) = \norm{A-X}^2 + 2\dotprod{A_1, X - \tau^{-1} I} + 2\dotprod{A_2, \zeta^{-1} I - X},
	\end{equation}
	where $A_1$ and $A_2$ are two positive semi-definite matrices. 
	The partial derivative $\partial L(X, A_1, A_2)/\partial{X}$ is 
	\begin{equation*}
	\frac{\partial L(X, A_1, A_2)}{\partial{X}} 
	= 
	2(X-A+A_1 - A_2).
	\end{equation*}
	By the general Karush-Kuhn-Tucker (KKT) condition \citep{lanckriet2004learning}, we have
	\begin{align*}
	&X_* = A-A_1+A_2 \\
	&A_1X_* = \tau^{-1} A_1, \quad A_2X_* = \zeta^{-1} A_2\\
	& A_1\succeq 0,\quad A_2\succeq 0.
	\end{align*}
	
	Since the optimization problem is strictly convex, there is a unique solution $(X_*,A_1,A_2)$ that satisfy the above KKT condition.
	Let $A = U \Lambda U^\top$ be the spectral decomposition of $A$. We construct $A_1$ and $A_2$ as follows:
	\begin{align}
	A_1 =& U\Lambda^{(1)}U^\top \quad\mbox{with}\quad \Lambda^{(1)}_{i,i} = \max\{\Lambda_{i,i} - \tau^{-1}, 0\} \label{eq:A1}\\ 
	A_2 =& U\Lambda^{(2)}U^\top \quad\mbox{with}\quad \Lambda^{(2)}_{i,i} = \max\{\zeta^{-1} - \Lambda_{i,i}, 0\}. \label{eq:A2}
	\end{align}
	$X_*$ is defined as $X_* = U\bar{\Lambda} U^\top$. 
	We can check that $A-A_1 + A_2 = U(\Lambda - \Lambda^{(1)} + \Lambda^{(2)}) U^\top = U\bar{\Lambda} U^\top  = X_*$. 
	The construction of $A_1$ and $A_2$ in Eqn.~\eqref{eq:A1},~\eqref{eq:A2} guarantees these two matrix are positive semi-definite.
	Furthermore, we can check that $A_1$ and $A_2$  satisfy $A_1 X_* = \tau^{-1} A_1$ and $A_2X_* =\zeta^{-1} A_2$. 
	Thus, $A_1$, $A_2$ and $X_*$ satisfy the KKT's condition which implies $X_* = U\bar{\Lambda} U^\top$ is the projection of $A$ onto $\cS$. 
\end{proof}

\section{Proof of Theorem~\ref{prop:quad_case}}

The proof of Theorem~\ref{prop:quad_case} is similar to that of Proposition~\ref{prop:project}.
\begin{proof}[Proof of Theorem~\ref{prop:quad_case}]
	By the definition of $Q_\alpha(\theta)$ and Eqn.~\eqref{eq:J_quad}, we have
	\begin{equation*}
	Q_\alpha(\theta) = f(\mu) + \frac{\alpha^2}{2}\dotprod{H, \Sigma} - \frac{\alpha^2}{2} \log\det\Sigma.
	\end{equation*}
	Then, taking partial derivative of $Q_\alpha$ with respect to $\mu$, we can obtain that 
	\[
	\frac{\partial Q_\alpha(\theta)}{\partial \mu} = \nabla_\mu f(\mu).
	\] 
	By setting $\frac{\partial Q_\alpha(\theta)}{\partial \mu}$ to zero, we can obtain that $Q_\alpha$ attains its minimum at $\mu_*$.
	
	For the $\Sigma$ part, we have the following Lagrangian \citep{lanckriet2004learning},
	\begin{equation*}
	L(\Sigma, A_1, A_2) 
	= 
	\frac{\alpha^2}{2}\dotprod{H, \Sigma} 
	- 
	\frac{\alpha^2}{2} \log\det\Sigma 
	+
	\frac{\alpha^2}{2}\dotprod{A_1, \Sigma - \tau^{-1}I}
	+
	\frac{\alpha^2}{2}\dotprod{A_2, \zeta^{-1}I - \Sigma},
	\end{equation*}  
	where $A_1$ and $A_2$ are two positive semi-definite matrices and $H$ denotes the Hessian matrix of the quadratic function $f(\cdot)$. 
	The $\partial L(\Sigma, A_1, A_2)/\partial{X}$ is
	\begin{equation*}
	\frac{\partial L(\Sigma, A_1, A_2)}{\partial{\Sigma}} 
	= 
	\frac{\alpha^2}{2}
	\left(
	H - \Sigma^{-1} + A_1 - A_2
	\right).
	\end{equation*}
	By the general Karush-Kuhn-Tucker (KKT) condition \citep{lanckriet2004learning}, we have
	\begin{align*}
	&\Sigma_*^{-1} = H+A_1-A_2 \\
	&A_1\Sigma_* = \tau^{-1}A_1, \quad A_2\Sigma_* = \zeta^{-1}A_2\\
	& A_1\succeq 0,\quad A_2\succeq 0.
	\end{align*}
	
	Since the optimization problem is strictly convex, there is a unique solution $(\Sigma_*,A_1,A_2)$ that satisfy the above KKT condition. We construct such a solution as follows. 
	Let $H = U \Lambda U^\top$ be the spectral decomposition of $H$, where $\Lambda$ is diagonal and $U$ is an orthogonal matrix. We define $\Sigma_*$ as $\Sigma_* = U \bar{\Lambda} U^\top$, where
	$\bar{\Lambda}$ is a diagonal matrix with $\bar{\Lambda}_{i,i} = \tau^{-1}$ if $\Lambda_{i,i} \leq \tau$, $\bar{\Lambda}_{i,i} = \zeta^{-1}$ if $\Lambda_{i,i} \geq \zeta$, and $\bar{\Lambda}_{i,i} = \Lambda_{i,i}^{-1}$ in other cases.
	$A_1$ and $A_2$ are defined as follows, where both $\Lambda^{(1)}_{i,i}$ and $\Lambda^{(2)}_{i,i}$ are diagonal matrices:
	\begin{align*}
	A_1 =& U\Lambda^{(1)}U^\top \quad\mbox{with}\quad \Lambda^{(1)}_{i,i} = \max\{\tau - \Lambda_{i,i}, 0\}  \\
	A_2 =& U\Lambda^{(2)}U^\top \quad\mbox{with}\quad \Lambda^{(2)}_{i,i} = \max\{\Lambda_{i,i} - \zeta, 0\}.
	\end{align*}
	
	Next, we will check that $A_1$, $A_2$ and $\Sigma_*$ satisfy the KKT's condition. First, we have
	\begin{align*}
	H+A_1-A_2 = U\Lambda U^\top + U\Lambda^{(1)}U^\top - U\Lambda^{(2)}U^\top = U \bar{\Lambda}^{-1} U^\top = \Sigma_*^{-1}.
	\end{align*}
	Then we also have $A_1\Sigma_* = U\Lambda^{(1)}U^\top U\bar{\Lambda}U^\top = U\Lambda^{(1)} \bar{\Lambda}U^\top = \tau^{-1}  U\Lambda^{(1)}U^\top = \tau^{-1} A_1$. 
	Similarly, it also holds that $A_2\Sigma_* = \zeta^{-1} A_2$.
	
	Finally, the construction of $A_1$ and $A_2$ guarantees these two matrix are positive semi-definite.
	
	Therefore, $\Sigma_*$ is the covariance part of the minimizer of $Q_\alpha(\theta)$, and $\left(\mu_*, \Pi_{\cS}\left(H^{-1}\right)\right)$ is the optimal solution of $Q_\alpha(\theta)$ under constraint $\cS$.
\end{proof}

\section{Proof of Theorem~\ref{prop:close}}
\begin{proof}[Proof of Theorem~\ref{prop:close}]
	By the Taylor's expansion of $f(z)$ at $\mu$, we have 
	\begin{equation}
	\left|
	f(z) -\left[ f(\mu) + \langle \nabla f(\mu),z-\mu \rangle + \frac12 (z-\mu)^\top \nabla^2 f(\mu) (z-\mu)
	\right] 
	\right| \leq \frac{\gamma}6 \| z-\mu\|_2^3 .
	\label{eq:taylor}
	\end{equation}
	By $z = \mu + \alpha \Sigma u$ with $u\sim N(0, I_d)$, we can upper bound the $J(\theta)$ as 
	\begin{align*}
	J(\theta) =& \frac{1}{(2\pi)^{d/2}}\int_u f(\mu + \alpha\Sigma^{1/2} u) \exp\left(-\frac{1}{2}\norm{u}^2\right) du \\
	\leq&\frac{1}{(2\pi)^{d/2}}\int_u \left(f(\mu) + \dotprod{\nabla f(\mu), \alpha \Sigma^{1/2}u} + \frac{\alpha^2}{2} u^\top \Sigma^{1/2} \nabla^2 f(\mu) \Sigma^{1/2} u + \frac{\gamma\alpha^3}{6}\|\Sigma^{1/2} u\|_2^3\right)\exp\left(-\frac{1}{2}\norm{u}^2\right) d u \\
	=&f(\mu) + \frac{\alpha^2}{2}\dotprod{\nabla^2f(\mu), \Sigma}+ \alpha^3\phi(\Sigma)\\
	\leq&f(\mu) + \frac{L\alpha^2}{2} \tr(\Sigma) + \alpha^3\phi(\Sigma),
	\end{align*}
	where the last inequality is because of $\|\nabla^2 f(\mu)\|_2 \leq L$, we have $|\tr(\nabla^2 f(\mu) \Sigma) | \leq L \tr(\Sigma)$ .
	
	By the fact that $Q_\alpha(\theta) = J(\theta) + R(\Sigma)$, we have
	\begin{equation*}
	Q_\alpha(\theta) \leq f(\mu) + \frac{L\alpha^2}{2} \tr(\Sigma) + \alpha^3\phi(\Sigma) + R(\Sigma).
	\end{equation*}
	
	Similarly, we can obtain that
	\begin{equation*}
	J(\theta) \geq f(\mu) - \frac{L\alpha^2}{2} \tr(\Sigma) -\alpha^3 \phi(\Sigma).
	\end{equation*}
	Therefore, we can obtain that 
	\begin{equation*}
	f(\mu) - \frac{L\alpha^2}{2} \tr(\Sigma) - \alpha^3\phi(\Sigma) + R(\Sigma)\leq Q_\alpha(\theta) .
	\end{equation*}
\end{proof}

\section{Proof of Theorem~\ref{prop:mu}}

\begin{proof}[Proof of Theorem~\ref{prop:mu}]
	Let $\HSig_*$ be $\Sigma$ part the minimizer $\hat{\theta}_*$ of $Q_\alpha$ under the  constraint $\Sigma\in\cS$.
	$\Sigma_*$ denotes the minimizer of $Q_\alpha$ given $\mu = \mu_* $ under the  constraint $\Sigma\in\cS$. Let us denote $\theta_* = (\mu_*, \Sigma_*)$.
	By Theorem~\ref{prop:close}, we have
	\begin{align*}
	f(\mu_*)+R(\Sigma_*)-\frac{L\alpha^2}{2}\tr(\Sigma_*)-\alpha^3\phi(\Sigma_*)\leq& Q_\alpha(\theta_*)   \leq   f(\mu_*)+\frac{L\alpha^2}{2}\tr(\Sigma_*)+\alpha^3\phi(\Sigma_*)+R(\Sigma_*)\\
	f(\hat{\mu}_*)
	+ R(\hat{\Sigma}_*)  
	- \frac{L\alpha^2}{2}\tr(\hat{\Sigma}_*)
	- \alpha^3\phi(\hat{\Sigma}_*)
	\leq&Q_\alpha(\hat{\theta}_*)
	\leq f(\hat{\mu}_*) 
	+ \frac{L\alpha^2}{2}\tr(\hat{\Sigma}_*)
	+ \alpha^3\phi(\hat{\Sigma}_*)
	+ R(\hat{\Sigma}_*)
	\end{align*}
	By the fact that $\hat{\theta}_*$ is the minimizer of $Q_\alpha(\theta)$ under constraint $\Sigma\in\cS$, then we have 
	\[
	Q_\alpha(\hat{\theta}_*) \leq Q_\alpha(\theta_*).
	\]
	Thus, we can obtain that 
	\begin{align*}
	&f(\Hu_*)+R(\HSig_*)-\frac{L\alpha^2}{2}\tr(\HSig_*)-\alpha^3\phi(\HSig_*) 
	\leq 
	f(\mu_*) 
	+ \frac{L\alpha^2}{2}\tr(\Sigma_*)
	+ \alpha^3\phi(\Sigma_*)
	+ R(\Sigma_*)\\
	\Rightarrow&
	f(\hat{\mu}_*)-f(\mu_*)
	\leq 	
	\frac{L\alpha^2}{2}\tr(\Sigma_* + \hat{\Sigma}_*) 
	+ 
	\alpha^3\left(\phi(\Sigma_*) + \phi(\hat{\Sigma}_*) \right)
	+
	R(\Sigma_*) - R(\hat{\Sigma}_*)
	\end{align*}
	Since $\mu_*$ is the solver of minimizing $f(\mu)$, we have
	\[
	0
	\leq 
	f(\hat{\mu}_*) - f(\mu_*)
	\] 
	Thus, we can obtain that 
	\begin{equation*}
	0
	\leq 
	f(\hat{\mu}_*) - f(\mu_*)   
	\leq
	\frac{L\alpha^2}{2}\tr(\Sigma_* + \hat{\Sigma}_*) 
	+ 
	\alpha^3\left(\phi(\Sigma_*) + \phi(\hat{\Sigma}_*) \right)
	+
	R(\Sigma_*) - R(\hat{\Sigma}_*). 
	\end{equation*}
	Next, we will bound the terms of right hand of above equation. First, we have
	\begin{equation}\label{eq:tr}
	\frac{L\alpha^2}{2}\tr(\Sigma_* + \hat{\Sigma}_*) 
	\leq
	\frac{dL\alpha^2}{2}(\lambda_{\max}(\Sigma_*)+\lambda_{\max}(\hat{\Sigma}_*) )
	\leq
	\frac{dL\alpha^2}{\tau} ,
	\end{equation}
	where the last inequality is because $\Sigma_*$ and $\hat{\Sigma}_*$ are in $\cS$. Then we bound the value of $\phi(\Sigma_*)$ as follows.
	\begin{align*}
	\phi(\Sigma_*) 
	\leq
	\frac{\gamma\norm{\Sigma_*^{1/2}}^3}{6} 
	\EE_u \left[\left(\norm{u}^4\right)^{3/4}\right]
	\leq 
	\frac{\gamma\norm{\Sigma_*^{1/2}}^3}{6} 
	\left(\EE_u \left[\norm{u}^4\right]\right)^{3/4}
	=
	\frac{\gamma (d^2+2d)^{3/4}}{6\tau^{3/2}},
	\end{align*}
	where the second inequality follows from Jensen's inequality. Similarly, we have
	\begin{equation*}
	\phi(\hat{\Sigma}_*) \leq  \frac{\gamma (d^2+2d)^{3/4}}{6\tau^{3/2}}.
	\end{equation*}
	Thus, we obtain that 
	\begin{equation}\label{eq:phi}
	\alpha^3\left(\phi(\Sigma_*) + \phi(\hat{\Sigma}_*) \right)
	\leq
	\frac{\gamma \alpha^3 (d^2+2d)^{3/4}}{3\tau^{3/2}}.
	\end{equation}
	Finally, we bound $R(\Sigma_*) - R(\hat{\Sigma}_*)$. 
	\begin{align*}
	R(\Sigma_*) - R(\hat{\Sigma}_*)
	=&\frac{\alpha^2}{2} 
	\left(
	-\log\det\Sigma_* 
	+ 
	\log\det\hat{\Sigma}_* 
	\right)\\
	\leq&
	\frac{d\alpha^2}{2}\left(\lambda_{\max}\left(\hat{\Sigma}_*\right) - \lambda_{\min}\left(\Sigma_*\right)\right)
	\\
	\leq&
	\frac{d\alpha^2}{2} \left(\tau^{-1} - \zeta^{-1} \right).
	\end{align*}
	Thus, we have
	\begin{equation}\label{eq:R}
	R(\Sigma_*) - R(\hat{\Sigma}_*) 
	\leq 
	\frac{d\alpha^2}{2} \left(\tau^{-1} - \zeta^{-1}  \right)
	\end{equation}
	Combining Eqn.~\eqref{eq:tr},~\eqref{eq:phi} and~\eqref{eq:R}, we obtain that
	\begin{equation}
	f(\hat{\mu}_*)  - f(\mu_*)
	\leq 
	\frac{dL\alpha^2}{\tau} 
	+ 
	\frac{\gamma \alpha^3 (d^2+2d)^{3/4}}{3\tau^{3/2}} 
	+
	\frac{d\alpha^2}{2} \left(\tau^{-1} - \zeta^{-1}  \right).
	\end{equation}
	By the property of strongly convex, we have
	\begin{align*}
	\norm{\mu_* - \hat{\mu}_*}^2 \leq \frac{2}{\sigma}\left( f(\mu_*) - f(\hat{\mu}_*)\right) 
	\leq 
	\frac{2dL\alpha^2}{\sigma\tau} 
	+ 
	\frac{2\gamma \alpha^3 (d^2+2d)^{3/4}}{3\sigma\tau^{3/2}} 
	+
	\frac{d\alpha^2}{\sigma} \left(\tau^{-1} - \zeta^{-1}  \right).
	\end{align*}
\end{proof}

\section{Proof of Theorem~\ref{prop:Sigma}}

First, we give the following property.
\begin{lemma}\label{lem:exp_tg}
	Let $u\sim N(0, I)$, $H$ be a positive semi-definite matrix, then we have
	\begin{equation*}
	\frac{1}{2}\cdot\EE_u\left(u^\top \Sigma^{1/2}H\Sigma^{1/2} u \cdot \left(\Sigma^{-1/2}uu^\top \Sigma^{-1/2} -  \Sigma^{-1}\right)\right)
	= H.
	\end{equation*}
\end{lemma}
\begin{proof}
	Let $J(\theta)$ be defined as Eqn.~\eqref{eq:J_org}. $f(\cdot)$ is a quadratic function with $H$ as its Hessian matrix. Then $J(\theta)$ can be represented as Eqn.~\eqref{eq:J_quad}. Therefore, we have
	\begin{equation*}
	\frac{\partial J(\theta)}{\Sigma} 
	= 
	\frac{\alpha^2}{2} H.
	\end{equation*}
	Let $\TG$ be defined as Eqn.~\eqref{eq:TG} with respect to the quadratic function $f(\cdot)$. $\TG$ can further reduce to 
	\begin{equation*}
	\EE_u\left[\TG\right] = 
	\EE_u\left[ \frac{1}{2} \left(u^\top \Sigma^{1/2}H\Sigma^{1/2} u \cdot \left(\Sigma^{-1/2}uu^\top \Sigma^{-1/2} -  \Sigma^{-1}\right)\right)\right] 
	- 
	\frac{1}{2} \Sigma^{-1} .
	\end{equation*}
	By Lemma~\ref{prop:TG}, we can obtain that 
	\begin{equation*}
	\EE_u\left[ \frac{1}{2} \left(u^\top \Sigma^{1/2}H\Sigma^{1/2} u \cdot \left(\Sigma^{-1/2}uu^\top \Sigma^{-1/2} -  \Sigma^{-1}\right)\right)\right] 
	= 2\alpha^{-2} \cdot
	\frac{\partial J(\theta)}{\Sigma} 
	= H.
	\end{equation*}
	This completes the proof.
\end{proof}
\begin{proof}[Proof of Theorem~\ref{prop:Sigma}]
	We will compute $\HSig_*$. First, we have the following Lagrangian 
	\begin{equation}
	L(\Sigma, A_1, A_2) 
	= 
	Q_\alpha(\theta) 
	+ 
	\frac{\alpha^2}{2}\dotprod{A_1,  \Sigma - \tau^{-1} I} 
	+ 
	\frac{\alpha^2}{2}\dotprod{A_2,  \zeta^{-1}I - \Sigma }, 
	\end{equation}
	where $A_1$ and $A_2$ are two positive semi-definite matrices. 
	Next, we will compute $\partial(L)/\partial\Sigma$ 
	\begin{equation}\label{eq:L_partial}
	\partial(L)/\partial\Sigma 
	= 
	\frac{\partial J(\theta)}{\partial\Sigma}
	-\frac{\alpha^2}{2}\Sigma^{-1} 
	+\frac{\alpha^2}{2}A_1
	-\frac{\alpha^2}{2}A_2.
	\end{equation}
	
	Furthermore, by Eqn.~\eqref{eq:log_lh}, ~\eqref{eq:nab_Sig} and $z = \Hu_*+\alpha\Sigma^{1/2}u$, we have
	\begin{align}
	\frac{\partial J(\theta)}{\partial \Sigma} =&\frac{\partial J(\theta)}{\partial \bar{\Sigma}} \cdot \frac{\partial \bar{\Sigma}}{\partial \Sigma} 
	\notag\\
	\overset{\eqref{eq:nab_Sig}}{=}& \EE_z\left[f(z) \left(\frac{1}{2}\Sigma^{-1}(z-\mu)(z-\mu)^\top\Sigma^{-1}\alpha^{-2} - \frac{1}{2} \Sigma^{-1}\alpha^{-2}\right)\right]\cdot\alpha^2
	\notag\\
	=&\frac{1}{2}\cdot\EE_u\left[f(\Hu_*+ \alpha \Sigma^{1/2}u) \left(\Sigma^{-1/2}uu^\top \Sigma^{-1/2} -  \Sigma^{-1}\right)\right] .
	\label{eq:J_Sig_0}
	\end{align}
	We can express $f(\Hu_*+\alpha\Sigma^{1/2}u)$ using Taylor expansion as follows:
	\begin{equation*}
	f(\Hu_*+\alpha\Sigma^{1/2}u) 
	= 
	f(\Hu_*) 
	+ 
	\dotprod{\nabla f(\Hu_*), \alpha\Sigma^{1/2}u} 
	+ 
	\frac{\alpha^2}{2}u^\top \Sigma^{1/2}\nabla^2f(\Hu_*)\Sigma^{1/2} u
	+ \tilde{\rho}\left(\alpha\Sigma^{1/2}u\right),
	\end{equation*}
	where $\tilde{\rho}\left(\alpha\Sigma^{1/2}u\right)$ satisfies that 
	\begin{equation}
	\label{eq:rho_t}
	|\tilde{\rho}\left(\alpha\Sigma^{1/2}u\right)| \leq \frac{\gamma\alpha^3\norm{\Sigma^{1/2}u}^3}{6},
	\end{equation}
	due to Eqn.~\eqref{eq:taylor}.
	By plugging the above Taylor expansion into Eqn.~\eqref{eq:J_Sig_0}, we obtain
	\begin{align*}
	\frac{\partial J(\theta)}{\partial \Sigma}
	=&
	\frac{1}{2}\cdot\EE_u\left[\left(\frac{\alpha^2}{2}u^\top \Sigma^{1/2}\nabla^2f(\Hu_*)\Sigma^{1/2} u
	+ \tilde{\rho}\left(\alpha\Sigma^{1/2}u\right)\right)
	\cdot
	\left(\Sigma^{-1/2}uu^\top \Sigma^{-1/2} -  \Sigma^{-1}\right)
	\right]
	\\
	=&\frac{\alpha^2}{2}\nabla^2 f(\Hu_*) 
	+ 
	\Phi(\Sigma),
	\end{align*}
	where the last equality uses Lemma~\ref{lem:exp_tg},
	and $\Phi(\Sigma)$ is defined as
	\begin{equation*}
	\Phi(\Sigma) 
	= 
	\frac12
	\cdot
	\EE_u
	\left[
	\tilde{\rho}\left(\alpha\Sigma^{1/2}u\right)
	\cdot
	\left(\Sigma^{-1/2}uu^\top \Sigma^{-1/2} -  \Sigma^{-1}\right)
	\right].
	\end{equation*}
	Replacing $\partial J(\theta)/\partial\Sigma$ to Eqn.~\eqref{eq:L_partial}, we have
	\begin{equation*}
	\frac{\partial L}{\partial \Sigma}
	=
	\frac{\alpha^2}{2}\nabla^2f(\Hu_*) 
	+
	\Phi(\Sigma) 
	-\frac{\alpha^2}{2}\Sigma^{-1} 
	+\frac{\alpha^2}{2}A_1
	-\frac{\alpha^2}{2}A_2.
	\end{equation*}
	By the KKT condition, we have
	\begin{align}
	&\HSig_* = \left(\nabla^2f(\Hu_*) + 2\alpha^{-2}\Phi(\HSig_*)  + A_1 - A_2\right)^{-1} \label{eq:Sig_opt}\\
	&A_1\HSig_* = \tau^{-1}A_1,\quad A_2\HSig_* = \zeta^{-1}A_2\notag\\
	&A_1\succeq 0,\quad A_2\succeq 0. \notag
	\end{align}
	
	Because the optimization problem is strictly convex, there is a unique solution $(\HSig_*,A_1,A_2)$ that satisfy the above KKT condition.
	Let $\nabla^2f(\Hu_*)  + 2\alpha^{-2}\Phi(\HSig_*) = U\Lambda U^\top$ be the spectral decomposition of $\nabla^2f(\Hu_*)+ 2\alpha^{-2}\Phi(\HSig_*)$, where $U$ is a orthonormal matrix and $\Lambda$ is a diagonal matrix,  then we construct $A_1$ and $A_2$ as follows
	\begin{align*}
	A_1 =& U\Lambda^{(1)}U^\top \quad\mbox{with}\quad \Lambda^{(1)}_{i,i} = \max\{\tau - \Lambda_{i,i}, 0\}\\
	A_2 =& U\Lambda^{(2)}U^\top \quad\mbox{with}\quad \Lambda^{(2)}_{i,i} = \max\{\Lambda_{i,i} - \zeta, 0\}.
	\end{align*}
	Substituting $A_1$ and $A_2$ in Eqn.~\eqref{eq:Sig_opt}, we can obtain that 
	\begin{equation}\label{eq:Sig_opt_1}
	\HSig_* 
	= 
	\Pi_{\cS}
	\left(
	\left(
	\nabla^2f(\mu_*) 
	+
	2\alpha^{-2}\Phi(\HSig_*)
	\right)^{-1}
	\right).
	\end{equation}
	Similar to the proof of Theorem~\ref{prop:project} and \ref{prop:mu}, we can check that $\HSig_*$, $A_1$ and $A_2$ satisfy the above KKT condition.
	
	Now we begin to bound the error between $\HSig_*$ and $\Pi_{\cS}\left(\left(\nabla^2f(\Hu_*) 
	\right)^{-1}\right)$. We have
	\begin{align}
	&\norm{\HSig_* - \Pi_{\cS}\left(\left(\nabla^2f(\Hu_*) 	
		\right)^{-1}\right)} \notag\\
	\overset{\eqref{eq:Sig_opt_1}}{=}&
	\norm{\Pi_{\cS}
		\left(
		\left(
		\nabla^2f(\Hu_*) 
		+
		2\alpha^{-2}\Phi(\HSig_*)
		\right)^{-1}
		\right) 
		- 
		\Pi_{\cS}
		\left(
		\left(
		\nabla^2f(\Hu_*) 	
		\right)^{-1}
		\right)} \notag\\
	\leq&
	\norm
	{
		\left(
		\nabla^2f(\Hu_*) 
		+
		2\alpha^{-2}\Phi(\HSig_*)
		\right)^{-1}
		-
		\left(
		\nabla^2f(\Hu_*) 	
		\right)^{-1}
	}\notag\\
	\leq&
	\norm{	\left(
		\nabla^2f(\Hu_*) 
		+
		2\alpha^{-2}\Phi(\HSig_*)
		\right)^{-1}
		\left(
		2\alpha^{-2}\Phi(\Sigma_*)
		\right)
		\left(
		\nabla^2f(\Hu_*) 	
		\right)^{-1}
	}\notag\\
	\leq&
	\norm
	{
		\left(
		\nabla^2f(\Hu_*) 
		+
		2\alpha^{-2}\Phi(\HSig_*)
		\right)^{-1}
	}_2
	\cdot
	\norm
	{
		\left(
		\nabla^2f(\Hu_*) 	
		\right)^{-1}
	}_2
	\cdot
	\norm
	{
		2\alpha^{-2}\Phi(\Sigma_*)
	}, \label{eq:tmp}
	\end{align}
	where $\norm{\cdot}_2$ is the spectral norm. 
	The first inequality is because the projection operator onto a convex set is non-expansive \citep{bertsekas2009convex}.  The second inequality used the following fact: for any two nonsingular matrices $A$ and $B$, it holds that
	\begin{equation}\label{eq:A_inv}
	A^{-1} - B^{-1} = A^{-1}\left(B - A\right)B^{-1}.
	\end{equation}
	The last inequality is because it holds that $\norm{AB} \leq \norm{A}_2\norm{B}$ for two any consistent matrices $A$ and $B$.
	
	Now we bound $2\alpha^{-2}\norm{\Phi(\HSig_*)}$ as follows
	\begin{align*}
	2\alpha^{-2}\norm{\Phi(\HSig_*)} 
	=& 
	\EE_u
	\left[
	\tilde{\rho}\left(\alpha\HSig_*^{1/2}u\right)
	\cdot
	\left(\HSig_*^{-1/2}uu^\top \HSig_*^{-1/2} -  \HSig_*^{-1}\right)
	\right]
	\\
	\overset{\eqref{eq:rho_t}}{\leq}&
	\frac{\alpha\gamma\norm{\HSig_*}_2^{3/2}}{6}
	\EE_u
	\left[
	\norm{u}^3
	\cdot
	\norm{
		\left(\HSig_*^{-1/2}uu^\top \HSig_*^{-1/2} -  \HSig_*^{-1}\right)
	}
	\right]\\
	\leq&
	\frac{\alpha\gamma\norm{\HSig_*}_2^{3/2}\cdot \norm{\HSig_*^{-1/2}}_2^2}{6}
	\EE_u
	\left[
	\norm{u}^3\left(\norm{u}^2 + d\right)
	\right]\\
	\leq& \frac{\alpha\gamma\zeta}{6\tau^{3/2}}
	\EE_u
	\left[\norm{u}^5 + d\norm{u}^3\right],
	\end{align*}
	where the last inequality follows from the fact that $\HSig_*$ is in the convex set $\cS$.
	
	Furthermore, we have
	\begin{align*}
	\EE_u\left[\norm{u}^5\right] 
	=
	\EE_u\left[\left(\norm{u}^6\right)^{5/6}\right] 
	\leq\left(\EE_u\left[\norm{u}^6\right]\right)^{5/6}
	=
	\left(
	d^3
	+6d^2
	+8d
	\right)^{5/6},
	\end{align*}
	where the first inequality is because of Jensen's inequality and last equality follows from Lemma~\ref{lem:gauss_moment}.
	Similarly, we have
	\begin{equation*}
	\EE_u\left[\norm{u}^3\right]
	\leq
	\left(\EE_u\left[\norm{u}^4\right]\right)^{3/4}
	= \left(d^2+2d\right)^{3/4}.
	\end{equation*}
	Therefore, we have 
	\begin{equation}
	\label{eq:tmp2}
	2\alpha^{-2}\norm{\Phi(\Sigma_*)}  
	\leq 
	\frac{\alpha\gamma\zeta}{6\tau^{3/2}}
	\EE_u
	\left[\norm{u}^5 + d\norm{u}^3\right]
	\leq
	\frac{\alpha\gamma\zeta}{6\tau^{3/2}}
	\left(
	\left(
	d^3
	+6d^2
	+8d
	\right)^{5/6}
	+
	d\left(d^2+2d\right)^{3/4}
	\right).
	\end{equation}
	By the condition 
	\begin{equation*}
	\alpha
	\leq
	\frac{3\tau^{3/2}\sigma}{\gamma\zeta}
	\cdot
	\left(
	\left(
	d^3
	+6d^2
	+8d
	\right)^{5/6}
	+
	d\left(d^2+2d\right)^{3/4}
	\right)^{-1},
	\end{equation*}
	we have
	\begin{equation*}
	2\alpha^{-2}\norm{\Phi(\HSig_*)}  
	\leq 
	\frac{\sigma}{2},
	\end{equation*}
	which implies 
	\begin{equation}
	\label{eq:tmp1}
	\norm
	{
		\left(
		\nabla^2f(\Hu_*) 
		+
		2\alpha^{-2}\Phi(\HSig_*)
		\right)^{-1}
	}_2
	\leq
	2\sigma^{-1}.
	\end{equation}
	Consequently, we can obtain that 
	\begin{align*}
	&\norm
	{
		\HSig_* - \Pi_{\cS}\left(\left(\nabla^2f(\Hu_*) 	
		\right)^{-1}\right)
	}
	\\
	\overset{\eqref{eq:tmp}}{\leq}&
	\norm
	{
		\left(
		\nabla^2f(\Hu_*) 
		+
		2\alpha^{-2}\Phi(\Sigma_*)
		\right)^{-1}
	}_2
	\cdot
	\norm
	{
		\left(
		\nabla^2f(\Hu_*) 	
		\right)^{-1}
	}_2
	\cdot
	\norm
	{
		2\alpha^{-2}\Phi(\Sigma_*)
	}
	\\
	\overset{\eqref{eq:tmp1}}{\leq}&
	2\sigma^{-1}
	\cdot
	\sigma^{-1}
	\cdot
	\norm
	{
		2\alpha^{-2}\Phi(\Sigma_*)
	}
	\\
	\overset{\eqref{eq:tmp2}}{\leq}&
	2\sigma^{-1}
	\cdot
	\sigma^{-1}
	\cdot
	\frac{\alpha\gamma\zeta}{6\tau^{3/2}}
	\left(
	\left(
	d^3
	+6d^2
	+8d
	\right)^{5/6}
	+
	d\left(d^2+2d\right)^{3/4}
	\right)
	\\
	=&
	\frac{\alpha\gamma\zeta}{3\tau^{3/2}\sigma^2}
	\cdot
	\left(
	\left(
	d^3
	+6d^2
	+8d
	\right)^{5/6}
	+
	d\left(d^2+2d\right)^{3/4}
	\right)
	.
	\end{align*}
	
	Similarly, we have
	\begin{equation*}
	\norm
	{
		\Sigma_* - \Pi_{\cS}\left(\left(\nabla^2f(\mu_*) 	
		\right)^{-1}\right)
	}
	\leq 
	\frac{\alpha\gamma\zeta}{3\tau^{3/2}\sigma^2}
	\cdot
	\left(
	\left(
	d^3
	+6d^2
	+8d
	\right)^{5/6}
	+
	d\left(d^2+2d\right)^{3/4}
	\right).
	\end{equation*}
	Next, we will bound $\norm{\Sigma_* - \hat{\Sigma}_*}$ as follows
	\begin{align*}
	\norm{\Sigma_* - \hat{\Sigma}_*}
	\leq&
	\norm
	{
		\Pi_{\cS}\left(\left(\nabla^2f(\mu_*)  \right)^{-1}\right)
		-
		\Pi_{\cS}\left(\left(\nabla^2f(\hat{\mu}_*) 	
		\right)^{-1}\right)
	}
	\\
	&+
	\frac{2\alpha\gamma\zeta}{3\tau^{3/2}\sigma^2}
	\cdot
	\left(
	\left(
	d^3
	+6d^2
	+8d
	\right)^{5/6}
	+
	d\left(d^2+2d\right)^{3/4}
	\right)
	\end{align*}
	We also have
	\begin{align*}
	&\norm
	{
		\Pi_{\cS}\left(\left(\nabla^2f(\mu_*)\right)^{-1}\right)
		-
		\Pi_{\cS}\left(\left(\nabla^2f(\hat{\mu}_*) 	
		\right)^{-1}\right)
	}
	\\
	\leq&
	\norm{
		\left(\nabla^2f(\mu_*) \right)^{-1}
		-
		\left(\nabla^2f(\hat{\mu}_*) \right)^{-1}
	}
	\\
	\overset{\eqref{eq:A_inv}}{\leq}&
	\norm{
		\left(\nabla^2f(\mu_*)\right)^{-1}
	}
	\cdot
	\norm{
		\left(\nabla^2f(\hat{\mu}_*) \right)^{-1}
	}
	\cdot
	\norm{
		\nabla^2f(\mu_*) - \nabla^2f(\hat{\mu}_*)
	}
	\\
	\leq&
	\frac{\gamma}{\sigma^2}
	\norm{
		\mu_* - \hat{\mu}_*
	}
	\\
	\leq&
	\frac{\gamma}{\sigma^2}
	\cdot
	\left(
	\frac{2dL\alpha^2}{\sigma\tau} 
	+ 
	\frac{2\gamma \alpha^3 (d^2+2d)^{3/4}}{3\sigma\tau^{3/2}} 
	+
	\frac{d\alpha^2}{\sigma} \left(\tau^{-1} - \zeta^{-1}  \right)
	\right)^{1/2}.
	\end{align*}
	The first inequality is because of the property that projection operator onto a convex set is non-expansive \citep{bertsekas2009convex}.
	The third inequality is due to $f(\cdot)$ is $\sigma$-strongly convex and $\nabla^2f(\mu)$ is $\gamma$-Lipschitz continuous. 
	The last inequality follows from Theorem~\ref{prop:mu}.
	
	Therefore, we can obtain that 
	\begin{align*}
	\norm{\Sigma_* - \hat{\Sigma}_*}
	\leq&
	\frac{2\alpha\gamma\zeta}{3\tau^{3/2}\sigma^2}
	\cdot
	\left(
	\left(
	d^3
	+6d^2
	+8d
	\right)^{5/6}
	+
	d\left(d^2+2d\right)^{3/4}
	\right)
	\\
	&+
	\frac{\gamma}{\sigma^2}
	\cdot
	\left(
	\frac{2dL\alpha^2}{\sigma\tau} 
	+ 
	\frac{2\gamma \alpha^3 (d^2+2d)^{3/4}}{3\sigma\tau^{3/2}} 
	+
	\frac{d\alpha^2}{\sigma} \left(\tau^{-1} - \zeta^{-1}  \right)
	\right)^{1/2}.
	\end{align*}
\end{proof}

\section{Proof of Theorem~\ref{thm:Sig_conv}}

Since $\cS$ is convex, this implies that $\cS'$ is convex. 
Then we will have the follow properties.
\begin{lemma}\label{lem:proj}
	Let $\Sigma_{k+1}^{-1}$ be the projection of $\Sigma_{k+0.5}^{-1}$ onto a convex set $\cS'$, then we have
	\begin{align*}
	\norm{\Sigma_{k+1}^{-1} - \Pi_{\cS'}\left(\nabla^2f(\mu)\right)}^2 \leq \norm{\Sigma_{k+0.5}^{-1} - \Pi_{\cS'}\left(\nabla^2f(\mu)\right)}^2
	\end{align*}
\end{lemma}
\begin{proof}
	First, $\mathcal{S}'$ is convex and  $\Sigma_{k+1}^{-1}$ is the projection of $\Sigma_{k+0.5}^{-1}$  to $\mathcal{S}'$. 
	Furthermore, $\Pi_{\cS'}\left(\nabla^2f(\mu)\right)$ is the projection of $\nabla^2f(\mu)$ onto $\mathcal{S}'$. 
	Since the projection operator onto a convex set is non-expansive \citep{bertsekas2009convex}, we can obtain the result.
\end{proof}

Because we only consider the case that function $f(\cdot)$ is quadratic, then we can have a reduced form of $\TG(\Sigma)$.
\begin{lemma}\label{lem:quad_TG}
	Let $f(\cdot)$ is a quadratic function with  Hessian matrix being $H$, then $\TG(\Sigma_k)$ can be represented as 
	\begin{equation*}
	\TG(\Sigma_k) = \frac{1}{2b}\sum_{i=1}^{b}u_i^\top\Sigma_k^{1/2}H\Sigma_k^{1/2}u_i \cdot\left(\Sigma_k^{-1/2}u_iu_i^\top \Sigma_k^{-1/2} -  \Sigma_k^{-1}\right) - \left(\Sigma_k^{-1}\right).
	\end{equation*}
\end{lemma}
\begin{proof}
	Since $f(\cdot)$ is a quadratic function with  Hessian matrix being $H$, then we have
	\begin{equation*}
	f(\mu_k - \alpha \Sigma_k^{1/2}u_i) + f(\mu_k + \alpha \Sigma_k^{1/2}u_i) - 2f(\mu_k) = \alpha^2\cdot u_i^\top\Sigma_k^{1/2}H\Sigma_k^{1/2}u_i.
	\end{equation*}
	Therefore, we obtain that
	\begin{equation*}
	\TG(\Sigma_k) = \frac{1}{2b}\sum_{i=1}^{b}u_i^\top\Sigma_k^{1/2}H\Sigma_k^{1/2}u_i \cdot\left(\Sigma_k^{-1/2}u_iu_i^\top \Sigma_k^{-1/2} -  \Sigma_k^{-1}\right) - \left(\Sigma_k^{-1} \right).
	\end{equation*}
\end{proof}

Note that $\TG(\Sigma_k)$ is an unbiased estimation of $\partial Q_\alpha/\partial \Sigma$ up to a constant $2\alpha^{-2}$. 
We can view $\TG(\Sigma_k)$ is a kind of stochastic gradient.  
Thus, to analysis the convergence rate of $\Sigma$, we need to bound the the variance of $\TG(\Sigma_k)$. Before we give the variance of $\TG(\Sigma_k)$, we introduce a lemma to describe the property of moments of Gaussian distributions. 

\begin{lemma}[\cite{magnus1978moments}]\label{lem:gauss_moment}
	The $s$-th moment $\beta_s = \EE\left(u^\top Au\right)^s$ where $A$ is a positive definite matrix and $u\sim N(0,I)$ satisfies 
	\begin{align*}
	\beta_1 =& \tr(A)\\
	\beta_2 =& \left(\tr(A)\right)^2 + 2\tr(A^2)\\
	\beta_3 =& \left(\tr(A)\right)^3 + 6(\tr(A))(\tr(A^2)) + 8\tr(A^3)\\
	\beta_4 =& \left(\tr(A)\right)^4 + 32(\tr(A))(\tr(A^3))+ 12(\tr(A^2))^2 + 12\left(\tr(A)\right)^2\left(\tr (A^2)\right) + 48\tr(A^4)
	\end{align*}
\end{lemma}

By the moments of Gaussian distribution, we can bound the variance of $\TG$ as follows.
\begin{lemma}\label{lem:TG}
	Let function $f(\cdot)$ be quadratic with Hessian matrix $H$. $f(\cdot)$ is also $L$-smooth and $\sigma$-strongly convex. Then we have
	\begin{equation*}
	\EE\left[\norm{\TG}^2\right] \leq \frac{L^2\zeta^2}{4\tau^2}\left(d^4+11d^3+34d^2+32d\right) +2d\zeta^2 +  \norm{H}^2.
	\end{equation*}
\end{lemma}
\begin{proof}
	For convenience, let us denote
	\begin{equation}\label{eq:BG}
	\BG = \frac{1}{2b}\sum_{i=1}^{b}u_i^\top\Sigma_k^{1/2}H\Sigma_k^{1/2}u_i \cdot\left(\Sigma_k^{-1/2}u_iu_i^\top \Sigma_k^{-1/2}-\Sigma_k^{-1}\right).
	\end{equation}
	By Lemma~\ref{prop:TG},  it is easy to check that
	$\EE\left[\BG\right] = H$.
	Then we have
	\begin{align*}
	\EE_u\dotprod{\TG,\TG} =& \EE_u\dotprod{\BG - \Sigma_k^{-1} , \BG - \Sigma_k^{-1} }\\
	=&\EE_u\dotprod{\BG,\BG} - 2\EE_u\dotprod{\BG, \Sigma_k^{-1} }+\norm{\Sigma_k^{-1}}^2\\
	=&\EE_u\dotprod{\BG,\BG} - 2\dotprod{H, \Sigma_k^{-1}}+\norm{\Sigma_k^{-1}}^2\\
	=&\EE_u\dotprod{\BG,\BG} - \norm{H}^2 + \norm{\Sigma_{k}^{-1} - H}^2 .
	\end{align*}
	Next we will bound $\EE_u\dotprod{\BG,\BG}$. First, we have
	\begin{align*}
	\EE_u\dotprod{\BG,\BG}=&\frac{1}{4}\EE_u\left[\left(u^\top \Sigma_k^{1/2}H\Sigma_k^{1/2}u\right)^2\cdot\norm{\Sigma_k^{-1/2}uu^\top\Sigma_k^{-1/2} - \Sigma_k^{-1}}^2\right]\\
	\leq&\frac{\norm{\Sigma_k^{-1/2}}_2^4 \norm{\Sigma_k^{1/2}H\Sigma_k^{1/2}}_2^2}{4}\EE_u\left[\norm{u}^4\cdot\norm{uu^\top - I}^2\right]\\
	=&\frac{\norm{\Sigma_k^{-1/2}}_2^4 \norm{\Sigma_k^{1/2}H\Sigma_k^{1/2}}_2^2}{4} \EE_u\left[\norm{u}^8-2\norm{u}^6+d\norm{u}^4\right] .
	\end{align*} 
	Since $\Sigma_{k}\in\cS'$, we can bound $\norm{\Sigma_k}^{-1/2}$ and $\norm{\Sigma_{k}}$ as 
	\begin{equation*}
	\norm{\Sigma_k^{-1/2}}_2^4 \leq \zeta^2,\quad\mbox{and}\quad \norm{\Sigma_k} \leq \tau^{-1}.
	\end{equation*}
	Combining with the assumption that $f(\cdot)$ is $L$-smooth, we have
	\begin{equation*}
	\norm{\Sigma_k^{1/2}H\Sigma_k^{1/2}}_2^2 \leq L^2\norm{\Sigma_k}_2^2 \leq \frac{L^2}{\tau^2}.
	\end{equation*}

	By Lemma~\ref{lem:gauss_moment}, we have
	\begin{align*}
	\EE_u\left[\norm{u}^8-2\norm{u}^6+d\norm{u}^4\right] = d^4+11d^3+34d^2+32d.
	\end{align*}
	Therefore, we can obtain that 
	\begin{equation*}
	\EE_u\dotprod{\BG,\BG} \leq \frac{L^2\zeta^2}{4\sigma^2}\left(d^4+11d^3+34d^2+32d\right).
	\end{equation*}
	We also have
	\begin{align*}
	\norm{\Sigma_{k}^{-1} - H}^2 \leq 2\norm{\Sigma_{k}^{-1}}^2 + 2\norm{H}^2\leq 2\zeta^2 + 2\norm{H}^2.
	\end{align*}
	Combining the above inequalities, we can obtain the result.
\end{proof}

\begin{lemma}\label{lem:update}
	Assume that $\Sigma_{k+1}$ is updated as in \texttt{MiNES}. We have
	\[
	\EE\left[\norm{\Sigma_{k+1}^{-1}-\Pi_{\cS'}\left(\nabla^2f(\mu)\right)}^2\right] \leq (1-2\eta_2)  \norm{\Sigma_k^{-1} - \Pi_{\cS'}\left(\nabla^2f(\mu)\right)}^2  +  \eta_2^2\cdot\EE\left[\norm{\TG}^2\right].
	\]
\end{lemma}
\begin{proof}
	First, by Lemma~\ref{lem:proj}, we have
	\begin{equation*}
	\EE\norm{\Sigma_{k+1}^{-1} - \Pi_{\cS'}\left(\nabla^2f(\mu)\right)}^2 \leq \EE \norm{\Sigma_{k+0.5}^{-1} - \Pi_{\cS'}\left(\nabla^2f(\mu)\right)}^2.
	\end{equation*}
	By the update rule of $\Sigma$, we have
	\begin{align*}
	&\EE\left[\norm{\Sigma_{k+0.5}^{-1}-\Pi_{\cS'}\left(\nabla^2f(\mu)\right)}^2\right]\\
	=& \EE\left[ \norm{\Sigma_k^{-1} + \eta_2\TG - \Pi_{\cS'}\left(\nabla^2f(\mu)\right)}^2 \right]\\
	=&\EE\left[\norm{\Sigma_k^{-1} - \Pi_{\cS'}\left(\nabla^2f(\mu)\right)}^2 + 2\eta_2\dotprod{\TG, \Sigma_k^{-1} - \Pi_{\cS'}\left(\nabla^2f(\mu)\right)} +\eta_2^2\norm{\TG}^2\right]\\
	=& \norm{\Sigma_k^{-1} - \Pi_{\cS'}\left(\nabla^2f(\mu)\right)}^2
	-2\eta_2\dotprod{\Sigma_k^{-1} - (\nabla^2 f(\mu) ), \Sigma_k^{-1} - \Pi_{\cS'}\left(\nabla^2f(\mu)\right)}
	+ \eta_2^2\cdot\EE\left[\norm{\TG}^2\right]
	\end{align*}
	where the last equation is because $\TG$ is unbiased estimation of $\partial Q_\alpha/\partial \Sigma$ up to a constant $2\alpha^{-2}$. That is,
	\begin{equation*}
	\EE\left[\TG(\Sigma_k)\right] = \nabla^2 f(\mu) -\Sigma_k^{-1}.
	\end{equation*}
	Furthermore, we have
	\begin{align*}
	&\dotprod{\Sigma_k^{-1} - \nabla^2 f(\mu) , \Sigma_k^{-1} - \Pi_{\cS'}\left(\nabla^2f(\mu) \right)}\\
	=&\norm{\Sigma_k^{-1} - \Pi_{\cS'}\left(\nabla^2f(\mu) \right)}^2- \dotprod{\nabla^2 f(\mu)  - \Pi_{\cS'}\left(\nabla^2f(\mu)\right), \Sigma_k^{-1} - \Pi_{\cS'}\left(\nabla^2f(\mu)\right)}\\
	\geq&\norm{\Sigma_k^{-1} - \Pi_{\cS'}\left(\nabla^2f(\mu)\right)}^2,
	\end{align*}
	where the last inequality follows from the following properties of projection onto convex set \citep{bertsekas2009convex}:
	\begin{equation*}
	\dotprod{\nabla^2 f(\mu)  - \Pi_{\cS'}\left(\nabla^2f(\mu)\right), \Sigma_k^{-1} - \Pi_{\cS'}\left(\nabla^2f(\mu)\right)} \leq 0.
	\end{equation*}
	By combining the above equations, we  complete the proof as follows:
	\begin{align*}
	&\EE\left[\norm{\Sigma_{k+1}^{-1}-\Pi_{\cS'}\left(\nabla^2f(\mu)\right)}^2\right]\\
	\leq& \norm{\Sigma_k^{-1} - \Pi_{\cS'}\left(\nabla^2f(\mu)\right)}^2 
	-2\eta_2 \norm{\Sigma_k^{-1} - \Pi_{\cS'}\left(\nabla^2f(\mu)\right)}^2  
	+ \eta_2^2\cdot\EE\left[\norm{\TG}^2\right]\\
	=&(1-2\eta_2)  \norm{\Sigma_k^{-1} - \Pi_{\cS'}\left(\nabla^2f(\mu)\right)}^2  +  \eta_2^2\cdot\EE\left[\norm{\TG}^2\right] .
	\end{align*}
	
\end{proof}

We can now prove Theorem~\ref{thm:Sig_conv}.
\begin{proof}[Proof of Theorem~\ref{thm:Sig_conv}]
	We will prove the convergence rate by induction. First, it is easy to see that
	\begin{equation*}
	\norm{\Sigma_1 - \Pi_{\cS'}\left(H\right)}^2 \leq \frac{\max\{\norm{\Sigma_1 - \Pi_{\cS'}\left(H\right)}^2, M\}}{1}.
	\end{equation*}
	Then we assume that the convergence rate holds with $k$. 
	Next we only need to show that it holds with $k + 1$.
	Denote $R = \max \{\norm{\Sigma_1 -\Pi_{\cS'}\left(H \right)}^2, M\}$ . By Lemma~\ref{lem:update}, Lemma~\ref{lem:TG} and $\eta_2^{(k)} = \frac{1}{k}$, we have
	\begin{align*}
	\EE\left[\norm{\Sigma_{k+1}^{-1}-\Pi_{\cS'}\left(H \right)}^2\right] \leq& \left(1-\frac{2}{k}\right)\EE\left[ \norm{\Sigma_k^{-1} -\Pi_{\cS'}\left(H\right)}^2\right] + \frac{1}{k^2} \EE\norm{\TG(\Sigma_k)}^2\\
	\leq& \left(1-\frac{2}{k}\right)\EE\left[ \norm{\Sigma_k^{-1} - \Pi_{\cS'}\left(H\right)}^2\right] + \frac{M}{k^2} \\
	\leq&\left(1-\frac{2}{k}\right)\frac{R}{k} + \frac{R}{k^2}\\
	\leq&\left(\frac{1}{k} - \frac{1}{k^2}\right)R\\
	\leq&\frac{1}{k+1} R.
	\end{align*}
\end{proof}

\section{Proof of Theorem~\ref{thm:Sig_conv_hp}}

Before the proof, we introduce the Orlicz $\psi$-norm that will be used in our proof, and present several of its properties. 
\subsection{Orlicz $\psi$-norm}
\begin{definition}[\cite{li2018note}]
	\label{def:psi}
	Let $x\in\RR^d$ be a random vector, the  Orlicz $\psi$-norm is defined as  
	\begin{equation*}
	\|x\|_{\psi}
	=
	\inf\left\{
	K > 0 :
	\EE \psi(\|x\| / K)
	\le
	1
	\right\}.
	\end{equation*}
	Specifically, we will use $\psi_{\xi}(x) = \exp(x^{\xi}) - 1$ with $\xi\in(0,\infty)$, in which case the corresponding Orlicz norm is
	\begin{equation*}
	\|x\|_{\psi_\xi}
	=
	\inf\left\{
	K > 0 :
	\EE \exp(\|x\| / K)^\xi
	\le
	2
	\right\}.
	\end{equation*}
\end{definition}

\begin{proposition}[\cite{Zhou2019}]
	\label{prop:psi_0.5}
	For any random variables $X$, $Y$ with $\norm{X}_\psi < \infty$ and $\norm{Y}_\psi \leq \infty$, we have the
	following inequalities for Orlicz $\psi_{1/2}$-norm
	\[
	\norm{X + Y}_{\psi_{1/2}} \leq 1.3937 \cdot\left(\norm{X}_{\psi_{1/2}} + \norm{Y}_{\psi_{1/2}}\right).
	\]
\end{proposition}
\begin{proof}
	Note that $\psi_{\xi}(x) = \exp(x^{\xi}) - 1$ is not convex when $\xi \in(0,1)$ and $x$ is around $0$.
	In order to make the function convex, let $\tpsi_\xi(x)$ be
	\begin{equation*}
	\tpsi_\xi(x) = \left\{
	\begin{aligned}
	&\exp(x^\xi) - 1 \qquad\quad x\geq x_c\\
	&\frac{x}{x_c}(\exp(x_c) - 1) \quad x\in[0, x_c)
	\end{aligned}
	\right.
	\end{equation*}
	for some appropriate $x_c > 0$. 
	Here $x_c$ is chosen such that the tangent line of function $\psi_{\xi}$ at $x_c$ passes through the origin, i.e.
	\[
	\xi x_c^{\xi - 1}\exp(x_c^\xi) = \frac{\exp(x_c^\xi) -1}{x_c}.
	\]
	Simplifying it leads to the following equation
	\[
	(1-\xi x_c^\xi)\exp(x_c^\xi) = 1 ,
	\]
	which can be solved numerically. When $\xi = 1 / 2$, we have $x_c \approx 2.5396$. 
	Using numerical calculation, we find that
	\begin{equation}
	\label{eq:psi_rel}
	0\leq \psi_{1/2}(x) - \tpsi_\xi(x) \leq 0.2666.
	\end{equation}
	Using the above equation, we have
	\begin{equation*}
	\EE \tpsi_{1/2} (|X|) \leq 1   \Rightarrow \EE \psi_{1/2}(|X|) \leq 1.2666 \quad\mbox{i.e.}\quad \EE\exp(|X|^{1/2} ) \leq 2.2666.
	\end{equation*}
	Let $K_1$, $K_2$ denote the $\psi_{1/2}$-norms of $X$ and $Y$, then we have
	\begin{equation*}
	\EE \psi_{1/2}(|X/K_1|) \leq 1 \quad\mbox{and}\quad \EE\psi_{1/2}(|Y/K_2|) \leq 1.
	\end{equation*}
	By Eqn.~\eqref{eq:psi_rel}, we have
	\begin{equation*}
	\EE \tpsi_{1/2}(|X/K_1|) \leq 1 \quad\mbox{and}\quad \EE\tpsi_{1/2}(|Y/K_2|) \leq 1.
	\end{equation*}
	
	Next, we will prove that $\tpsi_\xi$-norm satisfies triangle inequality, i.e.
	\begin{equation}
	\label{eq:triangle_psi}
	\norm{X+Y}_{\tpsi_\xi} \leq \norm{X}_{\tpsi_\xi} + \norm{Y}_{\tpsi_\xi}.
	\end{equation}
	
	Let us denote $K_1 = \norm{X}_{\tpsi_\xi}$ and $K_2 = \norm{Y}_{\tpsi_\xi}$. Because $\tpsi_{\xi}$ is monotonically increasing and convex, we have
	\begin{align*}
	\tpsi_\xi\left(
	\left|\frac{X+Y}{K_1 + K_2}\right|
	\right)
	\leq&
	\tpsi_\xi\left(
	\frac{K_1}{K_1 +K_2} \frac{|X|}{|K_1|}
	+
	\frac{K_2}{K_1 + K_2} \frac{|Y|}{|K_2|}
	\right)\\
	\leq& \frac{K_1}{K_1 +K_2}\tpsi_\xi\left(\left|\frac{X}{K_1}\right|\right)
	+ \frac{K_2}{K_1 +K_2}\tpsi_\xi\left(\left|\frac{Y}{K_2}\right|\right)\\
	\leq& 1,
	\end{align*}
	which implies that 
	\begin{equation*}
	\norm{X+Y}_{\tpsi_\xi} \leq K_1 + K_2 = \norm{X}_{\tpsi_\xi} + \norm{Y}_{\tpsi_\xi}.
	\end{equation*}

	By applying triangle inequality from Eqn.~\eqref{eq:triangle_psi} to Orlicz $\tpsi_{1/2}$-norm, we have
	\begin{equation*}
	\EE\tpsi_{1/2}\left(\left|\frac{X+Y}{K_1 + K_2}\right|\right) \le 1.
	\end{equation*}
	Along with Eqn.~\eqref{eq:psi_rel}, we have
	\begin{equation*}
	\EE\psi_{1/2}\left(\left|\frac{X+Y}{K_1 + K_2}\right|\right) \le 1.2666.
	\end{equation*}
	
	By applying Jensen’s inequality to concave function $f ( z ) = z^{\log_{2.2666} 2 }$, for constant
	$C_L = (\log_2 (2.2666))^2 = 1.3937$, we have
	\begin{align*}
	\EE\psi_{1/2}\left(\left|\frac{X+Y}{C_L(K_1 + K_2)}\right|\right) 
	=& \EE\exp \left(\left|\frac{X+Y}{K_1 + K_2}\right|^{1/2}\right)^{\log_{2.2666} 2} -1\\
	\leq&\left(\EE\exp \left(\left|\frac{X+Y}{K_1 + K_2}\right|^{1/2}\right)\right)^{\log_{2.2666} 2} - 1\\
	\leq&1.2666^{\log_{2.2666} 2} - 1\\
	\leq& 1,
	\end{align*}
	which implies that
	\begin{equation*}
	\norm{X + Y}_{\psi_{1/2}} \leq C_L \left(\norm{X}_{\psi_{1/2}} + \norm{Y}_{\psi_{1/2}}\right).
	\end{equation*}
\end{proof}

\begin{theorem}[\cite{li2018note}]
	\label{thm:concentrate}
	Let $\xi\in(0,\infty)$ be given. Assume that $(u_i ,i = 1,...,N)$ is a sequence of $\RR^d$-valued
	martingale differences with respect to filtration $\cF_i$, i.e. $\EE[u_i|\cF_{i_1} ] = 0$, and it satisfies $\norm{u_i}_{\psi_\xi} < \infty$ for each $i = 1,...,N$. Then for an arbitrary $N\geq 1$ and $z > 0$,
	\begin{equation*}
	\PP\left(\max_{n\leq N}\norm{\sum_{i=1}^{n} u_i}\geq z\right)\leq 4\left[3+\left(\frac{3}{\xi}\right)^{\frac{2}{\xi}}\frac{128\sum_{i=1}^{N}\norm{u_i}_{\psi_{\xi}}^2}{z^2}\right] \exp\left\{-\left(\frac{z^2}{64\sum_{i=1}^{N} \norm{u_i}_{\psi_\xi}^2}\right)^{\frac{\xi}{\xi+2}}\right\}.
	\end{equation*}
\end{theorem}

\subsection{Proof of Theorem~\ref{thm:Sig_conv_hp}}

Using properties of Orlicz $\psi$-norm, we can now prove Theorem~\ref{thm:Sig_conv_hp}. 
First, by the update rule of $\Sigma_k$ in \texttt{MiNES}, we have the following property.
\begin{lemma}\label{lem:update_hp}
	Let $\Sigma_{k+1}$ be updated as in \texttt{MiNES}, and the step size is set as $\eta_2^{(k)} = \frac{1}{k}$, we have
	\begin{equation}\label{eq:hp_iter}
	\norm{\Sigma_{k+1}^{-1} - \Pi_{\cS'}(H)}^2  \leq \frac{1}{k(k-1)} \sum_{i=2}^k(i-1)\dotprod{Z_i, \Sigma_k^{-1} - \Pi_{\cS'}(H)}
	+ \frac{1}{k(k-1)} \sum_{i=2}^{k}\frac{i-1}{i}\norm{\TG(\Sigma_i)}^2,
	\end{equation}
	where we denote $Z_k = \TG(\Sigma_k) - \EE\TG(\Sigma_k)$.
\end{lemma}
\begin{proof}
	For convenience, we denote $H^* = \Pi_{\cS'}(H)$. By the update rule of $\Sigma_k$, we have
	\begin{align*}
	\norm{\Sigma_{k+1}^{-1} - H^*}^2 
	=& \norm{\Pi_{\cS'}(\Sigma_k^{-1} + \eta_2^{(k)} \TG(\Sigma_k)) - H^*}^2\\
	\leq&\norm{\Sigma_k^{-1} + \eta_2^{(k)} \TG(\Sigma_k) - H^*}^2 \\
	=& \norm{\Sigma_k^{-1} - H^*}^2 + 2\eta_2^{(k)}\dotprod{\Sigma_k^{-1} - H^*, \TG(\Sigma_k)} + (\eta_2^{(k)})^2\norm{\TG(\Sigma_k)}^2\\
	=& \norm{\Sigma_k^{-1} - H^*}^2 + 2\eta_2^{(k)}\dotprod{\Sigma_k^{-1} - H^*, \EE\TG(\Sigma_k)} \\&+ 2\eta_2^{(k)}\dotprod{\Sigma_k^{-1} - H^*, \TG(\Sigma_k) - \EE\TG(\Sigma_k)}+ (\eta_2^{(k)})^2\norm{\TG(\Sigma_k)}^2\\
	\leq&(1-2\eta_2^{(k)})\norm{\Sigma_k^{-1} - H^*}^2 + 2\eta_2^{(k)}\dotprod{\Sigma_k^{-1} - H^*, \TG(\Sigma_k) - \EE\TG(\Sigma_k)}+ (\eta_2^{(k)})^2\norm{\TG(\Sigma_k)}^2\\
	=&\left(1-\frac{2}{k}\right) \norm{\Sigma_k^{-1} - H^*}^2 +\frac{2}{k}\dotprod{Z_k, \Sigma_k^{-1} - H^*} + \frac{1}{k^2}\norm{\TG(\Sigma_k)}^2.
	\end{align*}
	The first inequality is because the projection is non-expansive and the last equality is because we set $\eta_2^{(k)} = 1/k$ and denote that $Z_k = \TG(\Sigma_k) - \EE\TG(\Sigma_k)$.
	
	Unwinding this recursive inequality till $k = 2$, we get that for any $k \geq 2$, 
	\begin{align*}
	\norm{\Sigma_{k+1}^{-1}-H^*}^2 
	\leq& 
	2\sum_{i=2}^{k} \frac{1}{i}\left(\Pi_{j=i+1}^{k}\left(1-\frac{2}{j}\right)\right)\dotprod{Z_i, \Sigma_k^{-1} - H^*}
	+ 
	\sum_{i=2}^{k}\frac{1}{i^2}\left(\Pi_{j=i+1}^{k}\left(1-\frac{2}{j}\right)\right) \norm{\TG(\Sigma_i)}^2\\
	=& \frac{1}{k(k-1)} \sum_{i=2}^k(i-1)\dotprod{Z_i, \Sigma_k^{-1} - H^*}
	+ \frac{1}{k(k-1)} \sum_{i=2}^{k}\frac{i-1}{i}\norm{\TG(\Sigma_i)}^2
	\end{align*}
	Replacing $H^* = \Pi_{\cS'}(H)$ completes the proof.
\end{proof}

Next, we will bound the value of the right hand of Eqn.~\eqref{eq:hp_iter} by concentration inequalities.

\begin{lemma}\label{lem:diff_psi}
	Let us denote $Z_k = \TG(\Sigma_{k}) - \EE\left[ \TG(\Sigma_{k})\right]$ and $z_k = \dotprod{Z_k, \Sigma_{k}^{-1}-\Pi_{\cS'}(H)}$. Then we have
	\begin{equation*}
	\norm{\dotprod{Z_i, \Sigma_{k}^{-1} - \Pi_{\cS'}(H)}}_{\psi_{1/2}} \leq \frac{2\zeta^2L}{\tau}\cdot \norm{(u^\top u)^2}_{\psi_{1/2}} + \frac{dL(\zeta^2 - \tau^2)}{\tau} \norm{u^\top u}_{\psi_{1/2}}+ 2dL\zeta + 2dL^2,
	\end{equation*}
	where $u\sim N(0, I_d)$.
\end{lemma}
\begin{proof}
	By the definition of $\TG(\Sigma_{k})$, $\BG(\Sigma_{k})$ (Eqn.~\eqref{eq:BG}) and Lemma~\ref{prop:TG}, then we have
	\begin{equation}
	Z_i = \BG(\Sigma_{i}) - H.
	\end{equation}
	We let $A = \Sigma_k^{1/2} H \Sigma_{k}^{1/2}$, $H^* = \Pi_{\cS'}(H)$, $B = \Sigma_{k}^{-1} - H^*$ and $\TB =  \Sigma_k^{-1/2} B \Sigma_{k}^{-1/2}$.
	We can obtain that
	\begin{align*}
	\dotprod{Z_i, \Sigma_{k}^{-1} - H^*} =& \dotprod{\BG(\Sigma_{k}) - H,  B}\\
	=& \frac{u^\top A u}{2}\dotprod{\Sigma_{k}^{-1/2}uu^\top \Sigma_{k}^{-1/2} - \Sigma_{k}^{-1}, B} - \dotprod{H, B}\\
	=&\frac{1}{2}(u^\top A u \cdot u^\top \TB u - u^\top A u \cdot\tr(\TB)) - \dotprod{H, B}.
	\end{align*}
	Next we are going to  bound  $\norm{\dotprod{Z_i, \Sigma_{k}^{-1} - H^*}}_{\psi_{1/2}}$. First, by the properties of Proposition~\ref{prop:psi_0.5}, we have
	\begin{align*}
	\norm{\dotprod{Z_i, \Sigma_{k}^{-1} - H^*}}_{\psi_{1/2}} 
	\leq  &
	1.3937\left(\norm{\frac{1}{2}u^\top A u \cdot u^\top \TB u }_{\psi_{1/2}} 
	+ 
	\norm{\frac{1}{2}u^\top A u \cdot\tr(\TB)) + \dotprod{H, B}}_{\psi_{1/2}}\right) \\
	\leq&0.7\norm{A}_2\norm{\TB}_2\norm{(u^\top u)^2}_{\psi_{1/2}} 
	+ 
	1.3937^2\left(\frac{1}{2}\tr(\TB)\norm{u^\top Au}_{\psi_{1/2}} + |\dotprod{H,B}|\right)\\
	\leq&0.7\norm{A}_2\norm{\TB}_2\norm{(u^\top u)^2}_{\psi_{1/2}} 
	+ \tr(\TB)\norm{A}_2\norm{u^\top u}_{\psi_{1/2}} 
	+2|\dotprod{H,B}|.
	\end{align*}
	The second inequality can be derived as follows. Let $K = \norm{(u^\top u)^2}_{\psi_{1/2}}$, then we have
	\begin{align*}
	\EE\exp\left(\frac{u^\top A u \cdot u^\top \TB u}{\norm{A}_2\norm{\TB}_2K}\right)^{1/2} \leq \EE \exp\left(\frac{(u^\top u)^2}{K}\right)^{1/2}\leq 1,
	\end{align*}
	that is $$\norm{u^\top A u \cdot u^\top \TB u}_{\psi_{1/2}} \leq \norm{A}_2\norm{\TB}_2\norm{(u^\top u)^2}_{\psi_{1/2}}. $$
	By substituting the definitions of $A$, $B$ and $\TB$ into the above equation, we obtain 
	\begin{align*}
	\norm{\dotprod{Z_i, \Sigma_{k}^{-1} - H^*}}_{\psi_{1/2}} 
	\leq &
	\norm{\Sigma_k^{1/2} H \Sigma_{k}^{1/2}}_2\cdot\norm{\Sigma_{k}^{-2} - \Sigma_k^{-1/2} \Pi_{\cS'}(H) \Sigma_{k}^{-1/2}}\cdot\norm{(u^\top u)^2}_{\psi_{1/2}} \\
	&+\tr\left(\Sigma_{k}^{-2}-\Sigma_{k}^{-1}\Pi_{\cS'}(H)\right)\norm{\Sigma_k^{1/2} H \Sigma_{k}^{1/2}}_2\norm{u^\top u}_{\psi_{1/2}} \\
	&+2|\dotprod{H,\Sigma_{k} - \Pi_{\cS'}(H)}|.
	\end{align*}
	Since $\Sigma_{k} \in \cS'$ and $H\preceq L \cdot I$, we have 
	\begin{equation*}
	\norm{\dotprod{Z_i, \Sigma_{k}^{-1} - H^*}}_{\psi_{1/2}} \leq \frac{2\zeta^2L}{\tau}\cdot \norm{(u^\top u)^2}_{\psi_{1/2}} + \frac{dL(\zeta^2 - \tau^2)}{\tau} \norm{u^\top u}_{\psi_{1/2}}+ 2dL\zeta + 2dL^2.
	\end{equation*}
\end{proof}

In the next lemma, we will bound some quantities related to $\norm{\cdot}_{\psi_{1/2}}$.
\begin{lemma}\label{lem:psi}
	Let $u\sim N(0, I_d)$, then we have the following properties
	\begin{equation*}
	\norm{(u^\top u)^2}_{\psi_{1/2}} \leq 16d^2,\qquad \norm{u^\top u}_{\psi_{1/2}} \leq 4d^2.
	\end{equation*}
\end{lemma}
\begin{proof}
	Let $x$ be a random variable, by Definition~\ref{def:psi}, then $\norm{\cdot}_{\psi_{1}}$ satisfies  that
	\begin{equation*}
	\|x\|_{\psi_1}
	=
	\inf\left\{
	K > 0 :
	\EE \exp(|x| / K)
	\le
	2
	\right\}.
	\end{equation*}
	First, it is easy to check that $\norm{(u^\top u)^2}_{\psi_{1/2}} = \norm{(u^\top u)}_{\psi_{1}}^2$. Let $K = 4d$, then we have
	\begin{align*}
	\EE\exp(u^\top  u/ K) 
	=&\frac{1}{(2\pi)^{d/2}}\int_{-\infty}^\infty \exp\left(\sum_{i}^{d} (u^{(i)})^2/K\right)\cdot \exp \left(-\frac{\sum_{i}^{d} (u^{(i)})^2}{2}\right) du^{(1)}\dots du^{(d)}\\
	=&\frac{1}{(2\pi)^{d/2} \cdot (K/(K-2))^{d/2}}\int_{-\infty}^\infty \exp\left(-\frac{\sum_{i=1}^{d}(u^{(i)})^2}{2\cdot K/(K-2)}\right) du^{(1)}\dots du^{(d)} \cdot (K/(K-2))^{d/2}\\
	=&(K/(K-2))^{d/2} \\
	=&\left(\frac{2d}{2d - 1}\right)^{d/2}.
	\end{align*}
	When $d = 1$, we have $\left(\frac{2d}{2d - 1}\right)^{d/2} = \sqrt{2} < 2$. Then we will show that $\left(\frac{2d}{2d - 1}\right)^{d/2}$ decreases with $d$ increasing. Note that the monotonicity of  $\left(\frac{2d}{2d - 1}\right)^{d/2}$ is the same to $\ln\left(\left(\frac{2d}{2d - 1}\right)^{d/2}\right)$, then we have
	\begin{align*}
	\frac{\partial \ln\left(\left(\frac{2d}{2d - 1}\right)^{d/2}\right)}{\partial d} =& 
	\frac{\partial \left(\frac{d}{2} \ln \frac{2d}{2d - 1}\right)}{\partial d}
	\\
	=&\frac{1}{2}\ln \left(1 + \frac{1}{2d - 1}\right) - \frac{1}{2}\cdot\frac{1}{2d - 1} \\
	\leq& 0.
	\end{align*}
	Thus, we obtain that $\left(\frac{2d}{2d - 1}\right)^{d/2}$ is a decreasing function, which implies that $\left(\frac{2d}{2d - 1}\right)^{d/2} \leq 2$ for all $d\geq 1$. Therefore, we get the result that 
	\[
	\norm{(u^\top u)}_{\psi_{1}} \leq 4d.
	\] 
	By the relation  $\norm{(u^\top u)^2}_{\psi_{1/2}} = \norm{(u^\top u)}_{\psi_{1}}^2$, we obtain 
	\begin{equation}
	\norm{(u^\top u)^2}_{\psi_{1/2}} \leq 16d^2.
	\end{equation}
	
	Next, for the term $\norm{u^\top u}_{\psi_{1/2}} $, we will use the definition of $\norm{\cdot}_{\psi_{1/2}}$, we have
	\begin{align*}
	\EE\exp\left(u^\top u/K\right)^{1/2} 
	\leq& \EE \exp\left(\sum_{i=1}^{d}|u^{(i)}|/\sqrt{K}\right)\\
	=&\frac{1}{(2\pi)^{d/2}}\int_{-\infty}^\infty \exp\left(\frac{\sum_{i=1}^{d}|u^{(i)}|}{\sqrt{K}}\right)
	\cdot
	\exp\left(-\frac{\sum_{i}^{d} (u^{(i)})^2}{2}\right)
	du^{(1)}\dots du^{(d)}\\
	=&2^d\cdot \frac{1}{(2\pi)^{d/2}}\int_{0}^\infty \exp\left(\frac{\sum_{i=1}^{d}|u^{(i)}|}{\sqrt{K}}\right)
	\cdot
	\exp\left(-\frac{\sum_{i}^{d} (u^{(i)})^2}{2}\right)
	du^{(1)}\dots du^{(d)}\\
	=&2^d\cdot \frac{1}{(2\pi)^{d/2}}\int_{0}^\infty\exp\left(-\frac{\sum_{i=1}^{d}\left(u^{(i)} - 1/\sqrt{K}\right)^2}{2}\right)du^{(1)}\dots du^{(d)} \cdot\exp\left(\frac{d}{2K}\right)\\
	=&2^d\cdot \frac{1}{(2\pi)^{d/2}}\int_{-\infty}^{1/\sqrt{K}}\exp\left(-\frac{\sum_{i=1}^{d}\left(u^{(i)}\right)^2}{2}\right)du^{(1)}\dots du^{(d)} \cdot\exp\left(\frac{d}{2K}\right)
	\end{align*}
	Let us denote $$\Phi(1/\sqrt{K})= \frac{1}{\sqrt{2\pi}}\int_{-\infty}^{1/\sqrt{K}} \exp\left(-\frac{x^2}{2}\right)\;dx, $$
	then we have
	\begin{equation*}
	\EE\exp\left(u^\top u/K\right)^{1/2} \leq 2^d \left(\Phi(1/\sqrt{K})\right)^d \cdot \exp\left(\frac{d}{2K}\right).
	\end{equation*}
	Let us denote that $K = 4d^2$, then we can check that it holds for all $d\geq 1$
	\[
	2^d \left(\Phi\left(\frac{1}{2d}\right)\right)^d \cdot \exp\left(\frac{1}{8d}\right) \leq 2,
	\]	
	which implies that
	\begin{equation*}
	\norm{u^\top u}_{\psi_{1/2}} \leq 4d^2.
	\end{equation*}
\end{proof}

It is well-known that if $u\sim N (0, I_d)$ is a standard Gaussian random variable, then $\|u\|^2$ follows a $\chi^2$ distribution with $d$ degree of freedom. 
The following inequality due to \citet{Laurent2000Adaptive} gives a bound on the tail bound of $\chi^2$. 
\begin{lemma}[$\chi^2$ tail bound~\cite{Laurent2000Adaptive}]
	\label{lem:chi}
	Let $q_1, \dots, q_n$ be independent $\chi^2$ random variables, each with one degree of freedom. For any vector $\gamma = (\gamma_1, \dots , \gamma_n) \in \RR_+^n$ with non-negative
	entries, and any $t > 0$,
	\[
	\PP\left[\sum_{i=1}^{n}\gamma_iq_i \geq \norm{\gamma}_1 + 2 \sqrt{\norm{\gamma}_2^2 t} + 2\norm{\gamma}_\infty t\right] \leq \exp(-t),
	\]
	where $\norm{\gamma} = \sum_{i=1}^{n} |\gamma_i|$.
\end{lemma} 

Now we begin to prove Theorem~\ref{thm:Sig_conv_hp}.
\begin{proof}[Proof of Theorem~\ref{thm:Sig_conv_hp}]
	By Lemma~\ref{lem:update_hp}, we have
	\begin{equation*}
	\norm{\Sigma_{k+1}^{-1} - \Pi_{\cS'}(H)}^2  \leq \frac{1}{k(k-1)} \sum_{i=2}^k(i-1)\dotprod{Z_i, \Sigma_k^{-1} - \Pi_{\cS'}(H)}
	+ \frac{1}{k(k-1)} \sum_{i=2}^{k}\frac{i-1}{i}\norm{\TG(\Sigma_i)}^2.
	\end{equation*}
	Furthermore, by Lemma~\ref{lem:diff_psi} and~\ref{lem:psi}, we have
	\begin{equation*}
	\norm{\dotprod{Z_i, \Sigma_{k}^{-1} - H^*}}_{\psi_{1/2}} \leq \left( \frac{32d^2\zeta^2L}{\tau} + \frac{4d^3L(\zeta^2 - \tau^2)}{\tau} + 2dL\zeta + 2dL^2\right).
	\end{equation*}
	Then, by Theorem~\ref{thm:concentrate}, with probability at least $1-\delta$, it holds that
	\begin{equation*}
	\bigg|\sum_{i=2}^k(i-1)\dotprod{Z_i, \Sigma_k^{-1} - \Pi_{\cS'}(H)}\bigg| \leq 8(k - 1) \cdot \left( \frac{32d^2\zeta^2L}{\tau} + \frac{4d^3L(\zeta^2 - \tau^2)}{\tau} + 2dL\zeta + 2dL^2\right)\log^{5/4}  \left(\frac{12 + 2592 k}{\delta}\right).
	\end{equation*}
	
	Moreover, by Lemma~\ref{lem:quad_TG} with $b = 1$, we have
	\begin{align*}
	\norm{\TG(\Sigma_{i})}^2 
	\leq& \frac{1}{2}\left(u^\top\Sigma_i^{1/2}H\Sigma_i^{1/2}u\right)^2
	\cdot \norm{\left(\Sigma_{i}^{-1/2}uu^\top\Sigma_{i}^{-1/2} - \Sigma_{i}^{-1}\right)}^2 
	+ 2\norm{H}^2\\
	\leq& \frac{L^2}{2\tau^2} \norm{u}^4\cdot 2\left(\norm{\Sigma_{i}^{-1/2}uu^\top\Sigma_{i}^{-1/2}}^2 + \norm{\Sigma_{i}^{-1}}^2\right)+ 2\norm{H}^2\\
	\leq&\frac{L^2\zeta^2}{\tau^2}\left(\norm{u}^8 + \norm{u}^4\right) + 2L^2.
	\end{align*}
	By the definition of chi-squared distribution, we know that $\norm{u}^2$ is distributed according to the chi-square distribution with $d$ degrees of freedom, and it denoted as $\norm{u}^2 \sim \chi_d^2$. 
	By the properties of chi-squared distribution described in Lemma~\ref{lem:chi}, it holds that with probability at least $1- \delta$
	\[
	\norm{u}^2 \leq d + 2\sqrt{d \log(1/\delta)} + 2\log(1/\delta)\leq 2d + 3\log(1/\delta)
	\]
	Thus, we can obtain that it holds with probability at least $1-\delta$ that
	\begin{equation*}
	\norm{\TG(\Sigma_{i})}^2 \leq \frac{L^2\zeta^2}{\tau} \left((2d + 3\log(1/\delta))^4 + (2d + 3\log(1/\delta))^2\right) + 2L^2.
	\end{equation*}
	Furthermore, by the union bound, we have
	\begin{align*}
	\sum_{i=2}^{k}\frac{i-1}{i}\norm{\TG(\Sigma_i)}^2 \leq (k-1) \cdot \left(\frac{L^2\zeta^2}{\tau} \left((2d + 3\log( k/\delta))^4 + (2d + 3\log (k/\delta))^2\right) + 2L^2\right).
	\end{align*}
	
	Combining above results, with probability at least $1-\delta$, it holds that
	\begin{align*}
	\norm{\Sigma_{k+1}^{-1} - \Pi_{\cS'}(H)}^2  
	\leq &
	\frac{1}{k(k-1)} 
	\bigg( 8(k-1) \cdot \left( \frac{32d^2\zeta^2L}{\tau} + \frac{4d^3L(\zeta^2 - \tau^2)}{\tau} + 2dL\zeta + 2dL^2\right) \log^{5/4} \left(\frac{12 + 2592 k}{\delta}\right)\\
	&+ 
	(k-1) \cdot \left(\frac{L^2\zeta^2}{\tau} \left((2d + 3\log( k/\delta))^4 + (2d + 3\log (k/\delta))^2\right) + 2L^2\right)
	\bigg)\\
	=&\frac{1}{k}\Biggl(16\cdot \left( \frac{16d^2\zeta^2L}{\tau} + \frac{2d^3L(\zeta^2 - \tau^2)}{\tau} + dL\zeta + dL^2\right) \log^{5/4} \left(\frac{12 + 2592 k}{\delta}\right)\\
	&+ \left(\frac{L^2\zeta^2}{\tau} \left(\left(2d + 3\log (k/\delta)\right)^4 + (2d + 3\log (k/\delta))^2\right) + 2L^2\right)
	\Biggr).
	\end{align*}
\end{proof}
\section{Proof of Lemma~\ref{lem:appr}}

\begin{proof}[Proof of Lemma~\ref{lem:appr}]
	First, when $\zeta\geq L$ and $\tau \leq \sigma$, then by Proposition~\ref{prop:project}, we can obtain that $\Pi_{\cS'}(H)$ equals $H$. Then by Theorem~\ref{thm:Sig_conv_hp}, it holds with probability $1-2\delta$ that
	\begin{equation*}
	\norm{\Sigma_{k}^{-1} - H}_2 \leq \frac{\sigma}{4} .
	\end{equation*}
	By the definition of spectral norm, and have
	\begin{equation*}
	-\frac{\sigma}{4}\norm{x}^2 \overset{(a)}{\leq} x^\top \left(\Sigma_{k}^{-1} - H\right) x \overset{(b)}{\leq} \frac{\sigma}{4}\norm{x}^2, \quad\mbox{for all } x\in\RR^d.
	\end{equation*}
	Next, we first consider the case $\overset{(a)}{\leq}$, we have
	\begin{align*}
	&-\sigma/4\norm{x}^2 \leq x^\top \left(\Sigma_{k}^{-1} - H\right) x \\
	\Rightarrow&x^\top H x - \sigma/4\norm{x}^2 \leq x^\top \Sigma_k^{-1}x\\
	\Rightarrow&x^\top H x - \frac{\sigma/4}{\sigma} x^\top H x \leq x^\top \Sigma_k^{-1}x\\
	\Rightarrow&H\preceq \left(1+\frac{1}{3}\right)\Sigma_k^{-1}.
	\end{align*}
	Then, we consider the case $\overset{(b)}{\leq}$, we have
	\begin{align*}
	&x^\top \left(\Sigma_{k}^{-1} - H\right) x \leq \sigma/4\norm{x}^2\\
	\Rightarrow& x^\top \Sigma_k^{-1}x \leq x^\top H x + \frac{\sigma/4}{L}x^\top H x \\
	\Rightarrow&\left(1 - \frac{\sigma/4}{L+\sigma/4}\right)\Sigma_{k}^{-1} \preceq H.
	\end{align*}
\end{proof}

\section{Proof of Theorem~\ref{thm:mu}}

The proof of Theorem~\ref{thm:mu} is almost the same as that of \texttt{ZOHA} \citep{ye2018hessian}. For completeness, we still provide the proof here.
First, we will give the two important properties of $\tg(\mu)$ in the following lemma.
\begin{lemma}\label{lem:tg}
	If the function $f(\cdot)$ is quadratic, then expectation of  $\tg(\mu)$ is 
	\begin{equation*}
	\EE_u\left[\ti{g}(\mu_k)\right] = \Sigma_{k}\nabla f(\mu_k).
	\end{equation*}
	The variance of $\tg(\mu_k)$  is 
	\begin{equation*}
	\EE_u\left[\norm{\tg(\mu_k)}_{\Sigma_{k}^{-1}}\right] = (d+2)\cdot\norm{ \nabla f(\mu_k)}_{\Sigma_{k}}^2,
	\end{equation*}
	where $\norm{x}_A^2 = x^\top Ax$ with $A$ being positive semi-definite.
\end{lemma}
\begin{proof}
	By the definition of $\tg(\mu)$ with $b=1$ and the property that function $f(\cdot)$ can be presented as Eqn.~\eqref{eq:quad}, we have
	\begin{align*}
	\EE_u\left[\ti{g}(\mu_k)\right] =& \EE_u\left[\frac{f(\mu_k+ \alpha\Sigma_k^{1/2} u) -f(\mu_k - \alpha\Sigma_k^{1/2}u)}{2\alpha} \Sigma_k^{1/2} u\right]\\
	=&\EE_u\left[\dotprod{\nabla f(\mu_k), \Sigma_{k}^{1/2} u}\Sigma_{k}^{1/2} u\right]\\
	=&\Sigma_{k}\nabla f(\mu_k).
	\end{align*}
	Now we will bound the variance of $\tg$ as follows.
	\begin{align*}
	\EE_u\left[\tg(\mu_k)\Sigma_{k}^{-1}\tg(\mu_k)\right] 
	= &
	\EE_u \left[\dotprod{\nabla f(\mu_k), \Sigma_{k}^{1/2} u}^2(\Sigma_{k}^{1/2} u)^\top \Sigma_{k}^{-1} (\Sigma_{k}^{1/2} u)\right]
	\\
	=&\EE_u\left[\left( u^\top \Sigma_{k}^{1/2} \nabla f(\mu_k)\nabla^\top f(\mu_k) \Sigma_{k}^{1/2}  u\right)\cdot \norm{u}^2\right]\\
	=&d\norm{ \nabla f(\mu_k)}_{\Sigma_{k}}^2 + 2 \norm{ \nabla f(\mu_k)}_{\Sigma_{k}}^2\\
	=&(d+2)\cdot\norm{ \nabla f(\mu_k)}_{\Sigma_{k}}^2.
	\end{align*}
	where the third equality follows  the moments of products of quadratic forms in normal variable (Theorem~ 5.1 of \cite{magnus1978moments}). 
\end{proof}

With the properties of $\tg(\mu_k)$ at hand, we will prove Theorem~\ref{thm:mu}.
\begin{proof}[Proof of Theorem~\ref{thm:mu}]
	First, by Lemma~\ref{lem:appr} and definition of $\rho_k$, it holds with probability at least $1-\delta$ that 
	\begin{equation}\label{eq:prec_cond}
	\rho_k \Sigma_{k}^{-1} \preceq H \preceq (2 - \rho_k) \Sigma_{k}^{-1}
	\end{equation}
	with $\rho_k = 1 - \max\left(1/3, \frac{\sigma/4}{L + \sigma/4}\right)$. Because of $L\geq \sigma$, we can obtain that $\rho_k = \frac{2}{3}$. 
	Conditioned on Eqn.~\eqref{eq:prec_cond} holding, 
	we have
	\begin{align*}
	\EE_u\left[f(\mu_{k+1})\right] 
	=& 
	\EE_u\left[
	f(\mu_k) -\eta_1 \dotprod{\nabla f(\mu_k), \tg(\mu_k)} +\frac{\eta_1^2}{2}\tg(\mu_k)^\top H \tg(\mu_k)
	\right]\\
	=&	f(\mu_k) - \eta_1 \norm{\nabla f(\mu_k)}_{\Sigma_{k}}^2 
	+\frac{\eta_1^2}{2}\EE_u\left[\tg(\mu_k)^\top H \tg(\mu_k)\right]\\
	\leq&f(\mu_k) - \eta_1 \norm{\nabla f(\mu_k)}_{\Sigma_{k}}^2
	+\eta_1^2\EE_u\left[\tg(\mu_k)^\top \Sigma_{k}^{-1} \tg(\mu_k)\right]\\
	=&f(\mu_k) - \eta_1 \norm{\nabla f(\mu_k)}_{\Sigma_{k}}^2 + (d+2)\eta_1^2 \norm{\nabla f(\mu_k)}_{\Sigma_{k}}^2.
	\end{align*}
	where the second and last equalities are because of Lemma~\ref{lem:tg} and the first inequality follows from that $H\preceq 2\Sigma_{k}^{-1}$ which is implied in Eqn.~\eqref{eq:prec_cond}.
	
	By setting $\eta_1 = \frac{1}{2(d+2)}$, then we have
	\begin{align*}
	\EE_u\left[f(\mu_{k+1}) - f(\mu_*)\right]  
	\leq& f(\mu_k) - f(\mu_*) - \frac{1}{4(d+2)}\norm{\nabla f(\mu_k)}_{\Sigma_{k}}^2\\
	\leq&f(\mu_k) - f(\mu_*) - \frac{\rho_k}{4(d+2)}\norm{\nabla f(\mu_k)}_{H^{-1}}^2
	\end{align*}
	where the last inequality is because of Eqn.~\eqref{eq:prec_cond}.
	
	We expand $f(\mu)$ by Taylor's expansion at $\mu_*$, we have
	\begin{equation*}
	f(\mu) = f(\mu_*) + \frac{(\mu - \mu_*)^\top H (\mu-\mu_*)}{2},
	\end{equation*}
	and
	\begin{equation*}
	\nabla f(\mu) = H(\mu - \mu_*).
	\end{equation*}
	Therefore, we have
	\begin{align*}
	\EE_u\left[f(\mu_{k+1}) - f(\mu_*)\right] 
	\leq& 
	f(\mu_k) - f(\mu_*) - \frac{\rho_k}{4(d+2)}\norm{\nabla f(\mu_k)}_{H^{-1}}^2\\
	=&f(\mu_k) - f(\mu_*) - \frac{\rho_k}{4(d+2)}(\mu - \mu_*)^\top H (\mu-\mu_*)\\
	=&f(\mu_k) - f(\mu_*) - \frac{\rho_k}{2(d+2)}\left(f(\mu_k) - f(\mu_*)\right)\\
	=&\left(1 - \frac{\rho_k}{2(d+2)}\right)\cdot \left(f(\mu_k) - f(\mu_*) \right).
	\end{align*}
	With $\rho_k = \frac{2}{3}$, we complete the proof.
\end{proof}
\end{document}